\documentclass{amsart}
\usepackage[all]{xy}
\usepackage{enumerate}
\usepackage{mathrsfs}
\usepackage{amssymb}
\usepackage{graphicx}
\usepackage{tikz-cd}

\theoremstyle{plain}
\newtheorem{theorem}{Theorem}[section]
\newtheorem{lemma}[theorem]{Lemma}
\newtheorem{corollary}[theorem]{Corollary}
\newtheorem{proposition}[theorem]{Proposition}
\newtheorem{conjecture}[theorem]{Conjecture}
\theoremstyle{definition}

\newtheorem{example}[theorem]{Example}

\theoremstyle{remark}
\newtheorem{remark}[theorem]{Remark}

\newcommand\Aut{\operatorname{Aut}}
\newcommand\GL{\operatorname{GL}}

\newcommand\Hom{\operatorname{Hom}}
\newcommand\End{\operatorname{End}}

\newcommand\id{\operatorname{id}}
\newcommand\pr{\operatorname{pr}}

\newcommand\gr{\mathbf{gr}}

\newcommand\op{\mathrm{op}}
\newcommand\ab{\mathrm{ab}}
\newcommand\sgn{\mathrm{sgn}}

\newcommand\N{\mathbb{N}}
\newcommand\K{\Bbbk}

\newcommand\R{\mathbb{R}}

\newcommand\A{\mathbf{A}}

\newcommand\catC{\mathcal{C}}

\newcommand\catP{\mathcal{P}}

\newcommand\Lie{\mathcal{L}ie}

\newcommand\gpS{\mathfrak{S}}

\newcommand\im{\operatorname{im}}

\newcommand\Ob{\operatorname{Ob}}

\newcommand\calC{\mathcal{C}}

\newcommand\calO{\mathcal{O}}

\newcommand\KMod{\K\textbf{-}\mathrm{Mod}}

\newcommand{\adj}[4]{
\begin{tikzcd}[ampersand replacement=\&]
#1\arrow[r, shift left=1ex, "#3"{name=top}] \& #2\arrow[l, shift left=.5ex, "#4"{name=bottom}]
\arrow[phantom, from=top, to=bottom, , "\scriptscriptstyle\boldsymbol{\top}" rotate=180]
\end{tikzcd}
}

\title[Modules over the category of Jacobi diagrams in handlebodies]{Modules over the category of \\
Jacobi diagrams in handlebodies}
\author{Mai Katada}
\address{Graduate School of Mathematical Sciences, University of Tokyo, Tokyo 153-8914, Japan}
\email{mkatada@ms.u-tokyo.ac.jp}
\date{September 14, 2025}



\begin{document}

\begin{abstract}
The linear category $\A$ of Jacobi diagrams in handlebodies was introduced by Habiro and Massuyeau.
We study the category of modules over the category $\A$.
We generalize the adjunction given by Powell to an adjunction between the category of $\A$-modules and the category of modules over the linear PROP for Casimir Lie algebras by using the category $\A^{L}$ of extended Jacobi diagrams in handlebodies.
Then we study subquotient $\A$-modules of $\A(0,-)$.
\end{abstract}
\maketitle
\setcounter{tocdepth}{1}
\tableofcontents

\newcommand\catLie{\mathcal{C}at_{\mathcal{L}ie}}
\newcommand\catLieMod{\catLie\textbf{-}\mathrm{Mod}}
\newcommand\kgrop{\K\mathbf{gr}^{\operatorname{op}}}
\newcommand\kgropMod{\kgrop\textbf{-}\mathrm{Mod}}
\newcommand\Ass{\mathcal{A}ss^{\mathrm{u}}}
\newcommand\catAss{\mathcal{C}at_{\mathcal{A}ss^{\mathrm{u}}}}

\newcommand\catAssC{\mathcal{C}at_{\mathcal{A}ss^{\mathrm{u,C}}}}
\newcommand\catLieC{\mathcal{C}at_{\mathcal{L}ie^{\mathrm{C}}}}
\newcommand\catOC{\mathcal{C}at_{\calO^{\mathrm{C}}}}

\newcommand\catLieCMod{\catLieC\textbf{-}\mathrm{Mod}}
\newcommand\AMod{\A\textbf{-}\mathrm{Mod}}
\newcommand\ub{\mathrm{uB}}

\newcommand\centre[1]{\begin{array}{c} #1 \end{array}}
\newcommand\centreinput[1]{\centre{\input{#1}}}

\section{Introduction}
Finitely generated free groups and the automorphism groups of them are fundamental and important objects in mathematics, especially in algebraic topology.
Functors from the category $\gr$ of finitely generated free groups (or the opposite $\gr^{\op}$ of it) to the category of abelian groups often arise, and there is a substantial body of literature on the functor category on $\gr$ or $\gr^{\op}$ \cite{Djament--Vespa, Hartl--Pirashvili--Vespa, Powell--Vespa, Powellanalytic, Powellouter, PowellPassiMalcev, Kim--Vespa, Arone}. 
We will focus on the functors to the category of vector spaces over a field $\K$ of characteristic $0$.

The notion of polynomial functors was introduced by Eilenberg and MacLane \cite{Eilenberg--MacLane} and generalized by Hartl, Pirashvili and Vespa \cite{Hartl--Pirashvili--Vespa}.
The notion of outer functors, introduced by Powell and Vespa \cite{Powell--Vespa}, also plays an essential role in the functor category on $\gr$ and $\gr^{\op}$. 
The full subcategory of polynomial functors (or analytic functors, which are colimits of polynomial functors) has been studied.
In particular, Powell \cite{Powellanalytic} proved that the full subcategory of analytic functors on $\gr^{\op}$ is equivalent to the category of modules over the $\K$-linear PROP $\catLie$ for Lie algebras.
In some sense, modules over $\catLie$ are easier to handle or understand than analytic functors on $\gr^{\op}$.

Habiro and Massuyeau \cite{Habiro--Massuyeau} introduced the category $\A$ of Jacobi diagrams in handlebodies in order to extend the Kontsevich invariant to a functor from the category of bottom tangles in handlebodies. 
The category $\A$ is characterized as a $\K$-linear PROP freely generated by a Casimir Hopf algebra, and has a grading defined by the number of copies of the Casimir $2$-tensor. 
Since the category $\gr^{\op}$ is characterized as a PROP freely generated by a cocommutative Hopf algebra, the $\K$-linearization $\kgrop$ of the category $\gr^{\op}$ is identified with the degree $0$ part of $\A$. 
The $\K$-linear PROP $\catLieC$ for Casimir Lie algebras was introduced by Hinich and Vaintrob \cite{Hinich--Vaintrob} in their study of the relations between Casimir Lie algebras and the algebra of chord diagrams.
One of the motivations behind this project was to clarify the relations between the category of $\A$-modules and the category of $\catLieC$-modules.

The aim of this paper is to study the category $\AMod$ of $\A$-modules.
We generalize the adjunction given by Powell to an adjunction between $\catLieC$-modules and $\A$-modules. Here, we use an $(\A,\catLieC)$-bimodule induced by the hom-spaces of the category $\A^{L}$ of extended Jacobi diagrams in handlebodies, which was introduced in \cite{Katada2}.
Then we study the subquotient $\A$-modules of $\A(0,-)$. By using the indecomposable decomposition of the $\kgrop$-module $\A_d(0,-)$ obtained in \cite{Katada1, Katada2}, we prove that some subquotient $\A$-modules of $\A(0,-)$ are indecomposable.

\subsection{An adjunction between the categories $\catLieCMod$ and $\AMod$}
Let $\KMod$ denote the category of $\K$-vector spaces and $\K$-linear maps for a field $\K$ of characteristic $0$.
For a $\K$-linear category $\catC$, which is a category enriched over $\KMod$, let $\catC\textbf{-}\mathrm{Mod}$ denote the $\K$-linear category of $\K$-linear functors from $\catC$ to $\KMod$.

Powell \cite[Theorem 9.19]{Powellanalytic} constructed a $(\kgrop, \catLie)$-bimodule ${}_{\Delta}\catAss$, which yields
the hom-tensor adjunction
\begin{gather*}
 {}_{\Delta}\catAss\otimes_{\catLie} - : \adj{\catLieMod}{\kgropMod}{}{} : \Hom_{\kgropMod}({}_{\Delta}\catAss,-).
\end{gather*}
Moreover, he proved that the adjunction restricts to an equivalence of categories between $\catLieMod$ and a full subcategory $\kgropMod^{\omega}$ of $\kgropMod$ whose objects correspond to \emph{analytic functors} on $\gr^{\op}$.

We will give an alternative interpretation of the adjunction given by Powell by using the $\K$-linear symmetric monoidal category $\A^{L}$ of extended Jacobi diagrams. 
The set of objects of the category $\A^{L}$ is the free monoid generated by $H$ and $L$, and $\A^{L}$ includes $\A$ (resp. $\catLieC$) as a full subcategory whose objects are generated by $H$ (resp. $L$).
By restricting to the degree $0$ part, the category $\A^{L}_0$ includes $\kgrop\cong \A_0$ and $\catLie\cong (\catLieC)_0$ as full subcategories.
(See Section \ref{sectionAL} for details.)

\begin{proposition}[see Proposition \ref{adjunctionpowellAL0}]
    We have an isomorphism of $(\kgrop, \catLie)$-bimodules
    \begin{gather*}
         {}_{\Delta}\catAss \cong \A^{L}_0(L^{\otimes -},H^{\otimes -}).
    \end{gather*}
\end{proposition}

Moreover, by generalizing the above interpretation, we obtain an adjunction between $\catLieCMod$ and $\AMod$.

\begin{theorem}[see Theorem \ref{adjunctionA}]\label{adjunctionAintro}
We have an $(\A,\catLieC)$-bimodule $\A^{L}(L^{\otimes -},H^{\otimes -})$, which induces an adjunction
\begin{gather*}
    \A^{L}(L^{\otimes -},H^{\otimes -})\otimes_{\catLieC} - : \adj{\catLieCMod}{\AMod}{}{} : \Hom_{\AMod}(\A^{L}(L^{\otimes -},H^{\otimes -}),-).
\end{gather*}
\end{theorem} 

We define a full subcategory $\AMod^{\omega}$ of $\AMod$ in such a way that an object of $\AMod^{\omega}$ is an $\A$-module $M$ which satisfies $U(M)\in \kgropMod^{\omega}$, where $U:\AMod\to \kgropMod$ denotes the forgetful functor.

\begin{theorem}[see Theorem \ref{adjunctionAomega}]\label{adjunctionAomegaintro}
 We have $\A^{L}(L^{\otimes m},H^{\otimes -})\in \Ob(\AMod^{\omega})$ for each $m\ge 0$,
 and the adjunction in Theorem \ref{adjunctionAintro} restricts to an adjunction 
    \begin{gather*}
    \A^{L}(L^{\otimes -},H^{\otimes -})\otimes_{\catLieC} - : \adj{\catLieCMod}{\AMod^{\omega}}{}{} : \Hom_{\AMod^{\omega}}(\A^{L}(L^{\otimes -},H^{\otimes -}),-).
    \end{gather*}    
\end{theorem}

\begin{remark}
Recently, Minkyu Kim has independently obtained an adjunction, equivalent to Theorem \ref{adjunctionAintro}, by using a general method involving left ideals.
Moreover, he has proved that the restriction of the adjunction, equivalent to that in Theorem \ref{adjunctionAomegaintro}, induces an equivalence of categories.
(See Remark \ref{remarkKim}.)
\end{remark}

For a $\K$-linear PROP $\catP$, the Day convolution of modules over $\catP$ induces a $\K$-linear symmetric monoidal structure on the category of $\catP$-modules.

\begin{proposition}[see Propositions \ref{symmetricmonoidalAL0} and \ref{symmetricmonoidalAL}]
The left adjoint functors
\begin{gather*}
    \A^{L}_0(L^{\otimes-},H^{\otimes -})\otimes_{\catLie}-: \catLieMod\to\kgropMod
\end{gather*}
and 
\begin{gather*}
    \A^{L}(L^{\otimes-},H^{\otimes -})\otimes_{\catLieC}-: \catLieCMod\to\AMod
\end{gather*}
are symmetric monoidal with respect to the symmetric monoidal structures induced by the Day convolution.    
\end{proposition}

We have a graded algebra $A=U(\A(0,-))$ in the symmetric monoidal category $\kgropMod$ with the Day convolution tensor product.
In a similar way, we have a graded algebra $C=U(\catLieC(0,-))$ in the symmetric monoidal category $\catLieMod$ with the Day convolution tensor product, where $U:\catLieCMod\to \catLieMod$ is the forgetful functor.
The graded algebra $A$ corresponds to the graded algebra $C$ via the category equivalence between $\kgropMod^{\omega}$ and $\catLieMod$.

\begin{theorem}[see Theorem \ref{presentationofC} and Corollary \ref{quadraticA}]
    The graded algebra $C$ in $\catLieMod$ is quadratic, and the graded algebra $A$ in $\kgropMod$ is quadratic.
\end{theorem}

\subsection{Subquotient $\A$-modules of $\A(0,-)$}
In \cite{Katada1, Katada2}, the author studied the $\kgrop$-module $A_d=\A_d(0,-)$. One of the main results of \cite{Katada1, Katada2} is an indecomposable decomposition of $A_d$.
More precisely, for $d\ge 2$, we have an indecomposable decomposition of $\kgrop$-modules
$$A_d=A_d P\oplus A_d Q,$$
where $A_d P$ and $A_d Q$ are the $\kgrop$-submodules of $A_d$ that are generated by the symmetric element $P_d$ and the anti-symmetric element $Q_d$, respectively.
Moreover, $A_d P$ is a simple $\kgrop$-module which is isomorphic to $S^{2d}\circ\mathfrak{a}^{\#}$.
Here, $S^{\lambda}$ denotes the \emph{Schur functor} corresponding to a partition $\lambda$, and $\mathfrak{a}^{\#}=\Hom(-^{\ab},\K)$ denotes the dual of the abelianization functor. (See Section \ref{sectionAd} for details.)

For $d\ge 0$, the degree $\ge d$ part forms the $\A$-submodule $\A_{\ge d}(0,-)$ of $\A(0,-)$.
For $d\ge 0,\; d'\ge d+2$, the quotient $\A$-module $\A_{\ge d}(0,-)/\A_{\ge d'}(0,-)$ does not factor through $\kgrop$.
By using the $\kgrop$-module structure of $A_d$, we obtain the following.

\begin{theorem}[see Theorems \ref{indecomposable} and \ref{indecomposableA0}]\label{indecomposableA0intro}
For $d\ge 0,\; d'\ge d+2$, the $\A$-module $\A_{\ge d}(0,-)/\A_{\ge d'}(0,-)$ is indecomposable. Moreover, the $\A$-module $\A(0,-)$ is indecomposable.
\end{theorem}

Let $\A Q$ denote the $\A$-submodule of $\A(0,-)$ that is generated by $Q_2$, which contains all $A_d Q$ for $d\ge 2$.
For $d\ge 2$, the degree $\ge d$ part forms the $\A$-submodule $\A Q_{\ge d}$ of $\A Q$.
For the $\A$-module $\A Q$, we also have the indecomposability.

\begin{theorem}[see Proposition \ref{indecomposableQ} and Theorem \ref{indecomposableA0}]\label{indecomposableAQintro}
For $d\ge 2,\; d'\ge d+1$, the $\A$-module $\A Q_{\ge d}/\A Q_{\ge d'}$ is indecomposable. Moreover, the $\A$-module $\A Q$ is indecomposable.   
\end{theorem}

For each $m\ge 0$, we have the $\A$-module $\A^{L}(L^{\otimes m}, H^{\otimes -})$.
For $m=0$, we have $\A^{L}(L^{\otimes 0},H^{\otimes -})\cong \A(0,-)$ as $\A$-modules.
For $m=1$, we obtain a direct sum decomposition of $A^{L}_d=\A^{L}_d(L,H^{\otimes -})$ as $\kgrop$-modules, which is conjectured to be an indecomposable decomposition.

\begin{proposition}[see Proposition \ref{decompositionALd}]
    Let $d\ge 1$. We have a direct sum decomposition of $\kgrop$-modules
    \begin{gather*}
        A^{L}_d=A^{L}_d P\oplus A^{L}_d Q,
    \end{gather*}
    where $A^{L}_d P$ and $A^{L}_d Q$ are the $\kgrop$-submodules of $A^{L}_d$ that are generated by the symmetric element $P^{L}_d$ and by the two anti-symmetric elements $Q'_d$ and $Q''_d$, respectively.
    Moreover, $A^{L}_d P$ is a simple $\kgrop$-module which is isomorphic to $S^{2d+1}\circ \mathfrak{a}^{\#}$.
\end{proposition}

We make the following conjecture, which is an analogue of the case of $\A(0,-)$.

\begin{conjecture}[see Conjectures \ref{conjectureindecomposableAL} and \ref{conjectureindecomposableALQ}]
    For $d\ge 2$, the $\kgrop$-module $A^{L}_d Q$ is indecomposable.
    Moreover, the $\A$-module $\A^{L}(L,H^{\otimes -})$ is indecomposable.
\end{conjecture}

\subsection{Outline}
We organize the rest of the paper as follows.
From Section \ref{sectionPowell} to Section \ref{sectionsymmonstr}, we study the adjunction between the categories $\catLieCMod$ and $\AMod$.
More precisely, in Section \ref{sectionPowell}, we recall some known facts to state the result of Powell, and in Section \ref{sectionAL}, we recall the definition of the category $\A^{L}$ and study the hom-space of it, and in Section \ref{sectionadjunction}, we prove Theorems \ref{adjunctionAintro} and \ref{adjunctionAomegaintro}, and then in Section \ref{sectionsymmonstr}, we study the adjunctions with respect to the symmetric monoidal structure induced by the Day convolution.
In Section \ref{sectionAmod}, we study the subquotient $\A$-modules of $\A(0,-)$ and prove Theorems \ref{indecomposableA0intro} and \ref{indecomposableAQintro}.
Lastly, in Section \ref{sectionhandlebodygp}, we provide a perspective on modules over the handlebody groups.

\subsection*{Acknowledgements}
The author would like to thank Dr. Minkyu Kim for discussions on the adjunction between $\AMod$ and $\catLieCMod$, and Professor Kazuo Habiro for suggestions regarding the algebra $A$ in $\AMod$, which inspired the writing of Section \ref{sectionsymmonstr}.
She would like to thank Professor Geoffrey Powell for pointing out errors in the initial proof of the indecomposability of the $\A$-module $\A(0,-)$, and for his valuable comments throughout drafts of the present paper.
She would also like to thank Professor Christine Vespa for helpful comments.
This work was supported by JSPS KAKENHI Grant Number 24K16916.

\section{Functors on the opposite category of the category of free groups}\label{sectionPowell}
Here we recall some known facts about the opposite $\gr^{\op}$ of the category $\gr$ of free groups, and describe the adjunction between the functor category on $\gr^{\op}$ and the category of modules over the $\K$-linear PROP $\catLie$ for Lie algebras, which is given by Powell \cite{Powellanalytic}.

\subsection{The category $\gr^{\op}$}

For $n\ge 0$, let $F_n$ be the group freely generated by $\{x_1,\cdots,x_n\}$.
Let $\gr$ denote the category of finitely generated free groups and group homomorphisms, and $\gr^{\op}$ the opposite category of $\gr$.

It is well known that the category $\gr^{\op}$ is a PROP (i.e., a symmetric monoidal category with non-negative integers as objects) freely generated by a cocommutative Hopf algebra $(H=1,\mu,\eta,\Delta,\varepsilon,S)$.
See \cite{Pirashvili} for details, and \cite{Habiro} for a combinatorial description.
Here, the morphism $\mu:2\to 1$ is the group homomorphism 
$$F_1\to F_2,\quad x\mapsto x_1x_2,$$ 
the morphism $\eta:0\to 1$ is the group homomorphism 
$$F_1\to F_0,\quad x\mapsto 1,$$ 
the morphism $\Delta:1\to 2$ is the group homomorphism 
$$F_2\to F_1,\quad x_1\mapsto x,\quad x_2\mapsto x,$$ 
the morphism $\varepsilon:1\to 0$ is the group homomorphism 
$$F_0\to F_1,\quad 1\mapsto 1\in F_1,$$ 
and the morphism $S:1\to 1$ is the group homomorphism 
$$F_1\to F_1,\quad x\mapsto x^{-1}.$$

Define $\mu^{[i]}:i\to 1$ inductively by 
$\mu^{[0]}=\eta,\; \mu^{[1]}=\id_1,\; \mu^{[i]}=\mu (\mu^{[i-1]}\otimes \id_1)$ for $i\ge 2$.
Let 
\begin{gather}\label{mumulti}
\mu^{[p_1,\cdots,p_n]}=\mu^{[p_1]}\otimes\cdots\otimes\mu^{[p_n]}: p_1+\cdots +p_n\to n.
\end{gather}
Similarly, we define $\Delta^{[i]}:1\to i$ by $\Delta^{[0]}=\varepsilon, \;\Delta^{[1]}=\id_1,\; \Delta^{[i]}=(\Delta^{[i-1]}\otimes \id_1)\Delta$ for $i\ge 2$, and let
\begin{gather*}
\Delta^{[q_1,\cdots,q_m]}=\Delta^{[q_1]}\otimes \cdots\otimes\Delta^{[q_m]} :m\to q_1+\cdots+q_m.
\end{gather*}
Define a group homomorphism 
\begin{gather}\label{permutation}
    P:\gpS_m\to \gr^{\op}(m,m),\quad \sigma\mapsto P_{\sigma}
\end{gather}
by $P_{(i,i+1)}=\id_{i-1}\otimes P_{1,1}\otimes \id_{m-i-1}$ for $1\le i\le m-1$, where $P_{1,1}\in \gr^{\op}(2,2)$ is the symmetry such that $P_{1,1}(x_1)=x_2,\; P_{1,1}(x_2)=x_1$.
Then we have the following factorization of morphisms of the category $\gr^{\op}$.

\begin{lemma}[{\cite[Lemma 2, Theorem 8]{Habiro}}]\label{lemmamorofkgrop}
     Any element of $\gr^{\op}(m,n)$ can be decomposed into the following form
     $$\mu^{[p_1,\cdots,p_n]}P_{\sigma}(S^{e_1}\otimes\cdots\otimes S^{e_s})\Delta^{[q_1,\cdots,q_m]},$$
     where $s,q_1,\cdots,q_m,p_1,\cdots,p_n\ge 0$, $s=q_1+\cdots+ q_m=p_1+\cdots+p_n$, $\sigma\in \gpS_{s}$, $e_1,\cdots,e_s\in \{0,1\}$.
\end{lemma}

Note that the above factorization is not unique since we have
\begin{gather*}
    \mu(\id_1\otimes S)\Delta=\eta\varepsilon.
\end{gather*}

\subsection{The $\K$-linear PROPs for unital associative algebras and Lie algebras}

Let $\catAss$ denote the $\K$-linear PROP associated to the operad $\Ass$ encoding unital associative algebras.
That is, the category $\catAss$ is the $\K$-linear PROP freely generated by a unital associative algebra $(A=1,\mu_A,\eta_A)$, where the hom-space is 
\begin{gather*}
    \catAss(m,n)=\bigoplus_{f: \{1,\dots,m\}\to \{1,\dots,n\}} \bigotimes_{i=1}^{n} \Ass(|f^{-1}(i)|).
\end{gather*}
Define $\mu_A^{[p_1,\cdots,p_n]}$ and a group homomorphism $P:\gpS_m\to \catAss(m,m),\; \sigma\mapsto P_{\sigma}$ in a way similar to \eqref{mumulti} and \eqref{permutation}, respectively.
Then any element of $\catAss(m,n)$ has the following factorization.

\begin{lemma}\label{factorcatAss}
  Any element of $\catAss(m,n)$ can be written uniquely as a linear combination of morphisms
  $$\mu_A^{[p_1,\cdots,p_n]}P_{\sigma}$$
  with $p_1+\cdots+ p_n=m$, $\sigma\in \gpS_{m}$.
  That is, the set 
  \begin{gather*}
      \{\mu_A^{[p_1,\cdots,p_n]}P_{\sigma}\mid p_1+\cdots+ p_n=m,\;\sigma\in \gpS_{m}\}
  \end{gather*}
  is a basis for $\catAss(m,n)$.
\end{lemma}

Let $\catLie$ denote the $\K$-linear PROP associated to the Lie operad $\Lie$. That is, $\catLie$ is the $\K$-linear PROP freely generated by a Lie algebra $(L=1,[,]:L^{\otimes 2}\to L)$, where the hom-space is
\begin{gather*}
      \catLie(m,n)=\bigoplus_{f: \{1,\dots,m\}\twoheadrightarrow \{1,\dots,n\}} \bigotimes_{i=1}^{n} \Lie(|f^{-1}(i)|).
\end{gather*}
Define a group homomorphism $P:\gpS_m\to \catLie(m,m),\; \sigma\mapsto P_{\sigma}$ in a way similar to \eqref{permutation}.
By using the representation of an element of the operad $\Lie$ as a \emph{rooted trivalent tree}, the hom-space $\catLie(m,n)$ is explicitly described as follows. Here, a rooted trivalent tree means a trivalent tree with a root and labeled leaves, where each trivalent vertex has a specified cyclic order of adjacent edges.

\begin{lemma}\label{decompositionofcatLie}
The hom-space $\catLie(m,n)$ is the $\K$-vector space spanned by the set
\begin{gather*}
    \left\{(T_1\otimes \cdots\otimes T_n)P_{\sigma}\,\middle\vert\,
    \substack{T_i\in \Lie(m_i): \text{a rooted trivalent tree } (1\le i\le n), \\
    m_1+\cdots+m_n=m,\;\sigma\in \gpS_m}\right\}
\end{gather*}
modulo the equivariance relation
\begin{gather}\label{equivariancerel}
    (T_1 \rho_1\otimes \cdots\otimes T_n \rho_n)P_{\sigma}=(T_1\otimes \cdots\otimes T_n)P_{(\rho_1\otimes \cdots\otimes \rho_n) \sigma}
\end{gather}
for $\rho_i\in \gpS_{m_i}\; (1\le i\le n)$,
and the AS and IHX relations described below:
\begin{gather*}
    \scalebox{0.8}{$\centre{%% Creator: Inkscape 1.4.2 (ebf0e940, 2025-05-08), www.inkscape.org
%% Accompanies image file 'ASIHXsolid.pdf' (pdf, eps, ps)
\begingroup%
  \makeatletter%
  \providecommand\color[2][]{%
    \errmessage{(Inkscape) Color is used for the text in Inkscape, but the package 'color.sty' is not loaded}%
    \renewcommand\color[2][]{}%
  }%
  \providecommand\transparent[1]{%
    \errmessage{(Inkscape) Transparency is used (non-zero) for the text in Inkscape, but the package 'transparent.sty' is not loaded}%
    \renewcommand\transparent[1]{}%
  }%
  \providecommand\rotatebox[2]{#2}%
  \newcommand*\fsize{\dimexpr\f@size pt\relax}%
  \newcommand*\lineheight[1]{\fontsize{\fsize}{#1\fsize}\selectfont}%
  \ifx\svgwidth\undefined%
    \setlength{\unitlength}{252.56692481bp}%
    \ifx\svgscale\undefined%
      \relax%
    \else%
      \setlength{\unitlength}{\unitlength * \real{\svgscale}}%
    \fi%
  \else%
    \setlength{\unitlength}{\svgwidth}%
  \fi%
  \global\let\svgwidth\undefined%
  \global\let\svgscale\undefined%
  \makeatother%
  \begin{picture}(1,0.26936027)%
    \lineheight{1}%
    \setlength\tabcolsep{0pt}%
    \put(0,0){\includegraphics[width=\unitlength,page=1]{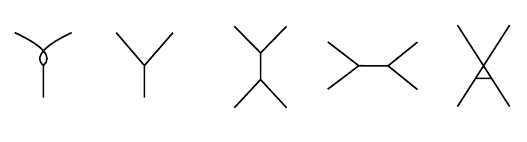}}%
    \put(0.55246772,0.13075562){\color[rgb]{0,0,0}\makebox(0,0)[lt]{\lineheight{1.25}\smash{\begin{tabular}[t]{l}$=$\end{tabular}}}}%
    \put(0.79650751,0.13430338){\color[rgb]{0,0,0}\makebox(0,0)[lt]{\lineheight{1.25}\smash{\begin{tabular}[t]{l}$\;-$\end{tabular}}}}%
    \put(0.12998457,0.13141701){\color[rgb]{0,0,0}\makebox(0,0)[lt]{\lineheight{1.25}\smash{\begin{tabular}[t]{l}$= \;-$\end{tabular}}}}%
    \put(0.33310404,0.08298741){\color[rgb]{0,0,0}\makebox(0,0)[lt]{\lineheight{1.25}\smash{\begin{tabular}[t]{l}$,$\\\end{tabular}}}}%
  \end{picture}%
\endgroup%
}$}.
\end{gather*}
\end{lemma}

The canonical morphism $\Lie\to \Ass$ of operads, which maps $[,]\in \catLie(2,1)$ to $\mu_A-\mu_A P_{(12)}\in \catAss(2,1)$, induces a morphism $\iota: \catLie\to \catAss$ of $\K$-linear PROPs.
Then for each $m\ge 0$, we have a right $\catLie$-module
\begin{gather*}
 \catAss(\iota(-),m):\catLie^{\op}\to \KMod.
\end{gather*}
Here, by a right $\catLie$-module, we mean a $\K$-linear functor from the opposite $\catLie^{\op}$ of the category $\catLie$ to the category $\KMod$.

\subsection{Functors on the category $\gr^{\op}$}

For a $\K$-linear PROP $\calC$, by a left $\calC$-module, we mean a $\K$-linear functor from $\calC$ to $\KMod$.
In what follows, all modules are left modules unless otherwise mentioned.
Let $\calC\textbf{-}\mathrm{Mod}$ denote the category of $\calC$-modules.

Let $\K\gr^{\op}$ denote the $\K$-linearization of the category $\gr^{\op}$.
Note that $\K$-linearization induces an isomorphism
\begin{gather*}
    \operatorname{Funct}(\gr^{\op},\KMod)\xrightarrow{\cong} \kgropMod,
\end{gather*}
where $\operatorname{Funct}(\gr^{\op},\KMod)$ denotes the functor category on $\gr^{\op}$.

Powell \cite{Powellanalytic} defined a $\kgrop$-module structure on the family of hom-spaces $\{\catAss(l,m)\}_m$ of the $\K$-linear category $\catAss$ for each $l\ge 0$, which is denoted by
\begin{gather*}
 {}_{\Delta}\catAss(l,-):\kgrop\to \KMod.
\end{gather*}
In what follows, we will describe the image of each of the generating morphisms of $\kgrop$ which appeared in Lemma \ref{lemmamorofkgrop}.
For $p_1,\cdots,p_n\ge 0, \; s=p_1+\cdots+p_n$, $\sigma\in \gpS_s$, we have
\begin{gather*}
{}_{\Delta}\catAss(l,\mu^{[p_1,\cdots,p_n]})=\mu_A^{[p_1,\cdots,p_n]}\circ -: \catAss(l,s)\to \catAss(l,n),\\
{}_{\Delta}\catAss(l,P_{\sigma})=P_{\sigma}\circ -: \catAss(l,s)\to \catAss(l,s).
\end{gather*}
For $e_1,\cdots,e_s\in\{0,1\}$, the morphism
\begin{gather*}
   {}_{\Delta}\catAss(l,S^{e_1}\otimes\cdots\otimes S^{e_s}): {}_{\Delta}\catAss(l,s)\to {}_{\Delta}\catAss(l,s)
\end{gather*}
is defined for $\mu_A^{[p_1,\cdots,p_s]}$ with $p_1,\cdots,p_s\ge 0,\; p_1+\cdots+p_s=l$ by
\begin{gather*}
   {}_{\Delta}\catAss(l,S^{e_1}\otimes\cdots\otimes S^{e_s})(\mu_A^{[p_1,\cdots,p_s]})=(S^{e_1}\cdot \mu_A^{[p_1]})\otimes\cdots\otimes (S^{e_s}\cdot \mu_A^{[p_s]}),
\end{gather*}
where 
\begin{gather*}
    S\cdot \mu_A^{[p]}=(-1)^p \mu_A^{[p]} P_{\text{reflection}},
\end{gather*}
where ``reflection" means the permutation 
$\begin{pmatrix}
    1&2&\cdots&p\\
    p&p-1&\cdots&1
 \end{pmatrix}\in \gpS_p$.
In a similar way, for $q_1,\cdots,q_m\ge 0,\; s=q_1+\cdots+q_m$, the morphism
\begin{gather*}
    {}_{\Delta}\catAss(l,\Delta^{[q_1,\cdots,q_m]})= {}_{\Delta}\catAss(l,m)\to {}_{\Delta}\catAss(l,s)
\end{gather*}
is defined for $\mu_A^{[p_1,\cdots,p_m]}$ with $p_1,\cdots,p_m\ge 0,\; p_1+\cdots+p_m=l$ by
\begin{gather*}
    {}_{\Delta}\catAss(l,\Delta^{[q_1,\cdots,q_m]})(\mu_A^{[p_1,\cdots,p_m]})= (\Delta^{[q_1]}\cdot\mu_A^{[p_1]}) \otimes\cdots\otimes(\Delta^{[q_m]}\cdot\mu_A^{[p_m]}),
\end{gather*}
where
\begin{gather*}
    \Delta^{[0]}\cdot\mu_A^{[p]}=\varepsilon\cdot\mu_A^{[p]}=
\begin{cases}
0 & (p\ge 1),\\
\id_0 & (p=0)
\end{cases}
\end{gather*}
and for $q\ge 1$,
\newcommand{\Sh}{\operatorname{Sh}}
\begin{gather*}
\Delta^{[q]}\cdot\mu_A^{[p]}
=\sum_{\substack{i_1+\cdots+i_q=p,\\i_1,\cdots,i_q\ge 0}}\sum_{\sigma\in \Sh(i_1,\cdots,i_q)} \mu_A^{[i_1,\cdots,i_q]}P_{\sigma^{-1}},
\end{gather*}
where, $\Sh(i_1,\cdots,i_q)\subset \gpS_p$ denotes the set of $(i_1,\cdots,i_q)$-shuffles.
For example, we have 
\begin{gather*}
\begin{split}
    {}_{\Delta}\catAss(0,\Delta)(\eta_A)&=\eta_A\otimes \eta_A,\\
    {}_{\Delta}\catAss(1,\Delta)(\id_1)&=\id_1\otimes \eta_A+\eta_A\otimes \id_1,\\
    {}_{\Delta}\catAss(2,\Delta)(\mu_A)&=\mu_A\otimes \eta_A+\id_2+P_{(12)}+\eta_A\otimes\mu_A.
\end{split}
\end{gather*}

Now we have a $(\kgrop,\catLie)$-bimodule ${}_{\Delta}\catAss={}_{\Delta}\catAss(-,-)$.
Here, by a $(\kgrop,\catLie)$-bimodule, we mean a $\K$-linear functor $\catLie^{\op}\otimes\kgrop\to \KMod$.
Powell \cite{Powellanalytic} proved that the usual hom-tensor adjunction yields an equivalence of categories.

\begin{proposition}[{\cite[Theorem 9.19]{Powellanalytic}}]\label{adjunctionPowell}
We have the hom-tensor adjunction
\begin{gather*}
 {}_{\Delta}\catAss\otimes_{\catLie} - : \adj{\catLieMod}{\kgropMod}{}{} : \Hom_{\kgropMod}({}_{\Delta}\catAss,-).
\end{gather*}
Moreover, the adjunction restricts to an equivalence of categories
\begin{gather*}
    \catLieMod \simeq \kgropMod^{\omega}.
\end{gather*}
\end{proposition}
Here, $\kgropMod^{\omega}$ denotes the full subcategory of $\kgropMod$ whose objects correspond to \emph{analytic functors} on $\gr^{\op}$, where a functor is analytic if it is the colimit of its \emph{polynomial subfunctors}.
See \cite{Powellanalytic} for details.

\section{The category $\A^{L}$ of extended Jacobi diagrams in handlebodies}\label{sectionAL}
In this section, we recall the definition and some important properties of the $\K$-linear category $\A$ of Jacobi diagrams in handlebodies which was introduced by Habiro and Massuyeau in \cite{Habiro--Massuyeau}.
We also recall the $\K$-linear PROP $\catLieC$ for Casimir Lie algebras introduced by Hinich and Vaintrob in \cite{Hinich--Vaintrob}.
Then we recall the $\K$-linear category $\A^{L}$ of extended Jacobi diagrams in handlebodies introduced in \cite{Katada2} and study its hom-spaces.

\subsection{The category $\A$ of Jacobi diagrams in handlebodies}

Habiro and Massuyeau introduced the category $\A$ of Jacobi diagrams in handlebodies in \cite{Habiro--Massuyeau} in order to extend the Kontsevich invariant for bottom tangles to a functor.
The objects of $\A$ are non-negative integers.
The hom-space $\A(m,n)$ is the $\K$-vector space spanned by ``$(m,n)$-Jacobi diagrams in handlebodies", which are explained below.

For $n\ge 0$, let 
$X_n=\scalebox{0.7}{$\centre{%% Creator: Inkscape 1.4.2 (ebf0e940, 2025-05-08), www.inkscape.org
%% Accompanies image file 'Xn.pdf' (pdf, eps, ps)
\begingroup%
  \makeatletter%
  \providecommand\color[2][]{%
    \errmessage{(Inkscape) Color is used for the text in Inkscape, but the package 'color.sty' is not loaded}%
    \renewcommand\color[2][]{}%
  }%
  \providecommand\transparent[1]{%
    \errmessage{(Inkscape) Transparency is used (non-zero) for the text in Inkscape, but the package 'transparent.sty' is not loaded}%
    \renewcommand\transparent[1]{}%
  }%
  \providecommand\rotatebox[2]{#2}%
  \newcommand*\fsize{\dimexpr\f@size pt\relax}%
  \newcommand*\lineheight[1]{\fontsize{\fsize}{#1\fsize}\selectfont}%
  \ifx\svgwidth\undefined%
    \setlength{\unitlength}{70.86614173bp}%
    \ifx\svgscale\undefined%
      \relax%
    \else%
      \setlength{\unitlength}{\unitlength * \real{\svgscale}}%
    \fi%
  \else%
    \setlength{\unitlength}{\svgwidth}%
  \fi%
  \global\let\svgwidth\undefined%
  \global\let\svgscale\undefined%
  \makeatother%
  \begin{picture}(1,0.24)%
    \lineheight{1}%
    \setlength\tabcolsep{0pt}%
    \put(0,0){\includegraphics[width=\unitlength,page=1]{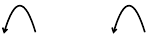}}%
    \put(0.56079311,0.08352336){\makebox(0,0)[lt]{\lineheight{1.45000005}\smash{\begin{tabular}[t]{l}$\cdots$\end{tabular}}}}%
    \put(0,0){\includegraphics[width=\unitlength,page=2]{Xn.pdf}}%
  \end{picture}%
\endgroup%
}$}$
be the oriented $1$-manifold consisting of $n$ component oriented arcs.
A \emph{Jacobi diagram} $D$ on $X_n$ is a uni-trivalent graph such that each trivalent vertex is oriented (i.e., each trivalent vertex has a fixed cyclic order of the three edges around it), univalent vertices are attached to $X_n$ (i.e., the set of univalent vertices is embedded into the interior of $X_n$), and each connected component has at least one univalent vertex.
Two Jacobi diagrams $D$ and $D'$ on $X_n$ are identified if there is a homeomorphism between them whose restriction to $X_n$ is isotopic to the identity map and which respects the orientations on trivalent vertices.
The STU relation described below is imposed on the $\K$-vector space spanned by Jacobi diagrams on $X_n$:
\begin{gather*}
    \scalebox{0.7}{$\centre{%% Creator: Inkscape 1.4.2 (ebf0e940, 2025-05-08), www.inkscape.org
%% Accompanies image file 'STU.pdf' (pdf, eps, ps)
\begingroup%
  \makeatletter%
  \providecommand\color[2][]{%
    \errmessage{(Inkscape) Color is used for the text in Inkscape, but the package 'color.sty' is not loaded}%
    \renewcommand\color[2][]{}%
  }%
  \providecommand\transparent[1]{%
    \errmessage{(Inkscape) Transparency is used (non-zero) for the text in Inkscape, but the package 'transparent.sty' is not loaded}%
    \renewcommand\transparent[1]{}%
  }%
  \providecommand\rotatebox[2]{#2}%
  \newcommand*\fsize{\dimexpr\f@size pt\relax}%
  \newcommand*\lineheight[1]{\fontsize{\fsize}{#1\fsize}\selectfont}%
  \ifx\svgwidth\undefined%
    \setlength{\unitlength}{252.56692481bp}%
    \ifx\svgscale\undefined%
      \relax%
    \else%
      \setlength{\unitlength}{\unitlength * \real{\svgscale}}%
    \fi%
  \else%
    \setlength{\unitlength}{\svgwidth}%
  \fi%
  \global\let\svgwidth\undefined%
  \global\let\svgscale\undefined%
  \makeatother%
  \begin{picture}(1,0.26936027)%
    \lineheight{1}%
    \setlength\tabcolsep{0pt}%
    \put(0,0){\includegraphics[width=\unitlength,page=1]{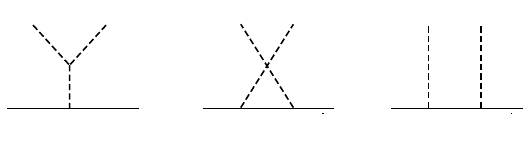}}%
    \put(0.27561785,0.12612326){\color[rgb]{0,0,0}\makebox(0,0)[lt]{\lineheight{1.25}\smash{\begin{tabular}[t]{l}$=$\end{tabular}}}}%
    \put(0.64879981,0.12628834){\color[rgb]{0,0,0}\makebox(0,0)[lt]{\lineheight{1.25}\smash{\begin{tabular}[t]{l}$-$\end{tabular}}}}%
    \put(0,0){\includegraphics[width=\unitlength,page=2]{STU.pdf}}%
  \end{picture}%
\endgroup%
}$}
\end{gather*}
The STU relation implies the AS and IHX relations (see \cite[Theorem 6]{Bar-Natan}).
The \emph{degree} of a Jacobi diagram is defined to be half the number of vertices. Note that the STU, AS and IHX relations respect the degree of Jacobi diagrams.

For $m\ge 0$, let $U_m$ be the handlebody of genus $m$ that is obtained from the cube $I^3$, where $I=[-1,1]\subset \R$, by attaching $m$ handles on the top square $I^2\times \{1\}$ of the cube along the upper line $I\times\{0\}\times\{1\}$.
An $(m,n)$-Jacobi diagram is a Jacobi diagram on $X_n$ mapped into the handlebody $U_m$ in such a way that the endpoints of $X_n$ are arranged on the bottom line $I\times \{0\}\times \{-1\}$ and that the $i$-th arc component of $X_n$ goes from the $2i$-th point to the $(2i-1)$-st point, where we count the endpoints from left to right.
See \cite{Habiro--Massuyeau} for details.
The following is an example of a $(2,3)$-Jacobi diagram:
$$\scalebox{1.0}{$\centre{%% Creator: Inkscape 1.4.2 (ebf0e940, 2025-05-08), www.inkscape.org
%% Accompanies image file '32JD.pdf' (pdf, eps, ps)
\begingroup%
  \makeatletter%
  \providecommand\color[2][]{%
    \errmessage{(Inkscape) Color is used for the text in Inkscape, but the package 'color.sty' is not loaded}%
    \renewcommand\color[2][]{}%
  }%
  \providecommand\transparent[1]{%
    \errmessage{(Inkscape) Transparency is used (non-zero) for the text in Inkscape, but the package 'transparent.sty' is not loaded}%
    \renewcommand\transparent[1]{}%
  }%
  \providecommand\rotatebox[2]{#2}%
  \newcommand*\fsize{\dimexpr\f@size pt\relax}%
  \newcommand*\lineheight[1]{\fontsize{\fsize}{#1\fsize}\selectfont}%
  \ifx\svgwidth\undefined%
    \setlength{\unitlength}{141.73228346bp}%
    \ifx\svgscale\undefined%
      \relax%
    \else%
      \setlength{\unitlength}{\unitlength * \real{\svgscale}}%
    \fi%
  \else%
    \setlength{\unitlength}{\svgwidth}%
  \fi%
  \global\let\svgwidth\undefined%
  \global\let\svgscale\undefined%
  \makeatother%
  \begin{picture}(1,0.9)%
    \lineheight{1}%
    \setlength\tabcolsep{0pt}%
    \put(0,0){\includegraphics[width=\unitlength,page=1]{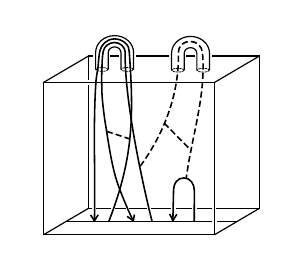}}%
  \end{picture}%
\endgroup%
}$}$$
Two $(m,n)$-Jacobi diagrams are identified if they are homotopic in $U_m$ relative to the endpoints of $X_n$.
Then the hom-space $\A(m,n)$ is defined by the $\K$-vector space spanned by $(m,n)$-Jacobi diagrams modulo the STU relation.
The composition of morphisms in $\A$ is somewhat complicated. To put it simply, for an $(m,n)$-Jacobi diagram $D$ and $(n,p)$-Jacobi diagram $D'$, the composition $D'\circ D$ is obtained by stacking a suitable cabling of $D$ on the top square of $D'$. 

Each hom-space $\A(m,n)$ is graded by the degree of Jacobi diagrams
and we have
\begin{gather*}
    \A(m,n)\cong \bigoplus_{d\ge 0}\A_d(m,n),
\end{gather*}
where $\A_d(m,n)$ denotes the degree $d$ part of $\A(m,n)$.
Moreover, the category $\A$ has a structure of an $\N$-graded $\K$-linear symmetric strict monoidal category.
The tensor product on objects is addition, the monoidal unit is $0$, and the tensor product on morphisms is juxtaposition followed by horizontal rescaling and relabeling of indices.
The category $\A$ has a \emph{Casimir Hopf algebra} $(H=1,\mu,\eta,\Delta,\varepsilon,S,\tilde{c})$, where
\begin{gather*}
\mu=\scalebox{0.8}{$\centre{%% Creator: Inkscape 1.4.2 (ebf0e940, 2025-05-08), www.inkscape.org
%% Accompanies image file 'mu.pdf' (pdf, eps, ps)
\begingroup%
  \makeatletter%
  \providecommand\color[2][]{%
    \errmessage{(Inkscape) Color is used for the text in Inkscape, but the package 'color.sty' is not loaded}%
    \renewcommand\color[2][]{}%
  }%
  \providecommand\transparent[1]{%
    \errmessage{(Inkscape) Transparency is used (non-zero) for the text in Inkscape, but the package 'transparent.sty' is not loaded}%
    \renewcommand\transparent[1]{}%
  }%
  \providecommand\rotatebox[2]{#2}%
  \newcommand*\fsize{\dimexpr\f@size pt\relax}%
  \newcommand*\lineheight[1]{\fontsize{\fsize}{#1\fsize}\selectfont}%
  \ifx\svgwidth\undefined%
    \setlength{\unitlength}{42.51968504bp}%
    \ifx\svgscale\undefined%
      \relax%
    \else%
      \setlength{\unitlength}{\unitlength * \real{\svgscale}}%
    \fi%
  \else%
    \setlength{\unitlength}{\svgwidth}%
  \fi%
  \global\let\svgwidth\undefined%
  \global\let\svgscale\undefined%
  \makeatother%
  \begin{picture}(1,1)%
    \lineheight{1}%
    \setlength\tabcolsep{0pt}%
    \put(0,0){\includegraphics[width=\unitlength,page=1]{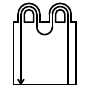}}%
  \end{picture}%
\endgroup%
}$},\; \eta=\scalebox{0.8}{$\centre{%% Creator: Inkscape 1.4.2 (ebf0e940, 2025-05-08), www.inkscape.org
%% Accompanies image file 'eta.pdf' (pdf, eps, ps)
\begingroup%
  \makeatletter%
  \providecommand\color[2][]{%
    \errmessage{(Inkscape) Color is used for the text in Inkscape, but the package 'color.sty' is not loaded}%
    \renewcommand\color[2][]{}%
  }%
  \providecommand\transparent[1]{%
    \errmessage{(Inkscape) Transparency is used (non-zero) for the text in Inkscape, but the package 'transparent.sty' is not loaded}%
    \renewcommand\transparent[1]{}%
  }%
  \providecommand\rotatebox[2]{#2}%
  \newcommand*\fsize{\dimexpr\f@size pt\relax}%
  \newcommand*\lineheight[1]{\fontsize{\fsize}{#1\fsize}\selectfont}%
  \ifx\svgwidth\undefined%
    \setlength{\unitlength}{42.51968504bp}%
    \ifx\svgscale\undefined%
      \relax%
    \else%
      \setlength{\unitlength}{\unitlength * \real{\svgscale}}%
    \fi%
  \else%
    \setlength{\unitlength}{\svgwidth}%
  \fi%
  \global\let\svgwidth\undefined%
  \global\let\svgscale\undefined%
  \makeatother%
  \begin{picture}(1,1)%
    \lineheight{1}%
    \setlength\tabcolsep{0pt}%
    \put(0,0){\includegraphics[width=\unitlength,page=1]{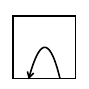}}%
  \end{picture}%
\endgroup%
}$},\; \Delta=\scalebox{0.8}{$\centre{%% Creator: Inkscape 1.4.2 (ebf0e940, 2025-05-08), www.inkscape.org
%% Accompanies image file 'Delta.pdf' (pdf, eps, ps)
\begingroup%
  \makeatletter%
  \providecommand\color[2][]{%
    \errmessage{(Inkscape) Color is used for the text in Inkscape, but the package 'color.sty' is not loaded}%
    \renewcommand\color[2][]{}%
  }%
  \providecommand\transparent[1]{%
    \errmessage{(Inkscape) Transparency is used (non-zero) for the text in Inkscape, but the package 'transparent.sty' is not loaded}%
    \renewcommand\transparent[1]{}%
  }%
  \providecommand\rotatebox[2]{#2}%
  \newcommand*\fsize{\dimexpr\f@size pt\relax}%
  \newcommand*\lineheight[1]{\fontsize{\fsize}{#1\fsize}\selectfont}%
  \ifx\svgwidth\undefined%
    \setlength{\unitlength}{42.51968504bp}%
    \ifx\svgscale\undefined%
      \relax%
    \else%
      \setlength{\unitlength}{\unitlength * \real{\svgscale}}%
    \fi%
  \else%
    \setlength{\unitlength}{\svgwidth}%
  \fi%
  \global\let\svgwidth\undefined%
  \global\let\svgscale\undefined%
  \makeatother%
  \begin{picture}(1,1.13333333)%
    \lineheight{1}%
    \setlength\tabcolsep{0pt}%
    \put(0,0){\includegraphics[width=\unitlength,page=1]{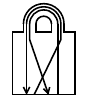}}%
  \end{picture}%
\endgroup%
}$},\;
\varepsilon=\scalebox{0.8}{$\centre{%% Creator: Inkscape 1.4.2 (ebf0e940, 2025-05-08), www.inkscape.org
%% Accompanies image file 'varepsilon.pdf' (pdf, eps, ps)
\begingroup%
  \makeatletter%
  \providecommand\color[2][]{%
    \errmessage{(Inkscape) Color is used for the text in Inkscape, but the package 'color.sty' is not loaded}%
    \renewcommand\color[2][]{}%
  }%
  \providecommand\transparent[1]{%
    \errmessage{(Inkscape) Transparency is used (non-zero) for the text in Inkscape, but the package 'transparent.sty' is not loaded}%
    \renewcommand\transparent[1]{}%
  }%
  \providecommand\rotatebox[2]{#2}%
  \newcommand*\fsize{\dimexpr\f@size pt\relax}%
  \newcommand*\lineheight[1]{\fontsize{\fsize}{#1\fsize}\selectfont}%
  \ifx\svgwidth\undefined%
    \setlength{\unitlength}{42.51968504bp}%
    \ifx\svgscale\undefined%
      \relax%
    \else%
      \setlength{\unitlength}{\unitlength * \real{\svgscale}}%
    \fi%
  \else%
    \setlength{\unitlength}{\svgwidth}%
  \fi%
  \global\let\svgwidth\undefined%
  \global\let\svgscale\undefined%
  \makeatother%
  \begin{picture}(1,1.13333333)%
    \lineheight{1}%
    \setlength\tabcolsep{0pt}%
    \put(0,0){\includegraphics[width=\unitlength,page=1]{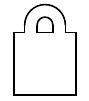}}%
  \end{picture}%
\endgroup%
}$},\\ 
S=\scalebox{0.8}{$\centre{%% Creator: Inkscape 1.4.2 (ebf0e940, 2025-05-08), www.inkscape.org
%% Accompanies image file 'S.pdf' (pdf, eps, ps)
\begingroup%
  \makeatletter%
  \providecommand\color[2][]{%
    \errmessage{(Inkscape) Color is used for the text in Inkscape, but the package 'color.sty' is not loaded}%
    \renewcommand\color[2][]{}%
  }%
  \providecommand\transparent[1]{%
    \errmessage{(Inkscape) Transparency is used (non-zero) for the text in Inkscape, but the package 'transparent.sty' is not loaded}%
    \renewcommand\transparent[1]{}%
  }%
  \providecommand\rotatebox[2]{#2}%
  \newcommand*\fsize{\dimexpr\f@size pt\relax}%
  \newcommand*\lineheight[1]{\fontsize{\fsize}{#1\fsize}\selectfont}%
  \ifx\svgwidth\undefined%
    \setlength{\unitlength}{42.51968504bp}%
    \ifx\svgscale\undefined%
      \relax%
    \else%
      \setlength{\unitlength}{\unitlength * \real{\svgscale}}%
    \fi%
  \else%
    \setlength{\unitlength}{\svgwidth}%
  \fi%
  \global\let\svgwidth\undefined%
  \global\let\svgscale\undefined%
  \makeatother%
  \begin{picture}(1,1.13333333)%
    \lineheight{1}%
    \setlength\tabcolsep{0pt}%
    \put(0,0){\includegraphics[width=\unitlength,page=1]{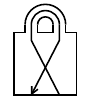}}%
  \end{picture}%
\endgroup%
}$},\;
\tilde{c}=\scalebox{0.8}{$\centre{%% Creator: Inkscape 1.4.2 (ebf0e940, 2025-05-08), www.inkscape.org
%% Accompanies image file 'ctilde.pdf' (pdf, eps, ps)
\begingroup%
  \makeatletter%
  \providecommand\color[2][]{%
    \errmessage{(Inkscape) Color is used for the text in Inkscape, but the package 'color.sty' is not loaded}%
    \renewcommand\color[2][]{}%
  }%
  \providecommand\transparent[1]{%
    \errmessage{(Inkscape) Transparency is used (non-zero) for the text in Inkscape, but the package 'transparent.sty' is not loaded}%
    \renewcommand\transparent[1]{}%
  }%
  \providecommand\rotatebox[2]{#2}%
  \newcommand*\fsize{\dimexpr\f@size pt\relax}%
  \newcommand*\lineheight[1]{\fontsize{\fsize}{#1\fsize}\selectfont}%
  \ifx\svgwidth\undefined%
    \setlength{\unitlength}{42.51968504bp}%
    \ifx\svgscale\undefined%
      \relax%
    \else%
      \setlength{\unitlength}{\unitlength * \real{\svgscale}}%
    \fi%
  \else%
    \setlength{\unitlength}{\svgwidth}%
  \fi%
  \global\let\svgwidth\undefined%
  \global\let\svgscale\undefined%
  \makeatother%
  \begin{picture}(1,1)%
    \lineheight{1}%
    \setlength\tabcolsep{0pt}%
    \put(0,0){\includegraphics[width=\unitlength,page=1]{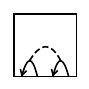}}%
  \end{picture}%
\endgroup%
}$}.
\end{gather*}
Here, a Casimir Hopf algebra $(H,\tilde{c})$ in a $\K$-linear symmetric monoidal category is a cocommutative Hopf algebra $H$ equipped with a Casimir $2$-tensor $\tilde{c}$, which is a morphism $\tilde{c}:I\to H^{\otimes 2}$ satisfying 
\begin{gather*}
    (\Delta\otimes \id_H)\tilde{c}=(\id_H\otimes \eta\otimes \id_H)\tilde{c}+\eta\otimes \tilde{c}, 
\end{gather*}
\begin{gather*}
    P_{H,H}\tilde{c}=\tilde{c}, 
\end{gather*}
where $P_{H,H}$ denotes the symmetry,
and
\begin{gather*}
    (ad\otimes ad)(\id_H\otimes P_{H,H}\otimes \id_H)(\Delta\otimes \tilde{c})=\tilde{c}\varepsilon,
\end{gather*}
where 
\begin{gather*}
    ad=\mu^{[3]}(\id_{H^{\otimes 2}}\otimes S)(\id_H\otimes P_{H,H})(\Delta\otimes \id_H).
\end{gather*}

Then the category $\A$ is characterized by the Casimir Hopf algebra in the following sense.

\begin{lemma}[{\cite[Theorem 5.11]{Habiro--Massuyeau}}]\label{presentationofA}
    The $\K$-linear PROP $\A$ is freely generated by the \emph{Casimir Hopf algebra} $(H=1,\mu,\eta,\Delta,\varepsilon,S,\tilde{c})$.
\end{lemma}

The degree $0$ part $\A_0$ of the category $\A$ forms a subcategory of $\A$.
Since the PROP $\gr^{\op}$ is freely generated by a cocommutative Hopf algebra, there exists a unique symmetric monoidal functor from $\gr^{\op}$ to $\A$ such that the cocommutative Hopf algebra of $\gr^{\op}$ is mapped to the cocommutative Hopf algebra of $\A$. 
This functor induces an isomorphism $\kgrop\xrightarrow{\cong} \A_0$ of $\K$-linear PROPs.
We will identify the morphisms of cocommutative Hopf algebras in $\kgrop$ and those in $\A$ via this functor.

Since the degree of the generating morphisms $\mu,\eta,\Delta,\varepsilon,S$ in $\A$ is $0$, and since the degree of $\tilde{c}$ is $1$, the degree of a morphism of $\A$ is equal to the degree by the number of copies of the Casimir $2$-tensor $\tilde{c}$.
A morphism of $\A_d(m,n)$ can be factorized as follows. 

\begin{lemma}[{\cite[Lemma 5.16]{Habiro--Massuyeau}}]\label{decompositionofA}
    Any element of $\A_d(m,n)$ is a linear combination of morphisms of the form
    \begin{gather*}
        \mu^{[p_1,\cdots,p_n]}P_{\sigma}(S^{e_1}\otimes\cdots\otimes S^{e_s}\otimes \id_{2d})(\Delta^{[q_1,\cdots,q_m]}\otimes \tilde{c}^{\otimes d}),
    \end{gather*}
    where $s,q_1,\cdots, q_m,p_1,\cdots,p_n\ge 0$, 
    $s=p_1+\cdots+p_n-2d=q_1+\cdots+q_m$, $\sigma\in \gpS_{s+2d}, e_1,\cdots,e_s\in \{0,1\}$.
\end{lemma}

\subsection{The $\K$-linear PROP $\catLieC$ for Casimir Lie algebras}
Here, we recall the $\K$-linear PROP for Casimir Lie algebras introduced by Hinich and Vaintrob in \cite{Hinich--Vaintrob}.
A \emph{Casimir Lie algebra} $(L,c)$ in a symmetric monoidal category is a Lie algebra $L$ equipped with a symmetric invariant $2$-tensor, which is called the \emph{Casimir element}, which is a morphism $c:I\to L^{\otimes 2}$ satisfying
\begin{gather}\label{Casimirrelation}
    P_{L,L}c=c, \quad ([,]\otimes \id_L)(\id_L\otimes c)=(\id_L\otimes [,])(c\otimes \id_L).
\end{gather}

Consider the operad $\Lie$ as a cyclic operad.
Let $\tau=(012\cdots n)\in \gpS_{1+n}$.
The symmetric group $\gpS_{1+n}$ is generated by the cyclic permutation $\tau$ and the subgroup $\gpS_{n}\subset \gpS_{1+n}$ that stabilizes $0$.
The right action of $\tau$ on an element of $\Lie$ is given by changing the first input into the output and the output into the last input.
For example, we have 
\begin{gather*}
     \scalebox{0.5}{$\centre{%% Creator: Inkscape 1.4.2 (ebf0e940, 2025-05-08), www.inkscape.org
%% Accompanies image file 'Lie2.pdf' (pdf, eps, ps)
\begingroup%
  \makeatletter%
  \providecommand\color[2][]{%
    \errmessage{(Inkscape) Color is used for the text in Inkscape, but the package 'color.sty' is not loaded}%
    \renewcommand\color[2][]{}%
  }%
  \providecommand\transparent[1]{%
    \errmessage{(Inkscape) Transparency is used (non-zero) for the text in Inkscape, but the package 'transparent.sty' is not loaded}%
    \renewcommand\transparent[1]{}%
  }%
  \providecommand\rotatebox[2]{#2}%
  \newcommand*\fsize{\dimexpr\f@size pt\relax}%
  \newcommand*\lineheight[1]{\fontsize{\fsize}{#1\fsize}\selectfont}%
  \ifx\svgwidth\undefined%
    \setlength{\unitlength}{42.51968504bp}%
    \ifx\svgscale\undefined%
      \relax%
    \else%
      \setlength{\unitlength}{\unitlength * \real{\svgscale}}%
    \fi%
  \else%
    \setlength{\unitlength}{\svgwidth}%
  \fi%
  \global\let\svgwidth\undefined%
  \global\let\svgscale\undefined%
  \makeatother%
  \begin{picture}(1,1)%
    \lineheight{1}%
    \setlength\tabcolsep{0pt}%
    \put(0,0){\includegraphics[width=\unitlength,page=1]{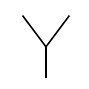}}%
  \end{picture}%
\endgroup%
}$}\cdot \tau=\scalebox{0.5}{$\centre{}$},\quad 
     \scalebox{0.5}{$\centre{%% Creator: Inkscape 1.4.2 (ebf0e940, 2025-05-08), www.inkscape.org
%% Accompanies image file 'Lie31.pdf' (pdf, eps, ps)
\begingroup%
  \makeatletter%
  \providecommand\color[2][]{%
    \errmessage{(Inkscape) Color is used for the text in Inkscape, but the package 'color.sty' is not loaded}%
    \renewcommand\color[2][]{}%
  }%
  \providecommand\transparent[1]{%
    \errmessage{(Inkscape) Transparency is used (non-zero) for the text in Inkscape, but the package 'transparent.sty' is not loaded}%
    \renewcommand\transparent[1]{}%
  }%
  \providecommand\rotatebox[2]{#2}%
  \newcommand*\fsize{\dimexpr\f@size pt\relax}%
  \newcommand*\lineheight[1]{\fontsize{\fsize}{#1\fsize}\selectfont}%
  \ifx\svgwidth\undefined%
    \setlength{\unitlength}{42.51968504bp}%
    \ifx\svgscale\undefined%
      \relax%
    \else%
      \setlength{\unitlength}{\unitlength * \real{\svgscale}}%
    \fi%
  \else%
    \setlength{\unitlength}{\svgwidth}%
  \fi%
  \global\let\svgwidth\undefined%
  \global\let\svgscale\undefined%
  \makeatother%
  \begin{picture}(1,1)%
    \lineheight{1}%
    \setlength\tabcolsep{0pt}%
    \put(0,0){\includegraphics[width=\unitlength,page=1]{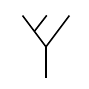}}%
  \end{picture}%
\endgroup%
}$}\cdot \tau =\scalebox{0.5}{$\centre{%% Creator: Inkscape 1.4.2 (ebf0e940, 2025-05-08), www.inkscape.org
%% Accompanies image file 'Lie32.pdf' (pdf, eps, ps)
\begingroup%
  \makeatletter%
  \providecommand\color[2][]{%
    \errmessage{(Inkscape) Color is used for the text in Inkscape, but the package 'color.sty' is not loaded}%
    \renewcommand\color[2][]{}%
  }%
  \providecommand\transparent[1]{%
    \errmessage{(Inkscape) Transparency is used (non-zero) for the text in Inkscape, but the package 'transparent.sty' is not loaded}%
    \renewcommand\transparent[1]{}%
  }%
  \providecommand\rotatebox[2]{#2}%
  \newcommand*\fsize{\dimexpr\f@size pt\relax}%
  \newcommand*\lineheight[1]{\fontsize{\fsize}{#1\fsize}\selectfont}%
  \ifx\svgwidth\undefined%
    \setlength{\unitlength}{42.51968504bp}%
    \ifx\svgscale\undefined%
      \relax%
    \else%
      \setlength{\unitlength}{\unitlength * \real{\svgscale}}%
    \fi%
  \else%
    \setlength{\unitlength}{\svgwidth}%
  \fi%
  \global\let\svgwidth\undefined%
  \global\let\svgscale\undefined%
  \makeatother%
  \begin{picture}(1,1)%
    \lineheight{1}%
    \setlength\tabcolsep{0pt}%
    \put(0,0){\includegraphics[width=\unitlength,page=1]{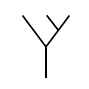}}%
  \end{picture}%
\endgroup%
}$}, \quad \scalebox{0.5}{$\centre{%% Creator: Inkscape 1.4.2 (ebf0e940, 2025-05-08), www.inkscape.org
%% Accompanies image file 'Lie41.pdf' (pdf, eps, ps)
\begingroup%
  \makeatletter%
  \providecommand\color[2][]{%
    \errmessage{(Inkscape) Color is used for the text in Inkscape, but the package 'color.sty' is not loaded}%
    \renewcommand\color[2][]{}%
  }%
  \providecommand\transparent[1]{%
    \errmessage{(Inkscape) Transparency is used (non-zero) for the text in Inkscape, but the package 'transparent.sty' is not loaded}%
    \renewcommand\transparent[1]{}%
  }%
  \providecommand\rotatebox[2]{#2}%
  \newcommand*\fsize{\dimexpr\f@size pt\relax}%
  \newcommand*\lineheight[1]{\fontsize{\fsize}{#1\fsize}\selectfont}%
  \ifx\svgwidth\undefined%
    \setlength{\unitlength}{42.51968504bp}%
    \ifx\svgscale\undefined%
      \relax%
    \else%
      \setlength{\unitlength}{\unitlength * \real{\svgscale}}%
    \fi%
  \else%
    \setlength{\unitlength}{\svgwidth}%
  \fi%
  \global\let\svgwidth\undefined%
  \global\let\svgscale\undefined%
  \makeatother%
  \begin{picture}(1,1)%
    \lineheight{1}%
    \setlength\tabcolsep{0pt}%
    \put(0,0){\includegraphics[width=\unitlength,page=1]{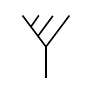}}%
  \end{picture}%
\endgroup%
}$}\cdot \tau =\scalebox{0.5}{$\centre{%% Creator: Inkscape 1.4.2 (ebf0e940, 2025-05-08), www.inkscape.org
%% Accompanies image file 'Lie42.pdf' (pdf, eps, ps)
\begingroup%
  \makeatletter%
  \providecommand\color[2][]{%
    \errmessage{(Inkscape) Color is used for the text in Inkscape, but the package 'color.sty' is not loaded}%
    \renewcommand\color[2][]{}%
  }%
  \providecommand\transparent[1]{%
    \errmessage{(Inkscape) Transparency is used (non-zero) for the text in Inkscape, but the package 'transparent.sty' is not loaded}%
    \renewcommand\transparent[1]{}%
  }%
  \providecommand\rotatebox[2]{#2}%
  \newcommand*\fsize{\dimexpr\f@size pt\relax}%
  \newcommand*\lineheight[1]{\fontsize{\fsize}{#1\fsize}\selectfont}%
  \ifx\svgwidth\undefined%
    \setlength{\unitlength}{42.51968504bp}%
    \ifx\svgscale\undefined%
      \relax%
    \else%
      \setlength{\unitlength}{\unitlength * \real{\svgscale}}%
    \fi%
  \else%
    \setlength{\unitlength}{\svgwidth}%
  \fi%
  \global\let\svgwidth\undefined%
  \global\let\svgscale\undefined%
  \makeatother%
  \begin{picture}(1,1)%
    \lineheight{1}%
    \setlength\tabcolsep{0pt}%
    \put(0,0){\includegraphics[width=\unitlength,page=1]{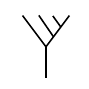}}%
  \end{picture}%
\endgroup%
}$}.
\end{gather*}

Let $\catLieC$ denote the $\K$-linear PROP for Casimir Lie algebras, which is a $\K$-linear PROP generated by the $\K$-linear PROP $\catLie$ and a symmetric element $c\in \catLieC(0,2)$ satisfying 
\begin{gather*}
    (\id_1\otimes f)(c\otimes \id_{n-1})=(f\tau \otimes \id_1)(\id_{n-1}\otimes c),
\end{gather*}
for $f\in \Lie(n)$ (see \cite[Section 3.1.6]{Hinich--Vaintrob} for details).
In other words, $\catLieC$ is the $\K$-linear symmetric monoidal category whose objects are generated by $L$ and whose morphisms are generated by $[,]:L^{\otimes 2}\to L$ and $c: I\to L^{\otimes 2}$ with relations generated by the AS, IHX-relations and the relations \eqref{Casimirrelation}.

The category $\catLieC$ has a grading given by the number of copies of the Casimir element $c$, that is, we have
\begin{gather*}
    \catLieC(m,n)=\bigoplus_{d\ge 0}\catLieC(m,n)_d,
\end{gather*}
where $\catLieC(m,n)_d$ is spanned by morphisms with $d$ copies of $c$.
Then the degree $0$ part $\catLieC(m,n)_0$ forms a subcategory $(\catLieC)_0$, which is isomorphic to $\catLie$.

It is easy to see that any element of $\catLieC(m,n)_d$ can be decomposed into a linear combination of morphisms of the form 
\begin{gather*}
    f(\id_{L^{\otimes m}}\otimes c^{\otimes d}),
\end{gather*}
where $f\in \catLie(m+2d,n)$. This fact can be described by using the \emph{upward Brauer category} $\ub$ \cite{Sam--Snowden, Kupers--Randal-Williams}, which is a $\K$-linear category whose objects are non-negative integers and whose hom-space $\ub(m,n)$ is spanned by partitionings of $\{1^+,\cdots,m^+, 1^-,\cdots,n^-\}$ into $(m+n)/2$ unordered pairs in such a way that the partitioning includes no pairs of two positive elements.
The assignment of the Casimir element $c$ to each unordered pair of two negative elements induces an inclusion 
$\ub\hookrightarrow \catLieC,$
which induces a $\K$-linear full functor
\begin{gather}\label{catLieuB}
    \catLie\otimes_{\mathbb{S}}\ub\twoheadrightarrow \catLieC,
\end{gather}
where $\mathbb{S}$ is the $\K$-linear category whose objects are non-negative integers and whose hom-space is $\mathbb{S}(m,n)=\K\gpS_m$ if $m=n$ and $\mathbb{S}(m,n)=0$ otherwise, and where the structure of a $\K$-linear category on $\catLie\otimes_{\mathbb{S}}\ub$ is the canonical one induced by the structures of $\catLie$ and $\ub$.

\begin{remark}
The hom-space $\catLieC(0,n)$ can be identified with the space of \emph{open Jacobi diagrams} whose univalent vertices are colored by $\{1,\cdots,n\}$.
See \cite[Section 13]{PowellPassiMalcev} for the construction of $\catLieC$ by using open Jacobi diagrams. The surjective morphism \eqref{catLieuB} appears in \cite[Example 13.15]{PowellPassiMalcev}.
It is well known that open Jacobi diagrams with one univalent vertex vanish by the AS and IHX relations, which implies that we have \begin{gather}\label{CatLieC01}
    \catLieC(0,1)=0.
\end{gather}
\end{remark}

By extending the description of the hom-space $\catLie(m,n)$ in Lemma \ref{decompositionofcatLie}, we will explicitly describe the $\K$-vector space $\catLieC(0,n)_d$.
Before stating the following lemma, we will introduce some notation.
For $\sigma\in \gpS_{2d}$ and an $n$-tuple $(T_1,\cdots,T_n)$ of rooted trivalent trees $T_i\in \Lie(m_i)$ with $1\le i\le n, \;m_1+\cdots+m_n=2d$,
which satisfy
\begin{gather}
     \label{star}
     \tag{$\ast$}
     \begin{split}
    &\sigma(1)=m_1+\cdots+m_{i-1}+1,\quad \sigma(3)=m_1+\cdots+m_{i-1}+2,\\ 
    &\sigma(2)= m_1+\cdots+m_{j-1}+1,\quad T_i=t([,]\otimes \id_{m_i-2})
     \end{split}
\end{gather}
for some $1\le i<j\le n$ and for some rooted trivalent tree $t\in \Lie(m_i-1)$,
set 
\begin{gather*}
    \overline{T_i}=t,\quad
    \widetilde{T_j}=([,]\otimes \id_{m_j-1})T_j,\quad
    \widetilde{\sigma}=\sigma'\sigma,
\end{gather*}
where $\sigma'\in \gpS_{2d}$ is the order-preserving permutation except for 
\begin{gather*}
    \sigma'(\sigma(3))=\sigma(2).
\end{gather*}
For $\sigma\in \gpS_{2d}$ and an $n$-tuple $(T_1,\cdots,T_n)$ which satisfy
\begin{gather}
     \label{starstar}
     \tag{$\ast\ast$}
     \begin{split}
    &\sigma(1)=m_1+\cdots+m_{i-1}+1,\quad \sigma(3)=m_1+\cdots+m_{i-1}+2,\\ 
    &\sigma(2)= m_1+\cdots+m_{i-1}+3,\quad T_i=t([,]\otimes \id_{m_i-2})
     \end{split}
\end{gather}
for some $1\le i\le n$, we set 
\begin{gather*}
  \widetilde{(\overline{T_i})}=(\id_1\otimes [,]\otimes \id_{m_i-3})t
\end{gather*}
and $\widetilde{\sigma}$ in the same way as above.
\begin{example}
    For $\sigma=(1256)\in \gpS_{6}$ and a triple $(T_1=\id_1, T_2=\scalebox{0.5}{$\centre{}$}, T_3=\scalebox{0.4}{$\centre{}$})$ of rooted trivalent trees, we have $i=2,j=3$, $\overline{T_2}=\scalebox{0.4}{$\centre{}$}$, $\widetilde{T_3}=\scalebox{0.4}{$\centre{}$}$, $\widetilde{\sigma}=(124356)$.
\end{example}
\begin{example}
     For $\sigma=(1245)\in \gpS_{6}$ and a triple $(T_1=\id_1, T_2=\scalebox{0.5}{$\centre{}$}, T_3=\scalebox{0.4}{$\centre{}$})$ of rooted trivalent trees, we have $i=j=2$, 
     $\widetilde{(\overline{T_2})}=\scalebox{0.4}{$\centre{}$}$, 
     $\widetilde{\sigma}=(12345)$.
\end{example}

\begin{lemma}\label{CatLieC0-}
The degree $d$ part $\catLieC(0,n)_d$ is the $\K$-vector space spanned by the set
\begin{gather}\label{generatorofcatLieC0-}
    \left\{(T_1\otimes \cdots\otimes T_n)P_{\sigma} c^{\otimes d} \,\middle\vert\,
    \substack{T_i\in \Lie(m_i): \text{a rooted trivalent tree } (1\le i\le n), \\
    m_1+\cdots+m_n=2d,\;\sigma\in \gpS_m}\right\}
\end{gather}
modulo the equivariance relation \eqref{equivariancerel}, the AS and IHX relations, the relation between the Casimir element $c$ and the symmetry
\begin{gather}\label{symmetryc}
    (T_1\otimes\cdots\otimes T_n)P_{\sigma \rho}c^{\otimes d}=(T_1\otimes\cdots\otimes T_n)P_{\sigma} c^{\otimes d}
\end{gather}
for $\rho\in \gpS_2 \wr \gpS_n$,
and the relations between the Casimir element $c$ and the Lie bracket
\begin{gather}\label{[,]cij}
    (T_1\otimes\cdots\otimes T_n)P_{\sigma} c^{\otimes d}=-
    (T_1\otimes\cdots \otimes \overline{T_i} \otimes\cdots \otimes \widetilde{T_j}\otimes\cdots \otimes T_n)P_{\widetilde{\sigma}} c^{\otimes d}
\end{gather}
for $(T_1\otimes \cdots\otimes T_n)P_{\sigma} c^{\otimes d}$ satisfying the condition \eqref{star} and 
\begin{gather}\label{[,]ci}
    (T_1\otimes\cdots\otimes T_n)P_{\sigma} c^{\otimes d}=-
    (T_1\otimes\cdots \otimes \widetilde{(\overline{T_i})}\otimes\cdots \otimes T_n)P_{\widetilde{\sigma}} c^{\otimes d}
\end{gather}
for $(T_1\otimes \cdots\otimes T_n)P_{\sigma} c^{\otimes d}$ satisfying the condition \eqref{starstar}.
\end{lemma}

\begin{proof}
The $\K$-vector space $\ub(0,2d)$ has a basis $\{\bar{\sigma}c_d\mid \bar\sigma\in \gpS_{2d}/(\gpS_2\wr\gpS_d)\}$, where $c_d=\{\{1^-,2^-\},\cdots,\{(2d-1)^-,2d^-\}\}$.
Therefore, by Lemma \ref{decompositionofcatLie}, the $\K$-vector space 
$$(\catLie\otimes_{\mathbb{S}}\ub)(0,n)_d=
\catLie(2d,n)\otimes_{\K[\gpS_{2d}]}\ub(0,2d)$$
is spanned by the set 
\begin{gather*}
    \left\{(T_1\otimes \cdots\otimes T_n)P_{\sigma} \otimes c_d \,\middle\vert\,
    \substack{T_i\in \Lie(m_i): \text{a rooted trivalent tree } (1\le i\le n), \\
    m_1+\cdots+m_n=2d,\;\sigma\in \gpS_m}\right\}
\end{gather*}
with relations generated by the equivariance relation \eqref{equivariancerel}, the AS and IHX relations, and the relation 
\begin{gather*}
    (T_1\otimes \cdots\otimes T_n)P_{\sigma \rho}\otimes c_d=(T_1\otimes \cdots\otimes T_n)P_{\sigma} \otimes c_d
\end{gather*}
for $\rho\in \gpS_2\wr\gpS_d$.
It follows from the surjective morphism \eqref{catLieuB} that $\catLieC(0,n)_d$ is spanned by the set \eqref{generatorofcatLieC0-}
with relations generated by 
the equivariance relation \eqref{equivariancerel}, the AS and IHX relations, the relation \eqref{symmetryc}, and the relations between the Casimir element $c$ and the Lie bracket generated by \eqref{Casimirrelation}.
Therefore, it suffices to prove that the relation \eqref{Casimirrelation} reduces to the relations \eqref{[,]cij} and \eqref{[,]ci} by the other relations.

By the relations \eqref{equivariancerel} and \eqref{symmetryc}, and the AS and IHX relations, the relation \eqref{Casimirrelation} reduces to the relation for $(T_1\otimes\cdots\otimes T_n)P_{\sigma} c^{\otimes d}$ satisfying the condition 
\begin{gather*}
\begin{split}
&\sigma(1)=m_1+\cdots+m_{i-1}+1,\quad \sigma(3)=m_1+\cdots+m_{i-1}+2,\\ 
&\sigma(2)=m_1+\cdots+m_{j-1}+1,\quad
T_i=t (([,](\id_1\otimes t'))\otimes \id_{m_i-m'_i-1})    
\end{split}
\end{gather*}
for some $1\le i<j\le n$, $t\in \Lie(m_i-m'_i)$ and $t'\in \Lie(m'_i)$, $m'_i\le m_i-1$, where the case of $t'=\id_1$ yields the relation \eqref{[,]cij}, and the relation for $(T_1\otimes\cdots\otimes T_n)P_{\sigma} c^{\otimes d}$ satisfying the condition
\begin{gather*}
\begin{split}
&\sigma(1)=m_1+\cdots+m_{i-1}+1,\quad \sigma(3)=m_1+\cdots+m_{i-1}+2,\\ 
&\sigma(2)=m_1+\cdots+m_{i-1}+m'_i+2 ,\quad
T_i=t (([,](\id_1\otimes t'))\otimes \id_{m_i-m'_i-1})
\end{split}
\end{gather*}
for some $1\le i\le n$, $t\in \Lie(m_i-m'_i)$ and $t'\in \Lie(m'_i)$, $m'_i\le m_i-1$, where the case of $t'=\id_1$ yields the relation \eqref{[,]ci}.
Since we have \eqref{CatLieC01}, these relations reduce to the relations \eqref{[,]cij} and \eqref{[,]ci} by induction on $m'_i$.
This completes the proof.
\end{proof}

\subsection{The category $\A^{L}$ of extended Jacobi diagrams in handlebodies}
The author introduced the $\K$-linear category $\A^{L}$ of extended Jacobi diagrams in handlebodies in \cite{Katada2}, which includes both $\A$ and $\catLieC$ as full subcategories.
We briefly recall the definition of $\A^{L}$.
We refer the readers to \cite[Sections 4.2, 4.3, Appendix A]{Katada2} for details.

The set of objects of the category $\A^{L}$ is the free monoid generated by two objects $H$ and $L$, where $H$ is regarded as a Hopf algebra and $L$ is regarded as a Lie algebra.

Recall that the hom-space $\A(m,n)$ of the category $\A$ is spanned by $(m,n)$-Jacobi diagrams, where Jacobi diagrams are attached to the oriented $1$-manifold $X_n$, that is, the set of univalent vertices is embedded into the interior of $X_n$.
The morphisms of $\A^{L}$ extend the morphisms of $\A$ in such a way that univalent vertices of a Jacobi diagram are allowed to be attached to the upper line $I\times \{0\}\times \{1\}$ in the top square and the bottom line $I\times \{0\}\times \{-1\}$ in the bottom square of the handlebody, and that each connected component of a Jacobi diagram has at least one univalent vertex which is attached to the bottom line or the oriented $1$-manifold.
This means that we do not allow connected components of Jacobi diagrams to be separated from the bottom line and the oriented $1$-manifold.
For $w,w'\in \Ob(\A^{L})$, the hom-space $\A^{L}(w,w')$ is the $\K$-vector space spanned by ``$(w,w')$-diagrams" modulo the STU, AS and IHX relations.
Here, a $(w,w')$-diagram is a Jacobi diagram partly attached to $X_{m'}$ mapped in $U_m$ and partly attached to the upper line or the bottom line of $U_m$ so that the object in $\A^{L}$ corresponding to the upper line (resp. the bottom line) is $w$ (resp. $w'$) when we count from left to right, where
$m=\sum_{i=1}^{r}m_i, \; m'=\sum_{i=1}^{s}m'_i$ for $w=H^{\otimes m_1}\otimes L^{\otimes n_1}\otimes \cdots\otimes H^{\otimes m_r}\otimes L^{\otimes n_r},w'= H^{\otimes m'_1}\otimes L^{\otimes n'_1}\otimes \cdots\otimes H^{\otimes m'_s}\otimes L^{\otimes n'_s}\in \Ob(\A^{L})$.
The following is an example of an $(H^{\otimes 2}\otimes L\otimes H \otimes L\otimes H,L\otimes H\otimes L^{\otimes 2}\otimes H)$-diagram:
\begin{gather*}
    \scalebox{1.0}{$\centre{%% Creator: Inkscape 1.4.2 (ebf0e940, 2025-05-08), www.inkscape.org
%% Accompanies image file 'wdiagram.pdf' (pdf, eps, ps)
\begingroup%
  \makeatletter%
  \providecommand\color[2][]{%
    \errmessage{(Inkscape) Color is used for the text in Inkscape, but the package 'color.sty' is not loaded}%
    \renewcommand\color[2][]{}%
  }%
  \providecommand\transparent[1]{%
    \errmessage{(Inkscape) Transparency is used (non-zero) for the text in Inkscape, but the package 'transparent.sty' is not loaded}%
    \renewcommand\transparent[1]{}%
  }%
  \providecommand\rotatebox[2]{#2}%
  \newcommand*\fsize{\dimexpr\f@size pt\relax}%
  \newcommand*\lineheight[1]{\fontsize{\fsize}{#1\fsize}\selectfont}%
  \ifx\svgwidth\undefined%
    \setlength{\unitlength}{107.71653543bp}%
    \ifx\svgscale\undefined%
      \relax%
    \else%
      \setlength{\unitlength}{\unitlength * \real{\svgscale}}%
    \fi%
  \else%
    \setlength{\unitlength}{\svgwidth}%
  \fi%
  \global\let\svgwidth\undefined%
  \global\let\svgscale\undefined%
  \makeatother%
  \begin{picture}(1,0.44736842)%
    \lineheight{1}%
    \setlength\tabcolsep{0pt}%
    \put(0,0){\includegraphics[width=\unitlength,page=1]{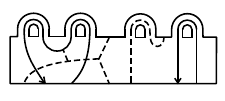}}%
  \end{picture}%
\endgroup%
}$}
\end{gather*}

The composition of morphisms in the category $\A^{L}$ is defined in a way similar to that in the category $\A$, that is, the composition $D'\circ D$ of two morphisms $D'$ and $D$ is defined by stacking a suitable cabling of $D$ on the top square of $D'$.

The category $\A^{L}$ includes the categories $\A$ and $\catLieC$ as full subcategories, that is, $\A$ is the full subcategory of $\A^{L}$ whose set of objects is the monoid generated by $H$, and $\catLieC$ is the full subcategory of $\A^{L}$ whose set of objects is the monoid generated by $L$. 

The degree of a $(w,w')$-diagram is defined by 
$$\frac{1}{2}\#\{\text{vertices}\} -\#\{\text{univalent vertices attached to the upper line}\},$$
which induces an $\N$-grading on the category $\A^{L}$.
For $\omega,\omega'\in \Ob(\A^{L})$, let $\A^{L}_d(\omega,\omega')$ denote the subspace of $\A^{L}(\omega,\omega')$ spanned by morphisms of degree $d$. 
We have 
\begin{gather}\label{AandCatLieCasfullsubcategoriesofAL}
    \A^{L}_d(H^{\otimes m},H^{\otimes n})=\A_d(m,n), \quad\A^{L}_d(L^{\otimes m},L^{\otimes n})=\catLieC(m,n)_d
\end{gather}
for $d,m,n\ge 0$.

We can naturally generalize the $\N$-graded $\K$-linear symmetric monoidal structure of the category $\A$ to the category $\A^{L}$.
The category $\A^{L}$ has a cocommutative Hopf algebra, which is induced by the Hopf algebra $(H,\mu,\eta,\Delta,\varepsilon,S)$ in $\A$, and a Lie algebra $(L,[,]:L^{\otimes 2}\to L)$ with a symmetric invariant $2$-tensor $c: I\to L^{\otimes 2}$,
and two morphisms $i:L\to H$ and $ad_{L}:H\otimes L\to L$ described as follows:
$$[,]=\scalebox{0.8}{$\centre{%% Creator: Inkscape 1.4.2 (ebf0e940, 2025-05-08), www.inkscape.org
%% Accompanies image file 'Liebracket.pdf' (pdf, eps, ps)
\begingroup%
  \makeatletter%
  \providecommand\color[2][]{%
    \errmessage{(Inkscape) Color is used for the text in Inkscape, but the package 'color.sty' is not loaded}%
    \renewcommand\color[2][]{}%
  }%
  \providecommand\transparent[1]{%
    \errmessage{(Inkscape) Transparency is used (non-zero) for the text in Inkscape, but the package 'transparent.sty' is not loaded}%
    \renewcommand\transparent[1]{}%
  }%
  \providecommand\rotatebox[2]{#2}%
  \newcommand*\fsize{\dimexpr\f@size pt\relax}%
  \newcommand*\lineheight[1]{\fontsize{\fsize}{#1\fsize}\selectfont}%
  \ifx\svgwidth\undefined%
    \setlength{\unitlength}{42.51968504bp}%
    \ifx\svgscale\undefined%
      \relax%
    \else%
      \setlength{\unitlength}{\unitlength * \real{\svgscale}}%
    \fi%
  \else%
    \setlength{\unitlength}{\svgwidth}%
  \fi%
  \global\let\svgwidth\undefined%
  \global\let\svgscale\undefined%
  \makeatother%
  \begin{picture}(1,1)%
    \lineheight{1}%
    \setlength\tabcolsep{0pt}%
    \put(0,0){\includegraphics[width=\unitlength,page=1]{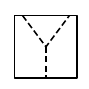}}%
  \end{picture}%
\endgroup%
}$},\; c=\scalebox{0.8}{$\centre{%% Creator: Inkscape 1.4.2 (ebf0e940, 2025-05-08), www.inkscape.org
%% Accompanies image file 'c.pdf' (pdf, eps, ps)
\begingroup%
  \makeatletter%
  \providecommand\color[2][]{%
    \errmessage{(Inkscape) Color is used for the text in Inkscape, but the package 'color.sty' is not loaded}%
    \renewcommand\color[2][]{}%
  }%
  \providecommand\transparent[1]{%
    \errmessage{(Inkscape) Transparency is used (non-zero) for the text in Inkscape, but the package 'transparent.sty' is not loaded}%
    \renewcommand\transparent[1]{}%
  }%
  \providecommand\rotatebox[2]{#2}%
  \newcommand*\fsize{\dimexpr\f@size pt\relax}%
  \newcommand*\lineheight[1]{\fontsize{\fsize}{#1\fsize}\selectfont}%
  \ifx\svgwidth\undefined%
    \setlength{\unitlength}{42.51968504bp}%
    \ifx\svgscale\undefined%
      \relax%
    \else%
      \setlength{\unitlength}{\unitlength * \real{\svgscale}}%
    \fi%
  \else%
    \setlength{\unitlength}{\svgwidth}%
  \fi%
  \global\let\svgwidth\undefined%
  \global\let\svgscale\undefined%
  \makeatother%
  \begin{picture}(1,1)%
    \lineheight{1}%
    \setlength\tabcolsep{0pt}%
    \put(0,0){\includegraphics[width=\unitlength,page=1]{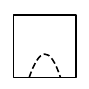}}%
  \end{picture}%
\endgroup%
}$},\;
i=\scalebox{0.8}{$\centre{%% Creator: Inkscape 1.4.2 (ebf0e940, 2025-05-08), www.inkscape.org
%% Accompanies image file 'i.pdf' (pdf, eps, ps)
\begingroup%
  \makeatletter%
  \providecommand\color[2][]{%
    \errmessage{(Inkscape) Color is used for the text in Inkscape, but the package 'color.sty' is not loaded}%
    \renewcommand\color[2][]{}%
  }%
  \providecommand\transparent[1]{%
    \errmessage{(Inkscape) Transparency is used (non-zero) for the text in Inkscape, but the package 'transparent.sty' is not loaded}%
    \renewcommand\transparent[1]{}%
  }%
  \providecommand\rotatebox[2]{#2}%
  \newcommand*\fsize{\dimexpr\f@size pt\relax}%
  \newcommand*\lineheight[1]{\fontsize{\fsize}{#1\fsize}\selectfont}%
  \ifx\svgwidth\undefined%
    \setlength{\unitlength}{42.51968504bp}%
    \ifx\svgscale\undefined%
      \relax%
    \else%
      \setlength{\unitlength}{\unitlength * \real{\svgscale}}%
    \fi%
  \else%
    \setlength{\unitlength}{\svgwidth}%
  \fi%
  \global\let\svgwidth\undefined%
  \global\let\svgscale\undefined%
  \makeatother%
  \begin{picture}(1,1)%
    \lineheight{1}%
    \setlength\tabcolsep{0pt}%
    \put(0,0){\includegraphics[width=\unitlength,page=1]{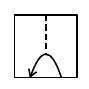}}%
  \end{picture}%
\endgroup%
}$},\;
ad_{L}=\scalebox{0.8}{$\centre{%% Creator: Inkscape 1.4.2 (ebf0e940, 2025-05-08), www.inkscape.org
%% Accompanies image file 'adL.pdf' (pdf, eps, ps)
\begingroup%
  \makeatletter%
  \providecommand\color[2][]{%
    \errmessage{(Inkscape) Color is used for the text in Inkscape, but the package 'color.sty' is not loaded}%
    \renewcommand\color[2][]{}%
  }%
  \providecommand\transparent[1]{%
    \errmessage{(Inkscape) Transparency is used (non-zero) for the text in Inkscape, but the package 'transparent.sty' is not loaded}%
    \renewcommand\transparent[1]{}%
  }%
  \providecommand\rotatebox[2]{#2}%
  \newcommand*\fsize{\dimexpr\f@size pt\relax}%
  \newcommand*\lineheight[1]{\fontsize{\fsize}{#1\fsize}\selectfont}%
  \ifx\svgwidth\undefined%
    \setlength{\unitlength}{42.51968504bp}%
    \ifx\svgscale\undefined%
      \relax%
    \else%
      \setlength{\unitlength}{\unitlength * \real{\svgscale}}%
    \fi%
  \else%
    \setlength{\unitlength}{\svgwidth}%
  \fi%
  \global\let\svgwidth\undefined%
  \global\let\svgscale\undefined%
  \makeatother%
  \begin{picture}(1,1.13333333)%
    \lineheight{1}%
    \setlength\tabcolsep{0pt}%
    \put(0,0){\includegraphics[width=\unitlength,page=1]{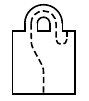}}%
  \end{picture}%
\endgroup%
}$}.$$ 
The morphism $c$ has degree $1$ and the other morphisms have degree $0$.
In \cite[Proposition A.7]{Katada2}, the author proved that the $\K$-linear symmetric monoidal category $\A^{L}$ is generated by the above morphisms. Moreover, we have the following relations involving the morphism $i$:
  \begin{enumerate}[($\A^{L}$-1)]
    \item  $i\; [\cdot,\cdot]=-\mu(i\otimes i)+\mu P_{H,H}(i\otimes i),$
     \label{a1}
    \item  $\Delta i=i\otimes\eta+\eta\otimes i,$
     \label{a2}
    \item  $\epsilon i=0,$
     \label{a3}
    \item $S i=-i.$
     \label{Si}
  \end{enumerate}
  
The Casimir Hopf algebra in $\A^{L}$ that characterizes the $\K$-linear PROP $\A$ is $(H,\mu,\eta,\Delta,\varepsilon,S,\tilde{c}=i^{\otimes2}c)$, and the Casimir Lie algebra that characterizes the $\K$-linear PROP $\catLieC$ is $(L,[,],c)$.

\subsection{The degree $0$ part $\A^{L}_0$}
The degree $0$ part $\A^{L}_0$ of the category $\A^{L}$ forms a wide subcategory of $\A^{L}$.
Since $\kgrop$ is regarded as the degree $0$ part $\A_0$ of the category $\A$, which is a full subcategory of $\A^{L}$, the category $\kgrop$ is a full subcategory of $\A^{L}_0$ whose objects are generated by $H$.
Similarly, $\catLie$ is a full subcategory of $\A^{L}_0$ whose objects are generated by $L$.

Here, we explicitly describe the hom-space $\A^{L}_0(L^{\otimes m},H^{\otimes n})$.
Let $C(m,n)$ denote the set of Jacobi diagrams (in a cube) with $m$ chords such that one end of each chord is attached to the upper line and the other end of it is attached to the oriented $1$-manifold $X_n$.
For example, the set $C(2,2)$ consists of the following six elements
\begin{gather*}
\scalebox{1.0}{$\centre{%% Creator: Inkscape 1.4.2 (ebf0e940, 2025-05-08), www.inkscape.org
%% Accompanies image file 'C221.pdf' (pdf, eps, ps)
\begingroup%
  \makeatletter%
  \providecommand\color[2][]{%
    \errmessage{(Inkscape) Color is used for the text in Inkscape, but the package 'color.sty' is not loaded}%
    \renewcommand\color[2][]{}%
  }%
  \providecommand\transparent[1]{%
    \errmessage{(Inkscape) Transparency is used (non-zero) for the text in Inkscape, but the package 'transparent.sty' is not loaded}%
    \renewcommand\transparent[1]{}%
  }%
  \providecommand\rotatebox[2]{#2}%
  \newcommand*\fsize{\dimexpr\f@size pt\relax}%
  \newcommand*\lineheight[1]{\fontsize{\fsize}{#1\fsize}\selectfont}%
  \ifx\svgwidth\undefined%
    \setlength{\unitlength}{42.51968504bp}%
    \ifx\svgscale\undefined%
      \relax%
    \else%
      \setlength{\unitlength}{\unitlength * \real{\svgscale}}%
    \fi%
  \else%
    \setlength{\unitlength}{\svgwidth}%
  \fi%
  \global\let\svgwidth\undefined%
  \global\let\svgscale\undefined%
  \makeatother%
  \begin{picture}(1,1)%
    \lineheight{1}%
    \setlength\tabcolsep{0pt}%
    \put(0,0){\includegraphics[width=\unitlength,page=1]{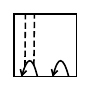}}%
  \end{picture}%
\endgroup%
}$},
\scalebox{1.0}{$\centre{%% Creator: Inkscape 1.4.2 (ebf0e940, 2025-05-08), www.inkscape.org
%% Accompanies image file 'C222.pdf' (pdf, eps, ps)
\begingroup%
  \makeatletter%
  \providecommand\color[2][]{%
    \errmessage{(Inkscape) Color is used for the text in Inkscape, but the package 'color.sty' is not loaded}%
    \renewcommand\color[2][]{}%
  }%
  \providecommand\transparent[1]{%
    \errmessage{(Inkscape) Transparency is used (non-zero) for the text in Inkscape, but the package 'transparent.sty' is not loaded}%
    \renewcommand\transparent[1]{}%
  }%
  \providecommand\rotatebox[2]{#2}%
  \newcommand*\fsize{\dimexpr\f@size pt\relax}%
  \newcommand*\lineheight[1]{\fontsize{\fsize}{#1\fsize}\selectfont}%
  \ifx\svgwidth\undefined%
    \setlength{\unitlength}{42.51968504bp}%
    \ifx\svgscale\undefined%
      \relax%
    \else%
      \setlength{\unitlength}{\unitlength * \real{\svgscale}}%
    \fi%
  \else%
    \setlength{\unitlength}{\svgwidth}%
  \fi%
  \global\let\svgwidth\undefined%
  \global\let\svgscale\undefined%
  \makeatother%
  \begin{picture}(1,1)%
    \lineheight{1}%
    \setlength\tabcolsep{0pt}%
    \put(0,0){\includegraphics[width=\unitlength,page=1]{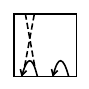}}%
  \end{picture}%
\endgroup%
}$},
\scalebox{1.0}{$\centre{%% Creator: Inkscape 1.4.2 (ebf0e940, 2025-05-08), www.inkscape.org
%% Accompanies image file 'C224.pdf' (pdf, eps, ps)
\begingroup%
  \makeatletter%
  \providecommand\color[2][]{%
    \errmessage{(Inkscape) Color is used for the text in Inkscape, but the package 'color.sty' is not loaded}%
    \renewcommand\color[2][]{}%
  }%
  \providecommand\transparent[1]{%
    \errmessage{(Inkscape) Transparency is used (non-zero) for the text in Inkscape, but the package 'transparent.sty' is not loaded}%
    \renewcommand\transparent[1]{}%
  }%
  \providecommand\rotatebox[2]{#2}%
  \newcommand*\fsize{\dimexpr\f@size pt\relax}%
  \newcommand*\lineheight[1]{\fontsize{\fsize}{#1\fsize}\selectfont}%
  \ifx\svgwidth\undefined%
    \setlength{\unitlength}{42.51968504bp}%
    \ifx\svgscale\undefined%
      \relax%
    \else%
      \setlength{\unitlength}{\unitlength * \real{\svgscale}}%
    \fi%
  \else%
    \setlength{\unitlength}{\svgwidth}%
  \fi%
  \global\let\svgwidth\undefined%
  \global\let\svgscale\undefined%
  \makeatother%
  \begin{picture}(1,1)%
    \lineheight{1}%
    \setlength\tabcolsep{0pt}%
    \put(0,0){\includegraphics[width=\unitlength,page=1]{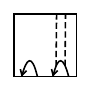}}%
  \end{picture}%
\endgroup%
}$},
\scalebox{1.0}{$\centre{%% Creator: Inkscape 1.4.2 (ebf0e940, 2025-05-08), www.inkscape.org
%% Accompanies image file 'C223.pdf' (pdf, eps, ps)
\begingroup%
  \makeatletter%
  \providecommand\color[2][]{%
    \errmessage{(Inkscape) Color is used for the text in Inkscape, but the package 'color.sty' is not loaded}%
    \renewcommand\color[2][]{}%
  }%
  \providecommand\transparent[1]{%
    \errmessage{(Inkscape) Transparency is used (non-zero) for the text in Inkscape, but the package 'transparent.sty' is not loaded}%
    \renewcommand\transparent[1]{}%
  }%
  \providecommand\rotatebox[2]{#2}%
  \newcommand*\fsize{\dimexpr\f@size pt\relax}%
  \newcommand*\lineheight[1]{\fontsize{\fsize}{#1\fsize}\selectfont}%
  \ifx\svgwidth\undefined%
    \setlength{\unitlength}{42.51968504bp}%
    \ifx\svgscale\undefined%
      \relax%
    \else%
      \setlength{\unitlength}{\unitlength * \real{\svgscale}}%
    \fi%
  \else%
    \setlength{\unitlength}{\svgwidth}%
  \fi%
  \global\let\svgwidth\undefined%
  \global\let\svgscale\undefined%
  \makeatother%
  \begin{picture}(1,1)%
    \lineheight{1}%
    \setlength\tabcolsep{0pt}%
    \put(0,0){\includegraphics[width=\unitlength,page=1]{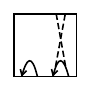}}%
  \end{picture}%
\endgroup%
}$},
\scalebox{1.0}{$\centre{%% Creator: Inkscape 1.4.2 (ebf0e940, 2025-05-08), www.inkscape.org
%% Accompanies image file 'C225.pdf' (pdf, eps, ps)
\begingroup%
  \makeatletter%
  \providecommand\color[2][]{%
    \errmessage{(Inkscape) Color is used for the text in Inkscape, but the package 'color.sty' is not loaded}%
    \renewcommand\color[2][]{}%
  }%
  \providecommand\transparent[1]{%
    \errmessage{(Inkscape) Transparency is used (non-zero) for the text in Inkscape, but the package 'transparent.sty' is not loaded}%
    \renewcommand\transparent[1]{}%
  }%
  \providecommand\rotatebox[2]{#2}%
  \newcommand*\fsize{\dimexpr\f@size pt\relax}%
  \newcommand*\lineheight[1]{\fontsize{\fsize}{#1\fsize}\selectfont}%
  \ifx\svgwidth\undefined%
    \setlength{\unitlength}{42.51968504bp}%
    \ifx\svgscale\undefined%
      \relax%
    \else%
      \setlength{\unitlength}{\unitlength * \real{\svgscale}}%
    \fi%
  \else%
    \setlength{\unitlength}{\svgwidth}%
  \fi%
  \global\let\svgwidth\undefined%
  \global\let\svgscale\undefined%
  \makeatother%
  \begin{picture}(1,1)%
    \lineheight{1}%
    \setlength\tabcolsep{0pt}%
    \put(0,0){\includegraphics[width=\unitlength,page=1]{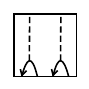}}%
  \end{picture}%
\endgroup%
}$},
\scalebox{1.0}{$\centre{%% Creator: Inkscape 1.4.2 (ebf0e940, 2025-05-08), www.inkscape.org
%% Accompanies image file 'C226.pdf' (pdf, eps, ps)
\begingroup%
  \makeatletter%
  \providecommand\color[2][]{%
    \errmessage{(Inkscape) Color is used for the text in Inkscape, but the package 'color.sty' is not loaded}%
    \renewcommand\color[2][]{}%
  }%
  \providecommand\transparent[1]{%
    \errmessage{(Inkscape) Transparency is used (non-zero) for the text in Inkscape, but the package 'transparent.sty' is not loaded}%
    \renewcommand\transparent[1]{}%
  }%
  \providecommand\rotatebox[2]{#2}%
  \newcommand*\fsize{\dimexpr\f@size pt\relax}%
  \newcommand*\lineheight[1]{\fontsize{\fsize}{#1\fsize}\selectfont}%
  \ifx\svgwidth\undefined%
    \setlength{\unitlength}{42.51968504bp}%
    \ifx\svgscale\undefined%
      \relax%
    \else%
      \setlength{\unitlength}{\unitlength * \real{\svgscale}}%
    \fi%
  \else%
    \setlength{\unitlength}{\svgwidth}%
  \fi%
  \global\let\svgwidth\undefined%
  \global\let\svgscale\undefined%
  \makeatother%
  \begin{picture}(1,1)%
    \lineheight{1}%
    \setlength\tabcolsep{0pt}%
    \put(0,0){\includegraphics[width=\unitlength,page=1]{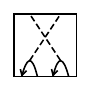}}%
  \end{picture}%
\endgroup%
}$}.
\end{gather*}

\begin{lemma}\label{AL0}
   The set $C(m,n)$ is a basis for the hom-space $\A^{L}_0(L^{\otimes m},H^{\otimes n})$.
\end{lemma}

\begin{proof}
Let $J(m,n)$ denote the set of Jacobi diagrams on $X_n$ in a cube with $2m$ vertices such that $m$ univalent vertices are attached to the upper line.
Note that each connected component of a Jacobi diagram in $J(m,n)$ is a rooted trivalent tree whose leaves are attached to the upper line and whose root is attached to $X_n$.

By the definition of the degree $0$ part of the hom-space of the category $\A^{L}$, $\A^{L}_0(L^{\otimes m},H^{\otimes n})$ is the $\K$-vector space spanned by the set $J(m,n)$ modulo the STU relation.
By the STU relation, any Jacobi diagram in $J(m,n)$ can be written as a linear combination of chord diagrams in $C(m,n)$.
Therefore, $\A^{L}_0(L^{\otimes m},H^{\otimes n})$ is spanned by the set $C(m,n)$.
Moreover, for each Jacobi diagram $D$ in $J(m,n)$, the linear combination of chord diagrams in $C(m,n)$ associated to $D$ is uniquely determined.
Therefore, $C(m,n)$ is a basis for $\A^{L}_0(L^{\otimes m},H^{\otimes n})$.
\end{proof}

\begin{lemma}\label{bijectionwithCmn}
We have
\begin{gather*}
    C(m,n)=\{\mu^{[p_1,\cdots,p_n]}P_{\sigma} i^{\otimes m} \mid  p_1,\cdots,p_n\ge 0, p_1+\cdots+p_n=m, \sigma\in \gpS_m\}.
\end{gather*}
\end{lemma}

\begin{proof}
For an element $C$ of $C(m,n)$, let $E(C)$ be the set of $m$ chords of $C$.
Define two maps $e$ and $e'$ from $E(C)$ to $\{1,\cdots,m\}$ as follows. The map $e$ assigns $e(c)\in \{1,\cdots,m\}$ to a chord $c\in E(C)$, where an endpoint of $c$ is attached to the $e(c)$-th point on the upper line when we count the endpoints from left to right. In a similar way, the map $e'$ assigns $e'(c)$ to $c$, where an endpoint of $c$ is attached to the $e'(c)$-th point on $X_n$ when we count the endpoints on $X_n$ from left to right on the left-most arc component of $X_n$ and from left to right on the second arc and so on.
Then we obtain a permutation $\sigma_C\in \gpS_m$ which assigns $e(c)$ to $e'(c)$.
For $j\in \{1,\cdots,n\}$, let $p_j(C)$ be the number of univalent vertices of $C$ that are attached to the $j$-th component of $X_n$.
Then we have a bijection
\begin{gather*}
   \Phi: C(m,n)\xrightarrow{\cong} \gpS_m\times \{(p_1,\cdots,p_n)\mid p_1,\cdots,p_n\ge 0, p_1+\cdots+p_n=m\}
\end{gather*}
defined by $\Phi(C)=(\sigma_{C},(p_1(C),\cdots,p_n(C)))$.
In other words, we have 
\begin{gather*}
    C(m,n)=\{\mu^{[p_1,\cdots,p_n]}P_{\sigma} i^{\otimes m} \mid  p_1,\cdots,p_n\ge 0, p_1+\cdots+p_n=m, \sigma\in \gpS_m\}.
\end{gather*}
\end{proof}

\subsection{The hom-space $\A^{L}(L^{\otimes m},H^{\otimes n})$}
Here, we explicitly describe the hom-space $\A^{L}(L^{\otimes m},H^{\otimes n})$ and more generally $\A^{L}(w,H^{\otimes n})$ for $w\in \Ob(\A^{L})$.

\begin{lemma}\label{lemmaALandAL0d}
For any $d, m,n\ge 0$, we have a $\K$-linear isomorphism
\begin{gather}\label{ALandAL0d}
\A^{L}_d(L^{\otimes m},H^{\otimes n})\cong \A^{L}_0(L^{\otimes -},H^{\otimes n})\otimes_{\catLie}\catLieC(m,-)_d.
\end{gather}
\end{lemma}

\begin{proof}
The composition of morphisms in $\A^{L}$ induces a $\K$-linear isomorphism
\begin{gather*}
    \A^{L}_0(-,H^{\otimes n})\otimes_{\A^{L}_0}\A^{L}_d(L^{\otimes m},-)\xrightarrow[\cong]{\circ} \A^{L}_d(L^{\otimes m},H^{\otimes n}),
\end{gather*}
which induces an injective $\K$-linear map
\begin{gather*}
    \A^{L}_0(L^{\otimes -},H^{\otimes n})\otimes_{\catLie}\A^{L}_d(L^{\otimes m},L^{\otimes -})\xrightarrow{\circ}  \A^{L}_d(L^{\otimes m},H^{\otimes n}).
\end{gather*}
It is easy to see that any element of $\A^{L}_d(L^{\otimes m},H^{\otimes n})$ is a $\K$-linear sum of morphisms of the form
\begin{gather*}
    f(\id_{L^{\otimes m}}\otimes c^{\otimes d})
\end{gather*}
with $f\in \A^{L}_0(L^{\otimes m+2d},H^{\otimes n})$.
Therefore, the composition map $\circ$ is surjective.
Hence, the identification \eqref{AandCatLieCasfullsubcategoriesofAL} of $\A^{L}_d(L^{\otimes m},L^{\otimes -})$ and $\catLieC(m,-)_d$ completes the proof.
\end{proof}

In a similar way, we obtain the following factorization of $\A^{L}_d(L^{\otimes m},H^{\otimes n})$.

\begin{lemma}\label{lemmaALandAd}
For any $d, m,n\ge 0$, we have $\K$-linear isomorphisms
\begin{gather}\label{ALandAd}
\A^{L}_d(L^{\otimes m},H^{\otimes n})\cong \A_d(-, n)\otimes_{\kgrop}\A^{L}_0(L^{\otimes m},H^{\otimes -}).
\end{gather}
\end{lemma}

We have the following factorization of an element of $\A^{L}_d(L^{\otimes m},H^{\otimes n})$.

\begin{lemma}\label{factorizationofAL}
Any element of $\A^{L}_d(L^{\otimes m},H^{\otimes n})$ is a linear combination of morphisms of the form
\begin{gather*}
        \mu^{[p_1,\cdots,p_n]}P_{\sigma}i^{\otimes (m+2d)}(\id_{L^{\otimes m}}\otimes c^{\otimes d})
\end{gather*}
with $p_1,\cdots,p_n\ge 0$, $p_1+\cdots+p_n=m+2d$, $\sigma\in \gpS_{m+2d}$.
\end{lemma}

\begin{proof}
By Lemma \ref{lemmaALandAL0d}, we have
\begin{gather*}
    \A^{L}_d(L^{\otimes m},H^{\otimes n})\cong \A^{L}_0(L^{\otimes -},H^{\otimes n})\otimes_{\catLie}\catLieC(m,-)_d.
\end{gather*}
By Lemmas \ref{AL0} and \ref{bijectionwithCmn}, the set $C(m',n)=\{\mu^{[p_1,\cdots,p_n]}P_{\sigma} i^{\otimes m'} \mid  p_1,\cdots,p_n\ge 0, p_1+\cdots+p_n=m', \sigma\in \gpS_{m'}\}$
is a basis for $\A^{L}_0(L^{\otimes m'},H^{\otimes n})$.
Since any element of $\catLieC(m,m')_d$ is a linear combination of 
$f(\id_{L^{\otimes m}}\otimes c^{\otimes d})$ for some $f\in \catLie(m+2d,m')$, the statement follows.
\end{proof}

\begin{remark}
Note that the factorization of an element of $\A^{L}_d(L^{\otimes m},H^{\otimes n})$ given in Lemma \ref{factorizationofAL} is not unique.
Indeed, we have the relation corresponding to the ``$4$T-relation".
For example, we have
\begin{gather*}
\begin{split}
    &\mu^{[2,1]}P_{(12)}i^{\otimes 3}(\id_L\otimes c)-\mu^{[2,1]}i^{\otimes 3}(\id_L\otimes c)
    \\
    &=\mu^{[1,2]}P_{(12)}i^{\otimes 3}(\id_L\otimes c)-\mu^{[1,2]}P_{(132)}i^{\otimes 3}(\id_L\otimes c).
\end{split}
\end{gather*}
\end{remark}

More generally, we have the following factorization of an element of $\A^{L}_d(w,H^{\otimes m'})$ for any $w\in \Ob(\A^{L})$.
For $w=H^{\otimes m_1}\otimes L^{\otimes n_1}\otimes \cdots\otimes H^{\otimes m_r}\otimes L^{\otimes n_r}$, set $m=\sum_{j=1}^{r}m_j,\; n=\sum_{j=1}^{r}n_j$.
We have an isomorphism $P_w:H^{\otimes m}\otimes L^{\otimes n} \xrightarrow{\cong} w$ induced by the symmetry of $\A^{L}$, which induces a $\K$-linear isomorphism
$$-\circ P_w:\A^{L}_d(w,H^{\otimes m'})\xrightarrow{\cong} \A^{L}_d(H^{\otimes m}\otimes L^{\otimes n},H^{\otimes m'}).$$

\begin{lemma}
    Any element of $\A^{L}_d(H^{\otimes m}\otimes L^{\otimes n},H^{\otimes m'})$ is a linear combination of morphisms of the form
\begin{gather*}
        \mu^{[p_1,\cdots,p_{m'}]}P_{\sigma}
        (((S^{e_1}\otimes\cdots \otimes S^{e_s})\Delta^{[q_1,\cdots,q_m]})\otimes (i^{\otimes n+2d}(\id_{L^{\otimes n}}\otimes c^{\otimes d})))
\end{gather*}
with $s, p_1,\cdots,p_{m'}, q_1,\cdots,q_m\ge 0$, 
$s=p_1+\cdots+p_{m'}-(n+2d)=q_1+\cdots+q_m$, $\sigma\in \gpS_{s+n+2d}$, $e_1,\cdots,e_s\in \{0,1\}$.
\end{lemma}

\begin{proof}
    By the STU relation, any element of $\A^{L}_d(H^{\otimes m}\otimes L^{\otimes n},H^{\otimes m'})$ is a linear combination of chord diagrams with $d$ chords attached to only $X_{m'}$ and with $n$ chords attached to both $X_{m'}$ and the upper line. 
    By pulling the univalent vertex of each of the $n$ chords that is attached to $X_{m'}$ toward the upper right side of $U_{m}$ along the chord so that the chord does not go through any handles of $U_{m}$, we can transform the chord diagrams into the following form
    \begin{gather*}
        f (\id_{H^{\otimes m}}\otimes i^{\otimes n}) 
    \end{gather*}
    with $f\in \A_d(m+n,m')$.
    By Lemma \ref{decompositionofA}, the morphism $f$ is a linear combination of 
    \begin{gather*}
         \mu^{[r_1,\cdots,r_{m'}]}P_{\rho}(S^{e_1}\otimes\cdots\otimes S^{e_{t}}\otimes \id_{2d})(\Delta^{[q_1,\cdots,q_{m+n}]}\otimes (i^{\otimes 2}c)^{\otimes d})
    \end{gather*}
    where $t, r_1,\cdots,r_{m'}, q_1,\cdots,q_{m+n}\ge 0$, $t=r_1+\cdots+r_{m'}-2d=q_1+\cdots+q_{m+n}$, $\rho\in \gpS_{t+2d}$, $e_1,\cdots,e_t\in \{0,1\}$.
    We have  
    \begin{gather*}
       \begin{split}
           &\mu^{[r_1,\cdots,r_{m'}]}P_{\rho}(S^{e_1}\otimes\cdots\otimes S^{e_{t}}\otimes \id_{2d})(\Delta^{[q_1,\cdots,q_{m+n}]}\otimes (i^{\otimes 2}c)^{\otimes d}) (\id_{H^{\otimes m}}\otimes i^{\otimes n})\\
           &=\mu^{[r_1,\cdots,r_{m'}]}P_{\rho} 
           (
           ((S^{e_1}\otimes\cdots\otimes S^{e_{s}})\Delta^{[q_1,\cdots,q_{m}]})
           \otimes 
           ((S^{e_{s+1}}\otimes\cdots\otimes S^{e_{t}})\Delta^{[q_{m+1},\cdots,q_{m+n}]}i^{\otimes n})
           \otimes 
           i^{\otimes 2d}c^{\otimes d}
           ).
       \end{split}
    \end{gather*}
    By the relations ($\A^{L}$-\ref{a2}), ($\A^{L}$-\ref{a3}) and ($\A^{L}$-\ref{Si}), 
    $(S^{e_{s+1}}\otimes\cdots\otimes S^{e_{t}})\Delta^{[q_{m+1},\cdots,q_{m+n}]}i^{\otimes n}$
    is a linear combination of 
    $P_{\tau} (\eta^{\otimes q} \otimes i^{\otimes n})$
    with $\tau\in \gpS_{q+n}$ if $q=q_{m+1}+\cdots+q_{m+n}-n\ge 0$ and otherwise vanishes.
    Therefore, any element of $\A^{L}_d(H^{\otimes m}\otimes L^{\otimes n},H^{\otimes m'})$ is a linear combination of
    \begin{gather*}
        \mu^{[p_1,\cdots,p_{m'}]}P_{\sigma}
        (((S^{e_1}\otimes\cdots \otimes S^{e_s})\Delta^{[q_1,\cdots,q_m]})\otimes (i^{\otimes n+2d}(\id_{L^{\otimes n}}\otimes c^{\otimes d}))),
    \end{gather*}
    which completes the proof.
\end{proof}

If $w'\in \Ob(\A^{L})\setminus \Ob(\A)$, then we cannot factorize morphisms of $\A^{L}(w,w')$ into linear combinations of chord diagrams as in the above cases.

\section{Functors on the category of Jacobi diagrams in handlebodies}\label{sectionadjunction}

In this section, we study the adjunction given by Powell, which we recalled in Proposition \ref{adjunctionPowell}, by using the bimodule induced by the hom-spaces of the category $\A^{L}_0$.
Moreover, by generalizing this point of view, we obtain an adjunction between the category of $\A$-modules and the category of $\catLieC$-modules.

\subsection{The adjunction given by Powell}

We reconsider the adjunction in Proposition \ref{adjunctionPowell} by using the hom-spaces of the category $\A^{L}_0$.

We have a $(\kgrop, \catLie)$-bimodule $\A^{L}_0(L^{\otimes -},H^{\otimes -})$ induced by the composition in the category $\A^{L}_0$, that is, the $\kgrop$-module structure is defined by the $\K$-linear map
\begin{gather*}
\A^{L}_0(H^{\otimes n},H^{\otimes n'})\otimes \A^{L}_0(L^{\otimes -},H^{\otimes n})\xrightarrow{\circ} \A^{L}_0(L^{\otimes -},H^{\otimes n'}),
\end{gather*}
where we identify $\kgrop(n,n')$ with $\A^{L}_0(H^{\otimes n},H^{\otimes n'})$,
and the right $\catLie$-module structure is defined by the $\K$-linear map
\begin{gather*}  
\A^{L}_0(L^{\otimes m},H^{\otimes -})\otimes \A^{L}_0(L^{\otimes m'},L^{\otimes m})\xrightarrow{\circ}
\A^{L}_0(L^{\otimes m'},H^{\otimes -}),  
\end{gather*}
where we identify $\catLie(m',m)$ with $\A^{L}_0(L^{\otimes m'},L^{\otimes m})$.

The hom-space $\A^{L}_0(L^{\otimes m},H^{\otimes n})$ can be identified with $\catAss(m,n)$ as follows.
The $\K$-linear PROP $\catAss$ is embedded into the  $\K$-linear PROP $\A^{L}_0$ via the unique morphism of $\K$-linear PROPs that maps $(A,\mu_A,\eta_A)$ to $(H,\mu,\eta)$ 
since $\catAss$ is freely generated by the unital associative algebra $(A,\mu_A,\eta_A)$ and the full subcategory $\kgrop$ of $\A^{L}_0$ is freely generated by the cocommutative Hopf algebra $(H,\mu,\eta,\Delta,\varepsilon,S)$.
We have a $\catAss$-module map
\begin{gather*}
    i^*_m: \catAss(m,-)\hookrightarrow\A^{L}_0(H^{\otimes m},H^{\otimes -}) \xrightarrow{-\circ i^{\otimes m}}\A^{L}_0(L^{\otimes m},H^{\otimes -})
\end{gather*}
defined by the composition of the embedding $\catAss(m,n)\hookrightarrow\A^{L}_0(H^{\otimes m},H^{\otimes n})$ and the composition map with the morphism $i^{\otimes m}$.

\begin{lemma}\label{catAssi}
For $m\ge 0$, $i^*_m$ is an isomorphism of $\catAss$-modules.
\end{lemma}

\begin{proof}
By Lemma \ref{factorcatAss}, the set $\{\mu_A^{[p_1,\cdots,p_n]}P_{\sigma}\mid p_1+\cdots+p_n=m ,\sigma\in \gpS_{m}\}$ is a basis for $\catAss(m,n)$.
By Lemmas \ref{AL0} and \ref{bijectionwithCmn}, $\A^{L}_0(L^{\otimes m},H^{\otimes n})$ has a basis $\{\mu^{[p_1,\cdots,p_n]}P_{\sigma}i^{\otimes m}\mid p_1+\cdots+p_n=m ,\sigma\in \gpS_{m}\}$.
Therefore, the $\K$-linear map $(i^*_m)_n:\catAss(m,n)\to\A^{L}_0(L^{\otimes m},H^{\otimes n})$ is an isomorphism for any $n\ge 0$ and thus the morphism $i^*_m$ is an isomorphism of $\catAss$-modules.
\end{proof}

\begin{remark}\label{kgropinAL}
For $m\ge 0$, we have a $\kgrop$-module map
\begin{gather*}
    -\circ i^{\otimes m}:\kgrop(m,-)\cong \A^{L}_0(H^{\otimes m},H^{\otimes -}) \to  \A^{L}_0(L^{\otimes m},H^{\otimes -}).
\end{gather*}
For $m=0$, the $\kgrop$-module map $-\circ i^{\otimes 0}=\id_{\K}$ is an isomorphism. For $m\ge 1$, however, the $\kgrop$-module map $-\circ i^{\otimes m}$ is epic but not monic due to the relation ($\A^{L}$-\ref{a3}).
\end{remark}

The $(\kgrop, \catLie)$-bimodule structure on $\A^{L}_0(L^{\otimes -},H^{\otimes -})$ that we defined above induces a $(\kgrop, \catLie)$-bimodule structure on $\catAss$ via the isomorphisms 
$(i^*_m)_n:\catAss(m,n)\xrightarrow{\cong}\A^{L}_0(L^{\otimes m},H^{\otimes n})$.
Then the $(\kgrop, \catLie)$-bimodule $\catAss$ is isomorphic to the $(\kgrop, \catLie)$-bimodule ${}_{\Delta}\catAss$ defined by Powell \cite{Powellanalytic}, which we recalled in Section \ref{sectionPowell}.
Indeed, the $\K$-linear isomorphisms
\begin{gather*}
    (-1)^m: \catAss(m,n)\to {}_{\Delta}\catAss(m,n)
\end{gather*}
form a natural isomorphism between the $(\kgrop, \catLie)$-bimodules $\catAss$ and ${}_{\Delta}\catAss$.

We can rewrite the statement of Proposition \ref{adjunctionPowell} as follows.

\begin{proposition}[cf. Proposition \ref{adjunctionPowell}]\label{adjunctionpowellAL0}
    We have the hom-tensor adjunction
    \begin{gather*}
        \A^{L}_0(L^{\otimes -},H^{\otimes -})\otimes_{\catLie} - : \adj{\catLieMod}{\kgropMod}{}{} : \Hom_{\kgropMod}(\A^{L}_0(L^{\otimes -},H^{\otimes -}),-).
    \end{gather*}
    Moreover, the adjunction restricts to an equivalence of categories
    \begin{gather*}
        \catLieMod\simeq \kgropMod^{\omega}.
    \end{gather*}
\end{proposition}

\subsection{An adjunction between the categories $\catLieCMod$ and $\AMod$}\label{sectionadjunctiononA}

We generalize the work of Powell \cite{Powellanalytic} to $\A$-modules by generalizing the perspective of the last subsection.

We have an $(\A,\catLieC)$-bimodule $\A^{L}(L^{\otimes -},H^{\otimes -})$ induced by the composition in the category $\A^{L}$, that is, the $\A$-module structure is defined by the $\K$-linear map
\begin{gather*}
\A^{L}(H^{\otimes n},H^{\otimes n'})\otimes \A^{L}(L^{\otimes -},H^{\otimes n})\xrightarrow{\circ} \A^{L}(L^{\otimes -},H^{\otimes n'}),
\end{gather*}
where we identify $\A(n,n')$ with $\A^{L}(H^{\otimes n},H^{\otimes n'})$,
and the right $\catLieC$-module structure is defined by the $\K$-linear map
\begin{gather*}  
\A^{L}(L^{\otimes m},H^{\otimes -})\otimes \A^{L}(L^{\otimes m'},L^{\otimes m})\xrightarrow{\circ}
\A^{L}(L^{\otimes m'},H^{\otimes -}),  
\end{gather*}
where we identify $\catLieC(m',m)$ with $\A^{L}(L^{\otimes m'},L^{\otimes m})$.

Then the general hom-tensor adjunction gives the following adjunction, which generalizes the adjunction given by Powell.

\begin{theorem}\label{adjunctionA}
We have an adjunction
\begin{gather*}
    \A^{L}(L^{\otimes -},H^{\otimes -})\otimes_{\catLieC} - : \adj{\catLieCMod}{\AMod}{}{} : \Hom_{\AMod}(\A^{L}(L^{\otimes -},H^{\otimes -}),-).
\end{gather*}
\end{theorem}

We have an isomorphism of $(\kgrop,\catLieC)$-bimodules
\begin{gather}\label{ALandAL0}
\A^{L}(L^{\otimes -},H^{\otimes -})\cong \A^{L}_0(L^{\otimes -},H^{\otimes -})\otimes_{\catLie}\catLieC
\end{gather}
induced by the $\K$-linear isomorphism \eqref{ALandAL0d}, and an isomorphism of $(\A,\catLie)$-bimodules
\begin{gather}\label{ALandA}
\begin{split}
\A^{L}(L^{\otimes -},H^{\otimes -})
&\cong \A\otimes_{\kgrop}\A^{L}_0(L^{\otimes -},H^{\otimes -})
\end{split}
\end{gather}
induced by the $\K$-linear isomorphism \eqref{ALandAd}.

\begin{remark}
Since by Lemma \ref{factorizationofAL}, any element of $\A^{L}_d(L^{\otimes m},H^{\otimes n})$ is a linear combination of morphisms of the form
\begin{gather*}
        \mu^{[p_1,\cdots,p_n]}P_{\sigma}i^{\otimes (m+2d)}(\id_{L^{\otimes m}}\otimes c^{\otimes d})=\mu^{[p_1,\cdots,p_n]}P_{\sigma}(\id_{H^{\otimes m}}\otimes {\tilde{c}}^{\otimes d})i^{\otimes m},
\end{gather*}
where $\tilde{c}=i^{\otimes 2}c$, the composition map $(i^*_m)_n$ with the morphism $i^{\otimes m}$
\begin{gather*}
    (i^*_m)_n: \A(m,n)\xrightarrow{\cong} \A^{L}(H^{\otimes m},H^{\otimes n})\xrightarrow{-\circ i^{\otimes m}}\A^{L}(L^{\otimes m},H^{\otimes n})
\end{gather*}
is surjective.
However, the map $(i^*_m)_n$ is not injective for $m\ge 1$ due to the relation ($\A^{L}$-\ref{a3}).
Remark \ref{kgropinAL} follows from this remark as the case of degree $0$.
\end{remark}

\subsection{Adjunctions between $\catLieMod, \catLieCMod,\kgropMod$ and $\AMod$}\label{sec4modcats}

Here, we study adjunctions between the categories $\catLieMod, \catLieCMod,\kgropMod$ and $\AMod$.
We will write $\A^{L}$ for the $(\A,\catLieC)$-bimodule $\A^{L}(L^{\otimes -}, H^{\otimes -})$ and $\A_0^{L}$ for the $(\kgrop,\catLie)$-bimodule $\A^{L}_0(L^{\otimes -}, H^{\otimes -})$ for simplicity.

The inclusion functor
\begin{gather*}
   \kgrop\hookrightarrow \A
\end{gather*}
induces a forgetful functor 
\begin{gather}\label{UforA}
   U: \AMod\to \kgropMod.
\end{gather}
The functor $U$ is exact and has a left adjoint functor \begin{gather}\label{FforA}
    F=\A\otimes_{\kgrop}-: \kgropMod\to \AMod.
\end{gather}
Note that from \eqref{ALandA}, we have an isomorphism of $\A$-modules
\begin{gather}\label{ALFAL0}
    \A^L\cong \A\otimes_{\kgrop}\A^{L}_0=F(\A^{L}_0).
\end{gather}

In a similar way, the inclusion functor
\begin{gather*}
   \catLie\hookrightarrow \catLieC
\end{gather*}
induces a forgetful functor
\begin{gather*}
    U:\catLieCMod\to \catLieMod,
\end{gather*}
which is exact and has a left adjoint functor
\begin{gather*}\label{FforcatLieC}
     F=\catLieC\otimes_{\catLie}-: \catLieMod\to \catLieCMod.
\end{gather*}
Then the relations between the forgetful functors and the adjoint functors are as follows.

\begin{lemma}\label{tensorandU}
We have the following commutative diagram (up to isomorphisms)
      \begin{gather*}
        \xymatrix{
        \catLieCMod 
        \ar[rrr]^-{\A^{L}\otimes_{\catLieC}-} \ar[d]_{U}
        &&& 
        \AMod\ar[d]^{U}\\
        \catLieMod \ar[rrr]_-{\A^{L}_0\otimes_{\catLie}-} 
        &&& \kgropMod
        \ar@{}[lllu]|{\circlearrowleft}.
        }
 \end{gather*}
\end{lemma}

\begin{proof}
It follows from the isomorphism \eqref{ALandAL0} of $(\kgrop,\catLieC)$-bimodules that for a $\catLieC$-module $K$, we have isomorphisms of $\kgrop$-modules
    \begin{gather*}
    \begin{split}
        U(\A^{L}\otimes_{\catLieC}K)
        &\cong(\A^{L}_0\otimes_{\catLie}\catLieC)\otimes_{\catLieC}K\\
        &\cong\A^{L}_0\otimes_{\catLie}U(K).      
    \end{split}
    \end{gather*}
\end{proof}

The following lemma compares the composites of right adjoints.

\begin{lemma}\label{homandU}
We have the following commutative diagram (up to isomorphisms)
      \begin{gather*}
        \xymatrix{
        \catLieCMod 
         \ar[d]_{U}
        &&& 
        \AMod\ar[d]^{U}\ar[lll]_-{\Hom_{\AMod}(\A^{L},-)}\\
        \catLieMod 
        &&& \kgropMod\ar[lll]^-{\Hom_{\kgropMod}(\A^{L}_0,-)} 
        \ar@{}[lllu]|{\circlearrowleft}.
        }
 \end{gather*}
\end{lemma}
\begin{proof}
    For an $\A$-module $M$, we have a natural isomorphism
    \begin{gather*}
        \Hom_{\kgropMod}(\A^{L}_0, U(M))\cong \Hom_{\AMod}(F(\A^{L}_0), M)
    \end{gather*}
    given by the adjunction between \eqref{UforA} and \eqref{FforA}.
    Therefore, by \eqref{ALFAL0}, we have 
    \begin{gather*}
        \Hom_{\kgropMod}(\A^{L}_0, U(M))\cong \Hom_{\AMod}(\A^{L}, M).
    \end{gather*}
    It is easy to see that this gives an isomorphism of $\catLie$-modules.
\end{proof}

The uniqueness of left adjoints gives the following lemma. 

\begin{lemma}\label{tensorandF}
We have the following commutative diagram (up to isomorphisms)
      \begin{gather*}
        \xymatrix{
        \catLieCMod 
        \ar[rrr]^-{\A^{L}\otimes_{\catLieC}-}
        &&& 
        \AMod\\
        \catLieMod \ar[rrr]_-{\A^{L}_0\otimes_{\catLie}-}\ar[u]^{F} 
        &&& \kgropMod\ar[u]_{F}
        \ar@{}[lllu]|{\circlearrowleft}.
        }
 \end{gather*}
\end{lemma}

\begin{proof}
    Since the functor $(\A^L\otimes_{\catLieC}-)\circ F$ is the left adjoint of $U\circ \Hom_{\AMod}(\A^{L},-)$ and the functor $F\circ (\A^{L}_0\otimes_{\catLie}-)$ is the left adjoint of $\Hom_{\kgropMod}(\A^{L}_0,-)\circ U$, the statement follows from Lemma \ref{homandU} and the uniqueness of left adjoints.
\end{proof}
% \begin{proof}
%     For a $\catLie$-module $J$, we have isomorphisms of $\A$-modules
%     \begin{gather*}
%     \begin{split}
%         \A^L\otimes_{\catLieC}F(J)
%         &=\A^L\otimes_{\catLieC}(\catLieC\otimes_{\catLie}J)\\
%         &\cong \A^L\otimes_{\catLie} J\\
%         &\cong (\A\otimes_{\kgrop}\A^{L}_0)\otimes_{\catLie}J\\
%         &= F(\A^{L}_0\otimes_{\catLie}J),
%     \end{split}
%     \end{gather*}
%     where the second isomorphism follows from \eqref{ALandA}.
% \end{proof}

Let $\AMod^{\omega}$ denote the full subcategory of $\AMod$ whose set of objects is 
\begin{gather*}
    \{M\in \Ob(\AMod)\mid U(M)\in \Ob(\kgropMod^{\omega})\}.
\end{gather*}

\begin{theorem}\label{adjunctionAomega}
    We have $\A^{L}\in \Ob(\AMod^{\omega})$, and
    the adjunction in Theorem \ref{adjunctionA} restricts to an adjunction 
    \begin{gather*}
    \A^{L}\otimes_{\catLieC} - : \adj{\catLieCMod}{\AMod^{\omega}}{}{} : \Hom_{\AMod^{\omega}}(\A^{L},-).
    \end{gather*}
\end{theorem}

\begin{proof} 
For a $\catLieC$-module $K$, by Lemma \ref{tensorandU} and Proposition \ref{adjunctionPowell}, we have
\begin{gather*}
    U(\A^{L}\otimes_{\catLieC}K)\cong \A^{L}_0\otimes_{\catLie}U(K) \in \Ob(\kgropMod^{\omega}).
\end{gather*}
Therefore, the functor $\A^{L}\otimes_{\catLieC} -$ restricts to a functor 
\begin{gather*}
    \A^{L}\otimes_{\catLieC} -: \catLieCMod\to \AMod^{\omega}.
\end{gather*}

It follows from $\A^{L}\cong \A^{L}\otimes_{\catLieC}\catLieC$
that we have $\A^{L}\in \Ob(\AMod^{\omega})$.
The statement follows since $\AMod^{\omega}$ is a full subcategory of $\AMod$.
\end{proof}

\begin{remark}\label{remarkKim}
Minkyu Kim has independently obtained an adjunction which is equivalent to the adjunction in Theorem \ref{adjunctionA}.
Moreover, he has proved that the restriction of the adjunction, which is equivalent to that in Theorem \ref{adjunctionAomega}, induces the category equivalence between $\catLieCMod$ and $\AMod^{\omega}$.
This was presented by him at the workshop ``Algebraic approaches to mapping class groups of surfaces'', which was held at the Graduate School of Mathematical Sciences, University of Tokyo \cite{Kimtalk}.
% In fact, the category equivalence can be verified by using Theorem \ref{adjunctionAomega}, Proposition \ref{adjunctionPowell} and Lemmas \ref{tensorandU} and \ref{homandU}.
\end{remark}

\begin{proposition}\label{FUomega}
The adjunction between \eqref{UforA} and \eqref{FforA} restricts to
 \begin{gather*}
    F : \adj{\kgropMod^{\omega}}{\AMod^{\omega}}{}{} : U.
    \end{gather*}
\end{proposition}

\begin{proof}
    By the definition of the category $\AMod^{\omega}$, the forgetful functor $U$ restricts to $U: \AMod^{\omega}\to \kgropMod^{\omega}$.
    
    We will check that the functor $F$ restricts to a functor $F: \kgropMod^{\omega}\to \AMod^{\omega}$.
    For $N\in \Ob(\kgropMod^{\omega})$, we have
    \begin{gather*}
        F(N)\cong F(\A^{L}_0\otimes_{\catLie} \Hom_{\kgropMod^{\omega}}(\A^{L}_0,N))
    \end{gather*}
    by Proposition \ref{adjunctionPowell}.
    By Lemma \ref{tensorandF}, we have 
    \begin{gather*}
         F(\A^{L}_0\otimes_{\catLie} \Hom_{\kgropMod^{\omega}}(\A^{L}_0,N))\cong  \A^{L}\otimes_{\catLieC}F(\Hom_{\kgropMod^{\omega}}(\A^{L}_0,N)),
    \end{gather*}
    which is an object of $\AMod^{\omega}$ by Theorem \ref{adjunctionAomega}.
    The statement follows since $\kgropMod^{\omega}$ (resp. $\AMod^{\omega}$) is a full subcategory of $\kgropMod$ (resp. $\AMod^{\omega}$).
\end{proof}

\begin{corollary}\label{projectiveanalytic}
    For each $m\ge 0$, $\A^{L}(L^{\otimes m},H^{\otimes -})$ is a projective object of $\AMod^{\omega}$.
\end{corollary}

\begin{proof}
    For each $m\ge 0$, the $\catLie$-module $\catLie(m,-)$ is projective by the Yoneda Lemma.
    Via the equivalence $\A^{L}_0\otimes_{\catLie}-$, $\catLie(m,-)$ is sent to the projective object $\A^{L}_0\otimes_{\catLie}\catLie(m,-)\cong \A^{L}_0(L^{\otimes m},H^{\otimes -})$ in $\kgropMod^{\omega}$.
    Since $F$ has an exact right adjoint $U$, $F$ preserves projectives, and thus
    $\A^{L}(L^{\otimes m},H^{\otimes -})\cong F(\A^{L}_0(L^{\otimes m},H^{\otimes -}))$ is a projective object of $\AMod^{\omega}$.
\end{proof}

\begin{remark}
    Since we have $\A^{L}(L^{\otimes m},H^{\otimes -})\cong \A^{L}\otimes_{\catLieC}\catLieC(m,-)$ and since $\catLieC(m,-)$ is a projective $\catLieC$-module, Corollary \ref{projectiveanalytic} can also be proved by using the fact that the functor $\A^{L}\otimes_{\catLieC}-$ yields an equivalence of categories (see Remark \ref{remarkKim}).
\end{remark}

\begin{remark}
The $\A$-module $\A(0,-)\cong \A^{L}(L^{\otimes 0},H^{\otimes -})$ is an object of $\AMod^{\omega}$, but $\A(m,-)$ is not for any $m\ge 1$.
By a polynomial $\A$-module, we mean an $\A$-module $M$ such that $U(M)$ corresponds to a polynomial $\gr^{\op}$-module. Then polynomial $\A$-modules are objects of $\AMod^{\omega}$.
Schur functors, which we recall in Section \ref{sectionSchurfunctor}, are examples of polynomial $\A$-modules.
Since the forgetful functor $U$ is exact, the stability properties of polynomial $\gr^{\op}$-modules under subobjects, quotients, and extensions (see \cite{Hartl--Pirashvili--Vespa}) hold for polynomial $\A$-modules.
\end{remark}

Since the category $\A$ is an $\N$-graded $\K$-linear category, we have a projection functor
\begin{gather*}
    \A\twoheadrightarrow \A_0\cong \kgrop,
\end{gather*}
which induces a $\K$-linear functor
\begin{gather}\label{TforA}
    T: \kgropMod\to \AMod,
\end{gather}
which is fully-faithful, exact and satisfies $U\circ T=\id_{\kgropMod}$. Note that any morphism of $\A$ of degree $\ge 1$ acts by zero on the image of $T$.
In a similar way, the projection to the degree $0$ part induces a $\K$-linear functor
\begin{gather*}
     T: \catLieMod\to \catLieCMod,
\end{gather*}
which has properties similar to those of \eqref{TforA} mentioned above.

\begin{proposition}
     We have the following commutative diagrams (up to isomorphisms)
  \begin{gather*}
        \xymatrix{
        \catLieCMod 
        \ar[rrr]^-{\A^{L}\otimes_{\catLieC}-}
        &&& 
        \AMod\\
        \catLieMod \ar[rrr]_-{\A^{L}_0\otimes_{\catLie}-}\ar[u]^{T} 
        &&& \kgropMod\ar[u]_{T}
        \ar@{}[lllu]|{\circlearrowleft}
        }
 \end{gather*}
 and
 \begin{gather*}
        \xymatrix{
        \catLieCMod 
        &&& 
        \AMod\ar[lll]_-{\Hom_{\AMod}(\A^{L},-)}\\
        \catLieMod \ar[u]^{T}
        &&& \kgropMod\ar[u]_{T} \ar[lll]^-{\Hom_{\kgropMod}(\A^{L}_0,-)} 
        \ar@{}[lllu]|{\circlearrowright}.
        }
 \end{gather*}
\end{proposition}

\begin{proof}
Let $J$ be a $\catLie$-module.
By \eqref{ALandAL0}, we have isomorphisms of $\kgrop$-modules
\begin{gather*}
    \A^{L}\otimes_{\catLieC}T(J)\cong \A^{L}_0\otimes_{\catLie}\catLieC\otimes_{\catLieC}T(J)\cong \A^{L}_0\otimes_{\catLie}J.
\end{gather*}
Since the action of any morphism of $\A$ of degree $\ge 1$ on $\A^{L}\otimes_{\catLieC}T(J)$ is trivial, the above isomorphism yields an isomorphism of $\A$-modules.

Let $N$ be an $\kgrop$-module. An element of $\Hom_{\AMod}(\A^{L}, T(N))(n)$ is a natural transformation $\alpha=\{\alpha_m: \A^{L}(L^{\otimes n},H^{\otimes
 m})\to T(N)(m)\}_m$ such that $\alpha_m$ maps all elements of $\A^{L}_{\ge 1}(L^{\otimes n},H^{\otimes m})$ to zero.
Therefore, we obtain a $\K$-linear isomorphism 
$$\Hom_{\AMod}(\A^{L}, T(N))(n) \cong \Hom_{\kgropMod}(\A^{L}_0,N)(n).$$
Since the action of any morphism of $\catLieC$ of degree $\ge 1$ on $\Hom_{\AMod}(\A^{L}, T(N))$ is trivial, the above $\K$-linear isomorphism yields an isomorphism of $\catLieC$-modules.
\end{proof}

\section{Symmetric monoidal structures}\label{sectionsymmonstr}

Here we study the monoidality of the adjoint functors that we have studied in Section \ref{sectionadjunction} with respect to the symmetric monoidal structures that are defined by using Day convolution.

\subsection{Day convolution}
Let $\catP$ be a $\K$-linear PROP.
The \emph{Day convolution} $M\boxtimes_{\catP} N$ of $\catP$-modules $M$ and $N$ is the $\catP$-module defined by
\begin{gather*}
    M\boxtimes_{\catP} N = \int^{m,n\in\catP}\catP(m+ n,-)\otimes M(m)\otimes N(n).
\end{gather*}
The Day convolution of $\catP$-modules gives a $\K$-linear symmetric monoidal structure on the category of $\catP$-modules, where the monoidal unit is $\catP(0,-)$.

\subsection{The adjunctions with respect to the symmetric monoidal structures}

Here we study the monoidality of the adjoint functors introduced in Section \ref{sectionadjunction} with respect to the $\K$-linear symmetric monoidal structures
\begin{gather*}
    (\kgropMod, \boxtimes_{\kgrop}, \kgrop(0,-)), \quad(\AMod, \boxtimes_{\A}, \A(0,-)),\\ (\catLieMod, \boxtimes_{\catLie}, \catLie(0,-)), \quad (\catLieCMod, \boxtimes_{\catLieC}, \catLieC(0,-))
\end{gather*}
induced by the Day convolution.

\begin{proposition}\label{symmetricmonoidalAL0}
The left adjoint functor
\begin{gather*}
    \A^{L}_0\otimes_{\catLie}-: \catLieMod\to\kgropMod
\end{gather*}
is symmetric monoidal.    
\end{proposition}

\begin{proof}
    For $\catLie$-modules $J$ and $J'$, define an isomorphism of $\kgrop$-modules
    \begin{gather*}
    \begin{split}
        \theta_{J,J'}:(\A^{L}_0\otimes_{\catLie}J)\boxtimes_{\kgrop}(\A^{L}_0\otimes_{\catLie}J')\to \A^{L}_0\otimes_{\catLie}(J\boxtimes_{\catLie}J')
    \end{split}
    \end{gather*}
    by the composition of the following isomorphisms
    \begin{gather*}
    \begin{split}
        &(\A^{L}_0\otimes_{\catLie}J)\boxtimes_{\kgrop}(\A^{L}_0\otimes_{\catLie}J')\\
        &=\int^{p,p'\in \kgrop}\kgrop(p+p', -)\otimes (\A^{L}_0\otimes_{\catLie}J)(p)\otimes  (\A^{L}_0\otimes_{\catLie}J')(p')\\
        &\cong \int^{m,m'\in \catLie} \left(\int^{p,p'\in \kgrop} \A^{L}_0(H^{\otimes p+p'}, H^{\otimes -})\otimes \A^{L}_0(L^{\otimes m},H^{\otimes p})\otimes \A^{L}_0(L^{\otimes m'},H^{\otimes p'})\right) \otimes J(m)\otimes J'(m')\\
        &\xrightarrow[\cong]{\int^{m,m'\in \catLie} (\circ\otimes J(m)\otimes J'(m'))} \int^{m,m'\in \catLie} \A^{L}_0(L^{\otimes m+m'}, H^{\otimes -})\otimes J(m)\otimes J'(m') \\
        &\cong  \int^{n,m,m'\in \catLie} \A^{L}_0(L^{\otimes n}, H^{\otimes -})\otimes \catLie(m+m',n)\otimes J(m)\otimes J'(m') \\
        &\cong \int^{n\in \catLie} \A^{L}_0(L^{\otimes n}, H^{\otimes -})\otimes \int^{m,m'\in \catLie} \catLie(m+m',n)\otimes J(m)\otimes J'(m')  \\
        &= \A^{L}_0\otimes_{\catLie}(J\boxtimes_{\catLie}J'),
    \end{split}
    \end{gather*}
    where the map $\circ$ is induced by the composition of the category $\A^{L}_0$. One can check that $\circ$ is an isomorphism by using Lemma \ref{AL0}.
    Define an isomorphism of $\kgrop$-modules
     \begin{gather*}
    \begin{split}
        \theta^{0}=\id: \kgrop(0,-)\xrightarrow{\cong} \A^{L}_0\otimes_{\catLie}\catLie(0,-).
    \end{split}
    \end{gather*}
    Then it is straightforward to check that $(\A^{L}_0\otimes_{\catLie}- , \theta_{J,J'}, \theta^0)$ satisfies the axioms of symmetric monoidal functors.
\end{proof}

The symmetric monoidal structure on $\kgropMod$ restricts to analytic functors.

\begin{corollary}
We have a symmetric monoidal category
\begin{gather*}
    (\kgropMod^{\omega}, \boxtimes_{\kgrop}, \kgrop(0,-)).
\end{gather*}
\end{corollary}

\begin{proof}
    It follows from Propositions \ref{adjunctionPowell} and \ref{symmetricmonoidalAL0} that for $N,N'\in \Ob(\kgropMod^{\omega})$, we have
    \begin{gather*}
    \begin{split}
        &N\boxtimes_{\kgrop} N'\\
        &\cong (\A^{L}_0\otimes_{\catLie}\Hom_{\kgrop}(\A^{L}_0, N))\boxtimes_{\kgrop} (\A^{L}_0\otimes_{\catLie}\Hom_{\kgrop}(\A^{L}_0, N'))\\
        &\cong \A^{L}_0\otimes_{\catLie}(\Hom_{\kgrop}(\A^{L}_0, N)\boxtimes_{\catLie}\Hom_{\kgrop}(\A^{L}_0, N')),
    \end{split}
    \end{gather*}
    which is an object of $\kgropMod^{\omega}$.
    Since $\kgrop(0,-)$ is analytic, the symmetric monoidal structure on $\kgropMod$ restricts to $\kgropMod^{\omega}$.
\end{proof}

\begin{remark}
\cite[Theorem 11.2]{Powellanalytic} states that the functor $ {}_{\Delta}\catAss\otimes_{\catLie} -: \catLieMod\to \kgropMod^{\omega}$ is symmetric monoidal with respect to the symmetric monoidal structure of $\kgropMod^{\omega}$ that is defined by pointwise tensor product and the corresponding  symmetric monoidal structure of $\catLieMod$.
Note that the Day convolution tensor product $\boxtimes_{\kgrop}$ for $\kgropMod$ is not equivalent to the pointwise tensor product.
\end{remark}

\begin{proposition}\label{symmetricmonoidalAL}
The left adjoint functor 
\begin{gather*}
    \A^{L}\otimes_{\catLieC}-: \catLieCMod\to\AMod
\end{gather*}
is symmetric monoidal.
\end{proposition}

\begin{proof}
    The proof is similar to that of Proposition \ref{symmetricmonoidalAL0}.
    In this case, we use the fact that the map induced by the composition of the category $\A^{L}$
    \begin{gather*}
        \int^{p,p'\in \A}\A^{L}(H^{\otimes p+p'}, H^{\otimes -})\otimes \A^{L}(L^{\otimes m},H^{\otimes p})\otimes \A^{L}(L^{\otimes m'},H^{\otimes p'})\to \A^{L}(L^{\otimes m+m'},H^{\otimes -})
    \end{gather*}
    is an isomorphism.
\end{proof}

\begin{remark}
By using the equivalence of categories between $\catLieCMod$ and $\AMod^{\omega}$ given by Kim (see Remark \ref{remarkKim}), 
one can also check that the symmetric monoidal structure on $\AMod$ restricts to $\AMod^{\omega}$.
\end{remark}

\begin{proposition}
The left adjoint functors 
$$F: \kgropMod\to \AMod,\quad  F: \catLieMod\to \catLieCMod$$
are symmetric monoidal.
\end{proposition}

\begin{proof}
For $\catLie$-modules $J,J'$, we have the following composition of isomorphisms 
\begin{gather*}
    \begin{split}
      &F(J)\boxtimes_{\catLieC}F(J')\\
      &= \int^{m,m'\in \catLieC}\catLieC(m+m',-)\otimes F(J)(m)\otimes F(J')(m')\\
      &\cong \int^{p,p'\in \catLie} \left(\int^{m,m'\in \catLieC}\catLieC(m+m',-)\otimes \catLieC(p,m)\otimes \catLieC(p',m') \right) \otimes J(p)\otimes J'(p')\\
      &\cong \int^{p,p'\in \catLie} \catLieC(p+p',-)\otimes J(p)\otimes J'(p')\\
      &\cong \int^{n\in \catLie} \catLieC(n,-)\otimes \left(\int^{p,p'\in \catLie}\catLie(p+p',n) \otimes J(p)\otimes J'(p')\right)\\
      &\cong F(J\boxtimes_{\catLie} J')
    \end{split}
\end{gather*}
and an isomorphism
\begin{gather*}
    \catLieC(0,-)\cong F(\catLie(0,-)),
\end{gather*}
which make $F$ a symmetric monoidal functor.

In a similar way, one can check that the functor $F:\kgrop\to \AMod$ is symmetric monoidal.
\end{proof}

\begin{proposition}
The forgetful functors 
\begin{gather*}
    U: \AMod\to \kgropMod,\quad 
    U: \catLieCMod\to \catLieMod
\end{gather*}
are symmetric lax monoidal.
\end{proposition}

\begin{proof}
    We have a canonical injective morphism of $\kgrop$-modules
    \begin{gather*}
        \nu_{M,M'}: U(M)\boxtimes_{\kgrop}U(M')\hookrightarrow  U(M\boxtimes_{\A}M')
    \end{gather*}
    for $M,M'\in \AMod$,
   and an inclusion map
   \begin{gather*}
      \nu^0:  \kgrop(0,-)\hookrightarrow \A(0,-),
   \end{gather*}
   such that $(U, \nu_{M,M'},\nu^0)$ is a symmetric lax monoidal functor.

   In a similar way, the functor $U: \catLieCMod\to \catLieMod$ is symmetric lax monoidal.
\end{proof}

\subsection{The algebra $C$ in $\catLieMod$ and the algebra $A$ in $\kgropMod$}

In a $\K$-linear monoidal category $(\catC,\otimes,I)$, the monoidal unit $I$ forms a canonical algebra $(I,\mu=\id:I\otimes I\xrightarrow{\cong} I,\eta=\id:I\to I)$.
Therefore, we have an algebra $\A(0,-)$ in the $\K$-linear symmetric monoidal category 
$(\AMod, \boxtimes_{\A}, \A(0,-))$,
and we have an algebra $\catLieC(0,-)$ in the $\K$-linear symmetric monoidal category $(\catLieCMod, \boxtimes_{\catLieC}, \catLieC(0,-))$.

Since the forgetful functor $U: \AMod\to\kgropMod$ is lax monoidal, the algebra $\A(0,-)$ in $(\AMod, \boxtimes_{\A}, \A(0,-))$ induces an algebra 
$$A=U(\A(0,-))$$
in $(\kgropMod,\boxtimes_{\kgrop},\kgrop(0,-))$.
Moreover, the algebra $A$ has the structure of a graded algebra, where
\begin{gather*}
A=\bigoplus_{d\ge 0}A_d, \quad A_d=\A_d(0,-),
\end{gather*}
and the multiplication consists of 
\begin{gather*}
    \mu:A_d\boxtimes_{\kgrop}A_{d'}\to A_{d+d'}
\end{gather*}
and the unit is
\begin{gather*}
    \eta: \kgrop(0,-)\xrightarrow{\cong} A_0\hookrightarrow \A(0,-).
\end{gather*}

In a similar way, the algebra 
$\catLieC(0,-)$ in $(\catLieCMod, \boxtimes_{\catLieC}, \catLieC(0,-))$ induces an algebra 
$$C=U(\catLieC(0,-))$$
in $(\catLieMod, \boxtimes_{\catLie}, \catLie(0,-))$.
The algebra $C$ is also graded, where
\begin{gather*}
    C=\bigoplus_{d\ge 0}C_d, \quad C_d=\catLieC(0,-)_d.
\end{gather*}
Recall that the $\K$-vector space $C_d(n)=\catLieC(0,n)_d$ for $n\ge 0$ is explicitly described in Lemma \ref{CatLieC0-}.

\begin{remark} 
Via the category equivalence in Proposition \ref{adjunctionPowell}, the graded algebra $A$ in $\kgropMod$ corresponds to the graded algebra $C$ in $\catLieMod$.
\end{remark}

In what follows, we will study the structure of the graded algebra $C$ in $\catLieMod$, and prove that the algebra $C$ is \emph{quadratic}.

Let $C_1^{\boxtimes}$ denote the tensor algebra in $(\catLieMod, \boxtimes_{\catLie}, \catLie(0,-))$ generated by $C_1$, that is, the graded algebra whose underlying $\catLie$-module is
\begin{gather*}
    C_1^{\boxtimes}= \bigoplus_{d\ge 0} C_1^{\boxtimes d},
\end{gather*}
whose multiplication
\begin{gather*}
    \mu: C_1^{\boxtimes}\boxtimes_{\catLie}C_1^{\boxtimes}\to C_1^{\boxtimes}
\end{gather*}
is defined by
\begin{gather*}
\begin{split}
   &\mu( h\otimes (f\otimes x_1\otimes\cdots\otimes x_{d_1})\otimes (g\otimes y_1\otimes \cdots \otimes y_{d_2}))\\
   &=(h(f\otimes g))\otimes x_1\otimes\cdots\otimes  x_{d_1}\otimes y_1\otimes \cdots\otimes y_{d_2}
\end{split}
\end{gather*}
for $h\in \catLie(p+q,r),\; f\in \catLie(m_1+\cdots+m_{d_1}, p),\; g\in \catLie(n_1+\cdots+n_{d_2}, q),\; x_1\in C_1(m_1),\;\cdots,\; x_{d_1}\in C_1(m_{d_1}),\; y_1\in C_1(n_1),\;\cdots,\; y_{d_2}\in C_1(n_{d_2})$,
and
whose unit
\begin{gather*}
    \eta: \catLie(0,-)\to C_1^{\boxtimes}
\end{gather*}
is defined by the inclusion
\begin{gather*}
    \eta :\catLie(0,-)\xrightarrow{\cong} C_1^{\boxtimes 0}\hookrightarrow C_1^{\boxtimes}.
\end{gather*}
Note that the $\catLie$-module $C_1=\catLieC(0,-)_1$ is concentrated in $2\in \Ob(\catLie)$ and $C_1(2)$ is the trivial representation of $\gpS_2$ generated by the Casimir element $c\in C_1(2)$. 

\begin{theorem}\label{presentationofC}
    We have a surjective morphism of graded algebras
    \begin{gather*}
       \phi: C_1^{\boxtimes}\to C
    \end{gather*}
    in $(\catLieMod, \boxtimes_{\catLie}, \catLie(0,-))$,
    and the kernel of $\phi$ is the two-sided ideal of $C_1^{\boxtimes}$ generated by the following two elements
    \begin{gather*}
      r_s:= \id_{4}\otimes  c\otimes c- P_{(13)(24)}\otimes c\otimes c\in C_1^{\boxtimes 2}(4), 
    \end{gather*}
    and 
    \begin{gather*}
      r_c:=  (([,] \otimes \id_2)P_{(23)})\otimes c\otimes c+ (\id_1\otimes [,]\otimes \id_1)\otimes c\otimes c\in C_1^{\boxtimes 2}(3).
    \end{gather*}
\end{theorem}

\begin{proof}
Define a morphism of algebras
\begin{gather*}
    \phi:C_1^{\boxtimes}\to C
\end{gather*}
by
\begin{gather*}
    \phi(f\otimes x_1\otimes\cdots\otimes x_d)=f(x_1\otimes \cdots\otimes x_d)
\end{gather*}
for $f\in \catLie(m_1+\cdots+m_d,p), x_1\in C_1(m_1),\cdots, x_d\in C_1(m_d)$.
Then $\phi$ is surjective since any element of $C_d(p)=\catLieC(0,p)_d$ is a linear combination of elements of the form $f c^{\otimes d}$ with $f\in \catLie(2d,p)$.

Let $R$ be the two-sided ideal of $C_1^{\boxtimes}$ generated by $r_s$ and $r_c$.
We have $\phi(r_s)=0$ and $\phi(r_c)=0$ by the second relation of \eqref{Casimirrelation}, which implies that $R\subset \ker \phi$.
Then we have a surjective morphism $\overline{\phi}$ of graded algebras
\begin{gather*}
   \overline{\phi}: C_1^{\boxtimes}/R\overset{\pr}{\twoheadrightarrow} C_1^{\boxtimes}/\ker \phi \overset{\phi}{\twoheadrightarrow} C.
\end{gather*}
In order to prove that $\ker\phi=R$, it suffices to construct a $\K$-linear map
$$\phi_d^{-}: C_d(n) \to (C_1^{\boxtimes}/R)_d(n)$$
such that $\phi_d^- \overline{\phi} =\id_{(C_1^{\boxtimes}/R)_d(n)}$ for each $d,n\ge 0$.

Define a $\K$-linear map $\widetilde{\phi}_d^{-}$ from the $\K$-vector space spanned by the set \eqref{generatorofcatLieC0-} to $(C_1^{\boxtimes}/R)_d(n)$ by 
\begin{gather*}
    \widetilde{\phi}_d^{-}((T_1\otimes \cdots\otimes T_n)P_{\sigma} c^{\otimes d})=\overline{(T_1\otimes \cdots\otimes T_n)P_{\sigma}\otimes c\otimes \cdots\otimes c}.
\end{gather*}
By Lemma \ref{CatLieC0-}, the $\K$-vector space $C_d(n)=\catLieC(0,n)_d$ is spanned by the set \eqref{generatorofcatLieC0-} with relations generated by the equivariance relation \eqref{equivariancerel}, the AS and IHX relations, \eqref{symmetryc}, \eqref{[,]cij} and \eqref{[,]ci}.
It is straightforward to check that the relations corresponding to these relations via $\widetilde{\phi}_d^{-}$ hold in $(C_1^{\boxtimes}/R)_d(n)$.
Therefore the $\K$-linear map $\widetilde{\phi}_d^{-}$ induces a $\K$-linear map $\phi_d^{-}:C_d(n) \to (C_1^{\boxtimes}/R)_d(n)$, which satisfies $\phi_d^- \overline{\phi} =\id_{(C_1^{\boxtimes}/R)_d(n)}$.
This completes the proof.
\end{proof}

It follows from the category equivalence in Proposition \ref{adjunctionPowell} that the algebra $A$ is quadratic, which can be stated in the following corollary.
Let $A_1^{\boxtimes}$ denote the tensor algebra in $(\kgropMod, \boxtimes_{\kgrop}, \kgrop(0,-))$ generated by $A_1=\A_1(0,-)$.
Note that the $\kgrop$-module $A_1$ is the simple $\kgrop$-module generated by the Casimir $2$-tensor $\tilde{c}$.

\begin{corollary}\label{quadraticA}
    We have a surjective morphism of graded algebras
    \begin{gather*}
        \psi: A_1^{\boxtimes}\to A
    \end{gather*}
    in $(\kgropMod, \boxtimes_{\kgrop}, \kgrop(0,-))$, and the kernel of $\psi$ is the two-sided ideal of $A_1^{\boxtimes}$ generated by the following two elements
    \begin{gather*}
        \tilde{r}_s:= \id_{4}\otimes  \tilde{c}\otimes \tilde{c}- P_{(13)(24)}\otimes \tilde{c}\otimes \tilde{c} \in A_1^{\boxtimes 2}(4)
    \end{gather*}
    and
    \begin{gather*}
        \tilde{r}_c:= (((\mu-\mu P_{(12)}) \otimes \id_2)P_{(23)})\otimes \tilde{c}\otimes \tilde{c}+ (\id_1\otimes (\mu-\mu P_{(12)})\otimes \id_1)\otimes \tilde{c}\otimes \tilde{c} \in A_1^{\boxtimes 2}(3),
    \end{gather*}
    where $\tilde{r}_c$ corresponds to the $4$T-relation.
\end{corollary}

\section{The $\A$-module $\A^L(L^{\otimes m},H^{\otimes -})$}\label{sectionAmod}

For each $m\ge 0$, we have the $\A$-module $\A^L(L^{\otimes m},H^{\otimes -})$.
For the first step, we study the $\A$-module $\A(0,-)=\A^L(L^{\otimes 0},H^{\otimes -})$. 
Then we study the $\A$-module $\A^{L}(L,H^{\otimes -})$.

\subsection{Filtration of $\A$-modules}\label{sectionfiltration}
Here we define a double filtration of $\A$-modules.

For each $d,k\ge 0$, we have the $(\A,\A)$-subbimodule $\A_{\ge d,\ge k}(-,-)$ of $\A(-,-)$ spanned by Jacobi diagrams of degree $\ge d$ with $\ge k$ trivalent vertices.
Then the $(\A,\A)$-bimodules $\A_{\ge d,\ge k}(-,-)$ form a double filtration of the $(\A,\A)$-bimodule $\A(-,-)$
\begin{gather*}
     \A_{\ge d,\ge k}(-,-)\supset \A_{\ge d+1,\ge k}(-,-),\quad 
    \A_{\ge d,\ge k}(-,-)\supset \A_{\ge d,\ge k+1}(-,-).
\end{gather*}
In the following, we will write $\A_{\ge d}(-,-)$ for $\A_{\ge d,\ge 0}(-,-)$.
The filtration $\A_{\ge d,\ge k}(-,-)$ induces a double filtration of $\A$-modules as follows.
Let $M$ be an $\A$-module.
We define a double filtration $M_{\ge d, \ge k}$ of $M$ by 
$$M_{\ge d,\ge k}=\A_{\ge d,\ge k}(-,-)\otimes_{\A}M\subset \A\otimes_{\A}M\cong M.$$
That is, for $n\ge 0$, the $\K$-vector space $M_{\ge d, \ge k}(n)$ is the subspace of $M(n)$ spanned by $\{M(f)(x)\mid f\in \A_{\ge d,\ge k}(m,n),\; x\in M(m),\; m\ge 0\}$.
Then we have 
\begin{gather*}
    M_{\ge d, \ge k}\supset M_{\ge d+1, \ge k},\quad M_{\ge d, \ge k}\supset M_{\ge d, \ge k+1}.
\end{gather*}
We have a quotient $\A$-module
$M_{\ge d,\ge k}/ M_{\ge d',\ge k'}$ for $d\le d'$ and $k\le k'$.

\begin{remark}
For an $\A$-module $M$, we have $M_{\ge 1}=0$ if and only if we have $M=TU(M)$, where $U$ is the forgetful functor defined in \eqref{UforA} and $T$ is the functor defined in \eqref{TforA}.
\end{remark}

We have the corresponding filtrations for $\catLieC$-modules.
For $d,k\ge 0$, we have the $(\catLieC,\catLieC)$-subbimodule $\catLieC(-,-)_{\ge d,\ge k}$ of $\catLieC(-,-)$ spanned by Jacobi diagrams of degree $\ge d$ with $\ge k$ trivalent vertices.
These induce a double filtration $K_{\ge d,\ge k}$ on a $\catLieC$-module $K$ by 
$$
K_{\ge d,\ge k}=\catLieC(-,-)_{\ge d,\ge k}\otimes_{\catLieC} K.
$$

The filtrations $\A^{L}\otimes_{\catLieC}(K_{\ge d,\ge k})$ and $(\A^{L}\otimes_{\catLieC}K)_{\ge d,\ge k}$ of the $\A$-module $\A^{L}\otimes_{\catLieC}K$ coincide, which can be checked as follows.
For an $(L^{\otimes m},H^{\otimes n})$-diagram $D$, let $d$ denote the degree of $D$ in $\A^{L}$, $k$ the number of trivalent vertices of $D$ and $p$ the number of univalent vertices of $D$ attached to $X_n$.
In terms of the generating morphisms of $\A^{L}$ as a $\K$-linear symmetric monoidal category, the degree $d$ is the number of copies of the morphism $c$, $k$ is the number of copies of the morphism $[,]$, $p$ is the number of copies of the morphism $i$.
Then we have
\begin{gather*}
    p=2d-k+m.
\end{gather*}
Let $\A^{L}_{\ge d,\ge k}$ denote the $(\A,\catLieC)$-subbimodule of $\A^{L}$ spanned by Jacobi diagrams of degree $\ge d$ with $\ge k$ trivalent vertices.
Then we have 
\begin{gather*}
\begin{split}
    \A^{L}\otimes_{\catLieC}(K_{\ge d,\ge k})
    &=  (\A^{L}\otimes_{\catLieC}\catLieC(-,-)_{\ge d,\ge k})\otimes_{\catLieC}K\\
    &\cong \A^{L}_{\ge d,\ge k}\otimes_{\catLieC}K\\
    &\cong (\A_{\ge d,\ge k}\otimes_{\A}\A^{L})\otimes_{\catLieC}K\\
    &= (\A^{L}\otimes_{\catLieC}K)_{\ge d,\ge k}.
\end{split}
\end{gather*}

\subsection{Schur functors}\label{sectionSchurfunctor}
Here, we recall Schur functors.
We refer the reader to \cite{Fulton--Harris} on representation theory of symmetric groups and general linear groups.

For a partition $\lambda$ of a non-negative integer, let $l(\lambda)$ denote the length of $\lambda$ and $|\lambda|$ the size of $\lambda$. We write $\lambda\vdash |\lambda|$.
Let $c_{\lambda}\in \K\gpS_{|\lambda|}$ denote the \emph{Young symmetrizer} corresponding to $\lambda$, that is, 
$$c_{\lambda}=\left(\sum_{\sigma \in R_{\lambda}}\sigma \right)\left(\sum_{\tau\in C_{\lambda}}\sgn(\tau)\tau\right),$$
where $R_{\lambda}$ is the row stabilizer and $C_{\lambda}$ is the column stabilizer of the canonical tableau of $\lambda$. (See \cite{Fulton--Harris} for details.) 
It is well known that $\{S_{\lambda}:=\K\gpS_{d}\cdot c_{\lambda}\mid \lambda\vdash d\}$ is the set of isomorphism classes of irreducible representations of $\gpS_{d}$.

\begin{remark}
Since we have $\catLie(m,n)=0$ for $m<n$, it is easy to see that any simple $\catLie$-module $J$ is concentrated in $d$ for some $d\ge 0$ and $J(d)$ is an irreducible representation of $\gpS_d$.
Therefore, the set of isomorphism classes of simple $\catLie$-modules is $\{S_{\lambda}\mid d\ge 0, \lambda\vdash d\}$.
\end{remark}

Let
$$S^{\lambda}:\KMod\to\KMod$$
denote the \emph{Schur functor} corresponding to $\lambda\vdash d$, which is defined by
$$S^{\lambda}(V)=V^{\otimes d}\otimes_{\K\gpS_{d}} S_{\lambda}.$$
Then $S^{\lambda}(V)$ is an irreducible representation of $\GL(V)$ if $\dim(V)\ge l(\lambda)$ and otherwise we have $S^{\lambda}(V)=0$.

Let 
$$\mathfrak{a}^{\#}=\Hom(-^{\ab},\K): \kgrop\to \KMod$$
denote the dual of the abelianization functor.

Regarding the $\gpS_d$-module $\K\gpS_d$ as a $\catLie$-module concentrated in $d$, we have an isomorphism of $\kgrop$-modules
\begin{gather*}
        \A^{L}_0\otimes_{\catLie}\K\gpS_{d}\cong (\mathfrak{a}^{\#})^{\otimes d}.
\end{gather*}
Therefore, for each partition $\lambda\vdash d$, we have an isomorphism of $\kgrop$-modules
\begin{gather*}
        \A^{L}_0\otimes_{\catLie}S_{\lambda}\cong S^{\lambda}\circ \mathfrak{a}^{\#}.
\end{gather*}
Since the functor $\A^{L}_0\otimes_{\catLie}-: \catLieMod\to \kgropMod^{\omega}$ is an equivalence of categories by Proposition \ref{adjunctionPowell}, the set of isomorphism classes of simple objects in $\kgropMod^{\omega}$ is $\{S^{\lambda}\circ \mathfrak{a}^{\#}\mid d\ge 0, \lambda\vdash d\}$.

By a \emph{graded} $\catLieC$-module, we mean a $\catLieC$-module $K$ which admits a grading $K=\bigoplus_{d\ge 0}K_d$ such that for any $f\in \catLieC(m,n)_d$, we have $K(f):K_{d'}(m)\to K_{d+d'}(n)$ for any $d'\ge 0$.
For example, $\catLieC(m,-)$ is a graded $\catLieC$-module for each $m\ge 0$.
A morphism $\alpha: K\to K'$ of graded $\catLieC$-modules is a morphism of $\catLieC$-modules which is compatible with some gradings of $K$ and $K'$.
Let $\mathrm{gr}\catLieCMod$ denote the category of graded $\catLieC$-modules and morphisms between them, which is a subcategory of $\catLieCMod$.

\begin{proposition}
    The set of isomorphism classes of simple objects in the category $\mathrm{gr}\catLieCMod$ is $\{T(S_{\lambda})\mid d\ge 0,\lambda\vdash d\}$. 
\end{proposition}

\begin{proof}
Let $K$ be a simple object of $\mathrm{gr}\catLieCMod$.
Since $K$ is non-trivial, there exists some some $d\ge 0$ such that $K_d\neq 0$.
We may take $d$ as the smallest one.
Since $K$ is simple, the $\catLieC$-submodule $K_{\ge d+1}=\bigoplus_{d'\ge d+1}K_{d'}$ should be $0$.
Therefore, we have $K=K_d$.
Since any morphism of $\catLieC$ of degree $\ge 1$ raises the grading of $K$, we have $K=TU(K)$, that is, $K$ is induced by a simple $\catLie$-module.
This completes the proof.
\end{proof}

\begin{remark}
The set of isomorphism classes of simple objects in $\catLieCMod$ is much bigger than that of $\mathrm{gr}\catLieCMod$.
For example, we have the following simple $\catLieC$-module.
Let 
\begin{gather*}
\begin{split}
    K(1)=S_1=\K1,\quad
    K(2)=S_{1^2}=\K a,
    \quad
    K(3)=S_{21}=\K\langle b,c\rangle,
\end{split}
\end{gather*}
where $a=1-(12), b=1+(12)-(13)-(132), c=1+(23)-(13)-(123)$,
and $K(m)=0$ for $m\neq 1,2,3$.
Then 
$$K(c\otimes \id_1)(1)=b, \quad K([,]\otimes \id_1)(b)=0,\quad K([,]\otimes \id_1)(c)=a,\quad K([,])(a)=1$$
defines a simple $\catLieC$-module.
\end{remark}

We define \emph{graded} $\A$-modules similarly.
Via the equivalence of categories between $\catLieCMod$ and $\AMod^{\omega}$ given by Kim (see Remark \ref{remarkKim}), the set of isomorphism classes of simple objects in the subcategory of $\AMod^{\omega}$ of graded $\A$-modules is $\{T(S^{\lambda}\circ \mathfrak{a}^{\#})\mid d\ge 0, \lambda \vdash d \}$.

\subsection{The $\kgrop$-module $A_d$}\label{sectionAd}

In \cite{Katada1, Katada2}, the author studied the $\kgrop$-module $A_d=\A_d(0,-)$.
As $\A$-modules, we have 
\begin{gather*}
    T(A_d)=\A_{\ge d}(0,-)/\A_{\ge d+1}(0,-).
\end{gather*}
One of the main results of \cite{Katada1, Katada2} is an indecomposable decomposition of the functor $A_d$, which is explained below.

For $d\ge 2$, let $A_d P$ and $A_d Q$ be the $\kgrop$-submodules of $A_d$ that are generated by the symmetric element $P_d$ and the anti-symmetric element $Q_d$ defined by
\begin{gather}\label{PdandQd}
    P_d=(\sum_{\sigma\in \gpS_{2d}}P_\sigma)\circ \tilde{c}^{\otimes d} ,\quad  Q_d=\tilde{c}^{\otimes d}-P_{(23)}\circ \tilde{c}^{\otimes d} \in \A_d(0,2d),
\end{gather}
respectively.
For $d=0,1$, we set $A_d P=A_d$.

\begin{theorem}[\cite{Katada1,Katada2}]\label{katadamainthm}
    Let $d\ge 2$.
    We have an indecomposable decomposition 
    $$A_d=A_d P\oplus A_d Q$$
    in $\kgropMod$.
    Moreover, $A_d P$ is a simple $\kgrop$-module isomorphic to $S^{2d}\circ\mathfrak{a}^{\#}$.
\end{theorem}

We have the following descending filtration of $\kgrop$-modules
\begin{gather*}
    A_d\supset A_d Q\supset A_{d,1}\supset A_{d,2}\supset \cdots\supset A_{d,2d-1}=0,
\end{gather*}
where $A_{d,k}=U(\A_{\ge d,\ge k}(0,-)/\A_{\ge d+1,\ge k}(0,-))$.

\begin{lemma}[\cite{Katada2}]\label{decompositionAdQ}
We have an isomorphism of $\kgrop$-modules
$$A_d Q/ A_{d,1}\cong \bigoplus_{\lambda\vdash d, \; \lambda\neq d} S^{2\lambda}\circ \mathfrak{a}^{\#},$$
where $2\lambda=(2\lambda_1,\cdots,2\lambda_{l(\lambda)})$ for a partition $\lambda=(\lambda_1,\cdots,\lambda_{l(\lambda)})$ of $d$. 
\end{lemma}

\subsection{The quotient $\A$-module $\A(0,-)/\A_{\ge d'}(0,-)$}

Here, we study the quotient $\A$-module $\A(0,-)/\A_{\ge d'}(0,-)$.
We prove that $\A(0,-)/\A_{\ge d'}(0,-)$ is indecomposable for $d'\ge 2$.

We have the following descending filtration of $\A(0,-)/\A_{\ge d'}(0,-)$ in $\AMod$:
\begin{gather*}
    \A(0,-)/\A_{\ge d'}(0,-)\supset \A_{\ge 1}(0,-)/\A_{\ge d'}(0,-)\supset\cdots \supset \A_{\ge d'}(0,-)/\A_{\ge d'}(0,-)= 0.
\end{gather*}

As we observed in Section \ref{sectionAd}, the $\A$-module
$\A_{\ge d}(0,-)/\A_{\ge d+1}(0,-)=T(A_d)$ factors through $\kgrop$.
However, in general, for $d\ge 0, d'\ge d+2$, the $\A$-module $\A_{\ge d}(0,-)/\A_{\ge d'}(0,-)$ does not factor through $\kgrop$, and we have the following.

\begin{theorem}\label{indecomposable}
For $d\ge 0,\; d'\ge d+2$,
the $\A$-module $\A_{\ge d}(0,-)/\A_{\ge d'}(0,-)$ is indecomposable. 
\end{theorem}

\begin{proof}
If $d=0, d'=2$, then it is easy to see that $\A(0,-)/\A_{\ge 2}(0,-)$ is a non-trivial extension of $\A(0,-)/\A_{\ge 1}(0,-)\cong \K$ by $\A_{\ge 1}(0,-)/\A_{\ge 2}(0,-)\cong T(S^{2}\circ \mathfrak{a}^{\#})$. Therefore, $\A(0,-)/\A_{\ge 2}(0,-)$ is indecomposable.

Otherwise, we have $d'\ge 3$.
Suppose that there exist $\A$-submodules $F$ and $G$ such that 
\begin{gather*}\label{supposedecompA}
    \A_{\ge d}(0,-)/\A_{\ge d'}(0,-)=F\oplus G \tag{$\ast$}.
\end{gather*} 
Then we have
\begin{gather*}
     \A_{\ge d'-1}(0,-)/\A_{\ge d'}(0,-)=F_{\ge d'-d-1,\ge 0}\oplus G_{\ge d'-d-1,\ge 0},
\end{gather*}
where $F_{\ge d'-d-1,\ge 0}$ (resp. $ G_{\ge d'-d-1,\ge 0}$) is the $\A$-submodule of $F$ (resp. $G$) defined in Section \ref{sectionfiltration}.
It follows from Theorem \ref{katadamainthm} and Lemma \ref{decompositionAdQ} that the $\A$-module $\A_{\ge d'-1}(0,-)/\A_{\ge d'}(0,-)=T(A_{d'-1})$ is decomposed into the direct sum of indecomposable $\A$-modules $T(A_{d'-1}P)$ and $T(A_{d'-1}Q)$, where $T(A_{d'-1}P)\cong T(S^{2(d'-1)}\circ\mathfrak{a}^{\#})$ and where any composition factor of $T(A_{d'-1}Q)$ is not isomorphic to $T(S^{2(d'-1)}\circ\mathfrak{a}^{\#})$.
Therefore, we may assume that $T(A_{d'-1}Q)\subset F_{\ge d'-d-1,\ge 0}\subset F$.

For any $n\ge 0$, we have a $\K$-linear isomorphism 
\begin{gather*}
    \A_{\ge d}(0,n)/\A_{\ge d'}(0,n)\cong A_d P(n)\oplus\cdots\oplus A_{d'-1}P(n)\oplus A_d Q(n)\oplus\cdots\oplus A_{d'-1}Q(n).
\end{gather*}
For a non-trivial element $z=x_d+\cdots+x_{d'-1}+y_d+\cdots+y_{d'-1}$, where $x_j\in A_j P(n)$, $y_j\in A_j Q(n)$, let
\begin{gather*}
t=
    \begin{cases}
     d' & \text{if }   x_d=\cdots=x_{d'-1}=0\\
     \min\{j\mid x_j\neq 0\} & \text{otherwise},
    \end{cases}
\end{gather*}
and 
\begin{gather*}
s=
    \begin{cases}
     d' & \text{if }   y_d=\cdots=y_{d'-1}=0\\
     \min\{j\mid y_j\neq 0\} & \text{otherwise}.
    \end{cases}
\end{gather*}
Suppose that $G(n)$ includes a non-trivial element $z$.
If $t>s$, then we have
\begin{gather*}
    G(\tilde{c}^{\otimes (d'-1-s)}\otimes \id_n)(z)=G(\tilde{c}^{\otimes (d'-1-s)}\otimes \id_n)(y_s),
\end{gather*}
which is a non-trivial element of $T(A_{d'-1}Q)(n+2(d'-1-s))$ and thus a non-trivial element of $F(n+2(d'-1-s))$.
If $t=s$, then there exists a morphism $f:n\to n$ in $\kgrop\cong \A_0$ such that 
\begin{gather*}
    G(f)(z)=x'_{t+1}+\cdots+x'_{d'-1}+y'_{t}+\cdots+y'_{d'-1},
\end{gather*}
where $x'_j\in A_j P(n)$ and $y'_j\in A_j Q(n)$, which means that we have $x'_{t}=0$, since $A_t P\cong S^{2t}\circ \mathfrak{a}^{\#}$ does not appear as a composition factor of $A_t Q$ by Lemma \ref{decompositionAdQ}.
Therefore, this case reduces to the case of $t>s$.
If $t<s$, then we have
\begin{gather*}
     G(\tilde{c}^{\otimes (d'-1-t)}\otimes \id_n)(z)=G(\tilde{c}^{\otimes (d'-1-t)}\otimes \id_n)(x_t)=\tilde{c}^{\otimes (d'-1-t)}\otimes x_t=u+v,
\end{gather*}
where $u$ is a non-trivial element of $A_{d'-1} P(n+2(d'-1-t))$ and $v$ is a non-trivial element of $A_{d'-1} Q(n+2(d'-1-t))$.
By an argument similar to the case of $t=s$, there exists a morphism $g\in \A_0(n+2(d'-1-t), n+2(d'-1-t))$ such that $G(g)(u+v)$ is a non-trivial element of $A_{d'-1} Q(n+2(d'-1-t))$ and thus a non-trivial element of $F(n+2(d'-1-t))$.
Therefore, in any cases, $G$ includes a non-trivial element of $F$, which contradicts \eqref{supposedecompA}.
Hence, we have $G=0$, which implies that $\A_{\ge d}(0,-)/\A_{\ge d'}(0,-)$ is indecomposable.
\end{proof}

In what follows, we study a submodule of $\A(0,-)$ which contains all $A_d Q$ with $d\ge 2$ as subquotient $\A$-modules.

Let $\A Q$ denote the $\A$-submodule of $\A_{\ge 2}(0,-)$ generated by the anti-symmetric element $Q_2\in \A_2(0,4),$
which is a generator of $A_2 Q$ as a $\kgrop$-module defined in \eqref{PdandQd}.
Then we have the following descending filtration of $\A Q$ in $\AMod$:
\begin{gather*}
    \A Q=\A Q_{\ge 2}\supset \A Q_{\ge 3}\supset\cdots \supset \A Q_{\ge d}\supset \cdots,
\end{gather*}
where 
$$\A Q_{\ge d}=\A Q\cap \A_{\ge d}(0,-).$$
The $\kgrop$-modules $A_d P$ and $A_d Q$ satisfy
\begin{gather*}
    T(A_d P)\cong \A_{\ge d}(0,-)/(\A_{\ge d+1}(0,-)+\A Q_{\ge d}), \quad  
    T(A_d Q)\cong \A Q_{\ge d}/\A Q_{\ge d+1}.
\end{gather*}
The cokernel $\A(0,-)/\A Q$ is a graded $\A$-module whose underlying $\kgrop$-module is
\begin{gather*}
    U(\A(0,-)/\A Q)\cong \bigoplus_{d\ge 0} A_d P\cong \bigoplus_{d\ge 0} S^{2d}\circ \mathfrak{a}^{\#}
\end{gather*}
and the action of the Casimir $2$-tensor is given by
\begin{gather*}
    (\tilde{c}\otimes \id_{2d})\cdot P_d=\frac{(2d)!}{(2d+2)!}P_{d+1},
\end{gather*}
where $P_d\in A_d P(2d)$ is the generator of $\A_d P$ defined in \eqref{PdandQd}.

We have the following descending filtration of the $\A$-module $\A Q/\A Q_{\ge d'}$
\begin{gather*}
    \A Q/\A Q_{\ge d'}\supset \A Q_{\ge 3}/\A Q_{\ge d'}\supset\cdots \supset \A Q_{\ge d'-1}/\A Q_{\ge d'}=T(A_{d'-1} Q).
\end{gather*}
The following proposition is a partial generalization of Theorem \ref{katadamainthm}.

\begin{proposition}\label{indecomposableQ}
For $d\ge 2,\; d'\ge d+1$, the $\A$-module $\A Q_{\ge d}/\A Q_{\ge d'}$ is indecomposable.
\end{proposition}

\begin{proof}
For $d'-d=1$, the $\A$-module $\A Q_{\ge d}/\A Q_{\ge d'}\cong T(A_d Q)$ is indecomposable by Theorem \ref{katadamainthm}.
For $d'-d=k\ge 2$, suppose that there exist $\A$-submodules $F$ and $G$ such that 
\begin{gather*}\label{supposedecompAQ}
    \A Q_{\ge d}/\A Q_{\ge d'}=F\oplus G \tag{$\ast$}.
\end{gather*} 
Then we have
\begin{gather*}
     \A Q_{\ge d'-1}/\A Q_{\ge d'}=F_{\ge d'-d-1,\ge 0}\oplus G_{\ge d'-d-1,\ge 0},
\end{gather*}
where $F_{\ge d'-d-1,\ge 0}$ (resp. $ G_{\ge d'-d-1,\ge 0}$) is the $\A$-submodule of $F$ (resp. $G$) defined in Section \ref{sectionfiltration}.
Since $\A Q_{\ge d'-1}/\A Q_{\ge d'}$ is indecomposable, we may assume that $\A Q_{\ge d'-1}/\A Q_{\ge d'}$ is an $\A$-submodule of $F$.
For any $n\ge 0$, we have a $\K$-linear isomorphism 
\begin{gather*}
    (\A Q_{\ge d}/\A Q_{\ge d'})(n)\cong A_d Q(n)\oplus A_{d+1} Q(n)\oplus \cdots\oplus A_{d'-1} Q(n).
\end{gather*}
For a non-trivial element $x=x_d+\cdots+x_{d'-1}$ of $(\A Q_{\ge d}/\A Q_{\ge d'})(n)$, where $x_j\in A_j Q(n)$, let $t=\min\{j\mid x_j\neq 0\}$.
If $G(n)$ includes a non-trivial element $x$, then we have
\begin{gather*}
    G(\tilde{c}^{\otimes (d'-1-t)}\otimes \id_n)(x)=G(\tilde{c}^{\otimes (d'-1-t)}\otimes \id_n)(x_t)=\tilde{c}^{\otimes (d'-1-t)}\otimes x_t,
\end{gather*}
which is a non-trivial element of $(\A Q_{\ge d'-1}/\A Q_{\ge d'})(n+2(d'-1-t))$ and thus a non-trivial element of $F(n+2(d'-1-t))$.
This contradicts \eqref{supposedecompAQ}.
Therefore, we have $G=0$, which implies that $\A Q_{\ge d'-1}/\A Q_{\ge d'}$ is indecomposable.
\end{proof}

\subsection{The $\A$-module $\A(0,-)$}

Here, we will study the $\A$-module $\A(0,-)$.

By Theorem \ref{adjunctionAomega}, $\A(0,-)$ is analytic.
By the Yoneda Lemma, $\A(0,-)$ is a projective $\A$-module and we have 
\begin{gather*}
    \End_{\AMod}(\A(0,-))\cong \End_{\A}(0)\cong \K.
\end{gather*}
It follows that idempotents of the $\A$-module $\A(0,-)$ are $0$ and $1$, which implies that the $\A$-module $\A(0,-)$ is indecomposable.
In what follows, we will give another proof of the indecomposability of $\A(0,-)$, which can also be applied to the $\A$-module $\A Q$.

\begin{theorem}\label{indecomposableA0}
    The $\A$-modules $\A(0,-)$ and $\A Q$ are indecomposable.
\end{theorem}

\begin{proof}
Suppose that we have non-trivial $\A$-submodules $F$ and $G$ of $\A(0,-)$ such that
$$
\A(0,-)=F\oplus G.
$$
Then for each $d'\ge 2$, the projection $\A(0,-)\twoheadrightarrow\A(0,-)/\A_{\ge d'}(0,-)$ induces
$$\A(0,-)/\A_{\ge d'}(0,-)\cong (F/(F\cap \A_{\ge d'}(0,-)))\oplus (G/(G\cap \A_{\ge d'}(0,-))).$$
Since $\A(0,-)/\A_{\ge d'}(0,-)$ is indecomposable by Theorem \ref{indecomposable}, we may assume that $G/(G\cap \A_{\ge d'}(0,-))=0$, which implies that $G\subset \A_{\ge d'}(0,-)$ and $F\not\subset \A_{\ge d'}(0,-)$.
Therefore, we have $G\subset \bigcap_{d'\ge 2}\A_{\ge d'}(0,-)=0$, which implies that $\A(0,-)$ is indecomposable.

By Proposition \ref{indecomposableQ}, the above argument holds for $\A Q$ and the projection $\A Q\twoheadrightarrow \A Q/\A Q_{\ge d'}$ for $d'\ge 3$.
\end{proof}

\subsection{The $\A$-module $\A(0,-)/\A_{\ge 4}(0,-)$}
In this subsection, we study the subquotient $\A$-modules of $\A(0,-)/\A_{\ge 4}(0,-)$.

First, we briefly recall from \cite{Katada1} the $\kgrop$-module structures of $A_d$ for $d\le 3$.
The $\kgrop$-module $A_0$ is spanned by the empty Jacobi diagram, and we have 
$$A_0\cong \K=S^0\circ \mathfrak{a}^{\#}.$$
The $\kgrop$-module $A_1$ is generated by $\tilde{c}$,
and we have 
$$A_1\cong S^{2}\circ \mathfrak{a}^{\#}.$$
The $\kgrop$-module $A_2$ is generated by $\tilde{c}^{\otimes 2}$ and we have the indecomposable decomposition
$$A_2=A_2 P\oplus A_2 Q.$$
Recall that $A_2 P$ is generated by the symmetric element $P_2$ and thus we have $A_2 P\cong S^{4}\circ \mathfrak{a}^{\#}$.
The $\kgrop$-module $A_2 Q$ is generated by the anti-symmetric element $Q_2$ and has a unique composition series
$$A_2 Q\supset A_{2,1}\supset A_{2,2}\supset 0$$
with composition factors $S^{2^2}\circ \mathfrak{a}^{\#}$, $S^{1^3}\circ \mathfrak{a}^{\#}$ and $S^{2}\circ \mathfrak{a}^{\#}$. 
Here, $A_{2,1}$ is generated by the element $a_{21}=\scalebox{0.8}{$\centre{%% Creator: Inkscape 1.4.2 (ebf0e940, 2025-05-08), www.inkscape.org
%% Accompanies image file 'a21.pdf' (pdf, eps, ps)
\begingroup%
  \makeatletter%
  \providecommand\color[2][]{%
    \errmessage{(Inkscape) Color is used for the text in Inkscape, but the package 'color.sty' is not loaded}%
    \renewcommand\color[2][]{}%
  }%
  \providecommand\transparent[1]{%
    \errmessage{(Inkscape) Transparency is used (non-zero) for the text in Inkscape, but the package 'transparent.sty' is not loaded}%
    \renewcommand\transparent[1]{}%
  }%
  \providecommand\rotatebox[2]{#2}%
  \newcommand*\fsize{\dimexpr\f@size pt\relax}%
  \newcommand*\lineheight[1]{\fontsize{\fsize}{#1\fsize}\selectfont}%
  \ifx\svgwidth\undefined%
    \setlength{\unitlength}{51.02362205bp}%
    \ifx\svgscale\undefined%
      \relax%
    \else%
      \setlength{\unitlength}{\unitlength * \real{\svgscale}}%
    \fi%
  \else%
    \setlength{\unitlength}{\svgwidth}%
  \fi%
  \global\let\svgwidth\undefined%
  \global\let\svgscale\undefined%
  \makeatother%
  \begin{picture}(1,0.55555556)%
    \lineheight{1}%
    \setlength\tabcolsep{0pt}%
    \put(0,0){\includegraphics[width=\unitlength,page=1]{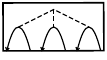}}%
  \end{picture}%
\endgroup%
}$}\in \A_2(0,3)$ and $A_{2, 2}$ is generated by the element $a_{22}=\scalebox{0.6}{$\centre{%% Creator: Inkscape 1.4.2 (ebf0e940, 2025-05-08), www.inkscape.org
%% Accompanies image file 'a22.pdf' (pdf, eps, ps)
\begingroup%
  \makeatletter%
  \providecommand\color[2][]{%
    \errmessage{(Inkscape) Color is used for the text in Inkscape, but the package 'color.sty' is not loaded}%
    \renewcommand\color[2][]{}%
  }%
  \providecommand\transparent[1]{%
    \errmessage{(Inkscape) Transparency is used (non-zero) for the text in Inkscape, but the package 'transparent.sty' is not loaded}%
    \renewcommand\transparent[1]{}%
  }%
  \providecommand\rotatebox[2]{#2}%
  \newcommand*\fsize{\dimexpr\f@size pt\relax}%
  \newcommand*\lineheight[1]{\fontsize{\fsize}{#1\fsize}\selectfont}%
  \ifx\svgwidth\undefined%
    \setlength{\unitlength}{65.19685039bp}%
    \ifx\svgscale\undefined%
      \relax%
    \else%
      \setlength{\unitlength}{\unitlength * \real{\svgscale}}%
    \fi%
  \else%
    \setlength{\unitlength}{\svgwidth}%
  \fi%
  \global\let\svgwidth\undefined%
  \global\let\svgscale\undefined%
  \makeatother%
  \begin{picture}(1,0.69565217)%
    \lineheight{1}%
    \setlength\tabcolsep{0pt}%
    \put(0,0){\includegraphics[width=\unitlength,page=1]{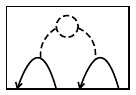}}%
  \end{picture}%
\endgroup%
}$}\in \A_2(0,2)$.
The set of submodules of $A_2$ consists of $(A_2 P)^{\oplus \epsilon}\oplus M$ for $M\subset A_2 Q$, $\epsilon\in \{0,1\}$, and the number of it is $8$ since composition factors of $A_2$ are multiplicity free.
See \cite[Theorem 7.9]{Katada1} for details.
The $\kgrop$-module $A_3$ is generated by 
$\tilde{c}^{\otimes 3}$ and we have the indecomposable decomposition
$$A_3=A_3 P\oplus A_3 Q.$$
Recall that $A_3 P$ is generated by the symmetric element $P_3$ and thus we have $A_3 P\cong S^{6}\circ \mathfrak{a}^{\#}$.
The $\kgrop$-module $A_3 Q$ is generated by the anti-symmetric element $Q_3$ and has the filtration
\begin{gather*}
    A_3 Q\supset A_{3, 1}\supset A_{3,2}\supset A_{3,3}\supset A_{3, 4}\supset 0,
\end{gather*}
where $A_{3,k}$ is generated by $a_{31}=a_{21}\otimes \tilde{c}\in \A_3(0,5)$ if $k=1$, 
by $a_{22}\otimes \tilde{c}\in \A_3(0,4)$ and $a_{32}=\scalebox{0.8}{$\centre{%% Creator: Inkscape 1.4.2 (ebf0e940, 2025-05-08), www.inkscape.org
%% Accompanies image file 'a32.pdf' (pdf, eps, ps)
\begingroup%
  \makeatletter%
  \providecommand\color[2][]{%
    \errmessage{(Inkscape) Color is used for the text in Inkscape, but the package 'color.sty' is not loaded}%
    \renewcommand\color[2][]{}%
  }%
  \providecommand\transparent[1]{%
    \errmessage{(Inkscape) Transparency is used (non-zero) for the text in Inkscape, but the package 'transparent.sty' is not loaded}%
    \renewcommand\transparent[1]{}%
  }%
  \providecommand\rotatebox[2]{#2}%
  \newcommand*\fsize{\dimexpr\f@size pt\relax}%
  \newcommand*\lineheight[1]{\fontsize{\fsize}{#1\fsize}\selectfont}%
  \ifx\svgwidth\undefined%
    \setlength{\unitlength}{70.86614173bp}%
    \ifx\svgscale\undefined%
      \relax%
    \else%
      \setlength{\unitlength}{\unitlength * \real{\svgscale}}%
    \fi%
  \else%
    \setlength{\unitlength}{\svgwidth}%
  \fi%
  \global\let\svgwidth\undefined%
  \global\let\svgscale\undefined%
  \makeatother%
  \begin{picture}(1,0.4)%
    \lineheight{1}%
    \setlength\tabcolsep{0pt}%
    \put(0,0){\includegraphics[width=\unitlength,page=1]{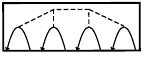}}%
  \end{picture}%
\endgroup%
}$}\in \A_3(0,4)$ if $k=2$, 
by $a_{33}=\scalebox{0.8}{$\centre{%% Creator: Inkscape 1.4.2 (ebf0e940, 2025-05-08), www.inkscape.org
%% Accompanies image file 'a33.pdf' (pdf, eps, ps)
\begingroup%
  \makeatletter%
  \providecommand\color[2][]{%
    \errmessage{(Inkscape) Color is used for the text in Inkscape, but the package 'color.sty' is not loaded}%
    \renewcommand\color[2][]{}%
  }%
  \providecommand\transparent[1]{%
    \errmessage{(Inkscape) Transparency is used (non-zero) for the text in Inkscape, but the package 'transparent.sty' is not loaded}%
    \renewcommand\transparent[1]{}%
  }%
  \providecommand\rotatebox[2]{#2}%
  \newcommand*\fsize{\dimexpr\f@size pt\relax}%
  \newcommand*\lineheight[1]{\fontsize{\fsize}{#1\fsize}\selectfont}%
  \ifx\svgwidth\undefined%
    \setlength{\unitlength}{51.02362205bp}%
    \ifx\svgscale\undefined%
      \relax%
    \else%
      \setlength{\unitlength}{\unitlength * \real{\svgscale}}%
    \fi%
  \else%
    \setlength{\unitlength}{\svgwidth}%
  \fi%
  \global\let\svgwidth\undefined%
  \global\let\svgscale\undefined%
  \makeatother%
  \begin{picture}(1,0.55555556)%
    \lineheight{1}%
    \setlength\tabcolsep{0pt}%
    \put(0,0){\includegraphics[width=\unitlength,page=1]{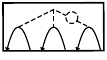}}%
  \end{picture}%
\endgroup%
}$}\in \A_3(0,3)$ if $k=3$ and 
by $a_{34}=\scalebox{0.8}{$\centre{%% Creator: Inkscape 1.4.2 (ebf0e940, 2025-05-08), www.inkscape.org
%% Accompanies image file 'a34.pdf' (pdf, eps, ps)
\begingroup%
  \makeatletter%
  \providecommand\color[2][]{%
    \errmessage{(Inkscape) Color is used for the text in Inkscape, but the package 'color.sty' is not loaded}%
    \renewcommand\color[2][]{}%
  }%
  \providecommand\transparent[1]{%
    \errmessage{(Inkscape) Transparency is used (non-zero) for the text in Inkscape, but the package 'transparent.sty' is not loaded}%
    \renewcommand\transparent[1]{}%
  }%
  \providecommand\rotatebox[2]{#2}%
  \newcommand*\fsize{\dimexpr\f@size pt\relax}%
  \newcommand*\lineheight[1]{\fontsize{\fsize}{#1\fsize}\selectfont}%
  \ifx\svgwidth\undefined%
    \setlength{\unitlength}{65.19685039bp}%
    \ifx\svgscale\undefined%
      \relax%
    \else%
      \setlength{\unitlength}{\unitlength * \real{\svgscale}}%
    \fi%
  \else%
    \setlength{\unitlength}{\svgwidth}%
  \fi%
  \global\let\svgwidth\undefined%
  \global\let\svgscale\undefined%
  \makeatother%
  \begin{picture}(1,0.56521739)%
    \lineheight{1}%
    \setlength\tabcolsep{0pt}%
    \put(0,0){\includegraphics[width=\unitlength,page=1]{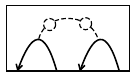}}%
  \end{picture}%
\endgroup%
}$}\in \A_3(0,2)$ if $k=4$.
As opposed to $A_2 Q$, the $\kgrop$-module $A_3 Q$ has $S^{2^2}\circ \mathfrak{a}^{\#}$ with multiplicity $2$ as composition factors, and thus has infinitely many $\kgrop$-submodules.
(See \cite[Section 8.4]{Katada2} for details.)

Now we study the $\A$-module structure of $\A(0,-)/\A_{\ge d}(0,-)$ for $d=2,3,4$.
For $d=2$, it is easy to see that the $\A$-module $\A(0,-)/\A_{\ge 2}(0,-)$ is a non-trivial extension of $T(S^2\circ \mathfrak{a}^{\#})$ by $T(S^0\circ \mathfrak{a}^{\#})$, or in other words, $\A(0,-)/\A_{\ge 2}(0,-)$ has a unique composition series
\begin{gather*}
    \A(0,-)/\A_{\ge 2}(0,-)\supset \A_{\ge 1}(0,-)/\A_{\ge 2}(0,-)\supset 0
\end{gather*}
with composition factors $T(S^0\circ \mathfrak{a}^{\#})$ and $T(S^2\circ \mathfrak{a}^{\#})$.

For $d=3$, since the $\A$-submodule of $\A(0,-)/\A_{\ge 3}(0,-)$ that is generated by the empty Jacobi diagram is $\A(0,-)/\A_{\ge 3}(0,-)$ itself, the set of $\A$-submodules of $\A(0,-)/\A_{\ge 3}(0,-)$ consists of $\A(0,-)/\A_{\ge 3}(0,-)$ and submodules of $\A_{\ge 1}(0,-)/\A_{\ge 3}(0,-)$.
The $\A$-module $\A_{\ge 1}(0,-)/\A_{\ge 3}(0,-)$ has the following composition series
\begin{gather*}
    \A_{\ge 1}(0,-)/\A_{\ge 3}(0,-)\supset \A_{\ge 2}(0,-)/\A_{\ge 3}(0,-) \supset \A Q/\A Q_{\ge 3}\supset \A_{21}\supset \A_{22}\supset 0
\end{gather*}
with composition factors
$T(S^2\circ \mathfrak{a}^{\#})$, $T(S^{4}\circ \mathfrak{a}^{\#})$, $T(S^{2^2}\circ \mathfrak{a}^{\#})$, $T(S^{1^3}\circ \mathfrak{a}^{\#})$ and $T(S^{2}\circ \mathfrak{a}^{\#})$, where $\A_{21}$ and $\A_{22}$ are the $\A$-submodules of $\A_{\ge 1}(0,-)/\A_{\ge 3}(0,-)$ generated by the classes represented by $a_{21}$ and $a_{22}$, respectively.
Any non-trivial $\A$-submodule of $\A_{\ge 1}(0,-)/\A_{\ge 3}(0,-)$ is a submodule of $\A_{\ge 2}(0,-)/\A_{\ge 3}(0,-)=T(A_2)$, and thus the number of $\A$-submodules of $\A_{\ge 1}(0,-)/\A_{\ge 3}(0,-)$ is $9$.

For $d=4$, any non-trivial $\A$-submodule of $\A(0,-)/\A_{\ge 4}(0,-)$ is a submodule of $\A_{\ge 1}(0,-)/\A_{\ge 4}(0,-)$, and any non-trivial $\A$-submodule of $\A_{\ge 1}(0,-)/\A_{\ge 4}(0,-)$ is a submodule of $\A_{\ge 2}(0,-)/\A_{\ge 4}(0,-)$.
In order to describe $\A$-submodules of $\A_{\ge 2}(0,-)/\A_{\ge 4}(0,-)$, we will study the $\A$-submodules of $\A Q/\A Q_{\ge 4}$.
For a submodule $M$ of $\A Q/\A Q_{\ge 4}$, let $\overline{M}=U(M/(M\cap \A Q_{\ge 3}))$. 
Then the $\kgrop$-module $\overline{M}$ is one of the $A_2 Q, A_{2,1},A_{2,2}$ and $0$, and the $\A$-module $M$ is classified as follows:
\begin{enumerate}[(i)]
\item if $\overline{M}=0$, then $M$ is a submodule of $T(A_3 Q)=\A Q_{\ge 3}/\A Q_{\ge 4}$,
\item if $\overline{M}=A_{2,2}$, then $M$ is one of the following nine submodules
\begin{itemize}
    \item the $\A$-submodule generated by $a_{22}$, 
    \item the $\A$-submodule generated by $a_{22}$ and $a_{32}$, which coincides with the $\A$-module $\A_{\ge 2,\ge 2}(0,-)/\A_{\ge 4,\ge 2}(0,-)$,
    \item the $\A$-submodule generated by $a_{22}$ and $c_{(31^2)}\cdot a_{31}$,
    \item the $\A$-submodule generated by $a_{22}$ and $c_{(21^3)}\cdot a_{31}$,
    \item the $\A$-submodule generated by $a_{22}$ and $a_{31}$,
    \item the $\A$-submodule generated by $a_{22}$ and $c_{(42)}\cdot Q_3$,
    \item the $\A$-submodule generated by $a_{22}$, $c_{(42)}\cdot Q_3$ and $a_{31}$,
    \item the $\A$-submodule generated by $a_{22}$ and $c_{(2^3)}\cdot Q_3$,
    \item the $\A$-submodule generated by $a_{22}$ and $Q_3$,
\end{itemize}
where $c_{\lambda}$ denotes the Young symmetrizer corresponding to a partition $\lambda$,
\item if $\overline{M}=A_{2, 1}$, then $M$ is one of the following four submodules
\begin{itemize}
    \item the $\A$-submodule generated by $a_{21}$, which coincides with the $\A$-module $\A_{\ge 2,\ge 1}(0,-)/\A_{\ge 4,\ge 1}(0,-)$,
    \item the $\A$-submodule generated by $a_{21}$ and $c_{(42)}\cdot Q_3$, 
    \item the $\A$-submodule generated by $a_{21}$ and $c_{(2^3)}\cdot Q_3$, 
    \item the $\A$-submodule generated by $a_{21}$ and $Q_3$,
\end{itemize} 
\item if $\overline{M}=A_2 Q$, then we have $M=\A Q/\A Q_{\ge 4}$.
\end{enumerate}
The set of $\A$-submodules of $\A_{\ge 2}(0,-)/\A_{\ge 4}(0,-)$ consists of 
\begin{itemize}
    \item $(T(A_3 P))^{\oplus \epsilon}\oplus M$ for $\epsilon\in \{0,1\}$ and an $\A$-submodule $M$ of $\A Q/\A Q_{\ge 4}$,
    \item the $\A$-submodule generated by $P_2$,
    \item the $\A$-submodule generated by $P_2$ and $a_{31}$,
    \item the $\A$-submodule generated by $P_2$ and $Q_3$,
    \item the $\A$-submodule generated by $P_2$ and $a_{22}$,
    \item the $\A$-submodule generated by $P_2$ and $a_{22}$ and $a_{31}$,
    \item the $\A$-submodule generated by $P_2$ and $Q_3$,
    \item the $\A$-submodule generated by $P_2$ and $a_{21}$,
    \item the $\A$-submodule generated by $P_2$ and $a_{21}$ and $Q_3$,
    \item the $\A$-submodule generated by $P_2$ and $Q_2$, that is, $\A_{\ge 2}(0,-)/\A_{\ge 4}(0,-)$.
\end{itemize}

\subsection{The $\A$-module $\A^L(L,H^{\otimes -})$}
Here, we study the $\A$-module structure of $\A^L(L,H^{\otimes -})$.
We write $A^L_d$ for the $\kgrop$-module $\A^L_{d}(L,H^{\otimes -})$ and $A^{L}_{d,\ge k}$ for $\A^L_{d,\ge k}(L,H^{\otimes -})$.
Then we have a filtration of $\kgrop$-modules
\begin{gather*}
    A^L_d\supset A^{L}_{d,\ge 1}\supset A^{L}_{d,\ge 2}\supset\cdots\supset A^{L}_{d,\ge 2d}\supset 0,
\end{gather*}
and the $\kgrop$-module structure of the graded quotient $A^L_d/A^{L}_{d,\ge 1}$ is 
\begin{gather}\label{decompositionofAdL0}
    A^L_d/A^{L}_{d,\ge 1}\cong (S^{1}\otimes (S^{d}\circ S^{2}))\circ \mathfrak{a}^{\#}\cong (S^{1}\otimes \bigoplus_{\lambda\vdash d}S^{2\lambda})\circ \mathfrak{a}^{\#}.
\end{gather}

Let $A^{L}_dP$ be the $\kgrop$-submodule of $A^L_d$ that is generated by the symmetric element
\begin{gather*}
    P^L_d=(\sum_{\sigma\in \gpS_{2d+1}}P_{\sigma})\circ (i \otimes \tilde{c}^{\otimes d}) \in \A^{L}_d(L,H^{\otimes 2d+1}).
\end{gather*}
Let $A^{L}_d Q$ be the $\kgrop$-submodule of $A^L_d$ that is generated by the two anti-symmetric elements
\begin{gather*}
   Q'_d=i \otimes \tilde{c}^{\otimes d}- P_{(12)}\circ (i \otimes \tilde{c}^{\otimes d}),\;
   Q''_d=i \otimes \tilde{c}^{\otimes d}- P_{(34)}\circ (i \otimes \tilde{c}^{\otimes d})  \in \A^{L}_d(L,H^{\otimes 2d+1}).
\end{gather*}

We have 
$$A^{L}_0=A^{L}_0 P\cong S^1\circ \mathfrak{a}^{\#}.$$
For $d\ge 1$, we obtain an analogue of a partial result of Theorem \ref{katadamainthm}.

\begin{proposition}\label{decompositionALd}
    Let $d\ge 1$. We have a direct sum decomposition 
    \begin{gather*}
        A^{L}_d=A^{L}_d P\oplus A^{L}_d Q
    \end{gather*}
    in $\kgropMod$ and $A^{L}_d P$ is a simple $\kgrop$-module isomorphic to $S^{2d+1}\circ \mathfrak{a}^{\#}$.
\end{proposition}

\begin{proof}
    The proof is analogous to that of \cite[Theorem 8.2]{Katada1}.
    Any element of $A^{L}_d(n)$ is a linear combination of 
    $f\circ (i\otimes \tilde{c}^{\otimes d})$ for $f\in \A_0(2d+1, n)$
    by Lemma \ref{factorizationofAL}.
    Define a $\kgrop$-module map 
    \begin{gather*}
        e: A^{L}_d\to A^{L}_d
    \end{gather*}
    by $e_n(f\circ (i\otimes \tilde{c}^{\otimes d}))=\frac{1}{(2d+1)!}f\circ P^L_d$ for $f\in \A_0(2d+1, n)$.
    This is well defined since the $4$T-relation is sent to $0$.
    Since $A^{L}_d P$ is generated by $P^L_d$, we have $\im e=A^{L}_d P$.
    Moreover, we have $e(A^{L}_dP)=A^{L}_dP$, which implies that $e$ is an idempotent (i.e., $e^2= e$).
    Therefore, we have 
    \begin{gather*}
        A^{L}_d=\im e\oplus \ker e, \quad \im e=A^{L}_d P,\quad \ker e=\im (1-e).
    \end{gather*}
    
    It follows from
    $$(\id_{2d+1}-P_{12})(\sum_{\sigma\in \gpS_{2d+1}}P_{\sigma})=0=(\id_{2d+1}-P_{34})(\sum_{\sigma\in \gpS_{2d+1}}P_{\sigma})$$
    that we have $A^{L}_d Q\subset \ker e$.
    For $f\in \A_0(2d+1, n)$, we have
    \begin{gather*}
        (1-e) (f\circ  (i\otimes \tilde{c}^{\otimes d}))=\frac{1}{(2d+1)!}\sum_{\sigma\in \gpS_{2d+1}}f\circ ((i\otimes \tilde{c}^{\otimes d})- P_{\sigma}(i\otimes \tilde{c}^{\otimes d})).
    \end{gather*}
    In order to prove that $\im (1-e)\subset A^{L}_d Q$, it suffices to prove that for any $\sigma\in \gpS_{2d+1}$, there exist $x,y\in \K\gpS_{2d+1}$ such that  
    $$(i\otimes \tilde{c}^{\otimes d})-P_{\sigma}(i\otimes \tilde{c}^{\otimes d})=x Q'_d+ y Q''_d.$$
    Since we have
    \begin{gather*}
        (i\otimes \tilde{c}^{\otimes d})-P_{\sigma\rho}(i\otimes \tilde{c}^{\otimes d})=(i\otimes \tilde{c}^{\otimes d})-P_{\sigma}(i\otimes \tilde{c}^{\otimes d})+P_{\sigma}((i\otimes \tilde{c}^{\otimes d})-P_{\rho}(i\otimes \tilde{c}^{\otimes d})),
    \end{gather*}
    by induction on the length of the permutation $\sigma$, 
    it suffices to prove the existence of $x,y$ if $\sigma$ is an adjacent transposition.
    If $\sigma=(12)$, then we set $x=\id_{2d+1}, y=0$. If $\sigma=(2j,2j+1)$ for $1\le j\le d$, then we set $x=y=0$.
    If $\sigma=(2j-1,2j)$ for $2\le j\le d$, then we set $x=0$,
    $$y=\begin{pmatrix}
        1&2&3&4&5&&\cdots&&2d+1\\
        1&2j-2&2j-1&2j&2j+1&\cdots&\widehat{2j-2}\cdots \widehat{2j+1} &\cdots&2d+1
    \end{pmatrix}.$$
    Therefore, we have the direct sum decomposition of $A^{L}_d$.

    It follows from the decomposition \eqref{decompositionofAdL0} of $A^L_d/A^{L}_{d,\ge 1}$ that the $\kgrop$-module $A^{L}_d P$ is isomorphic to $S^{2d+1}\circ \mathfrak{a}^{\#}$ since we have 
    $f\circ P^L_d=0$ for any $f\in \A^{L}_{0,\ge 1}(2d+1,-)$ and for any $f=(\id_{2d+1}-P_{\tau}), \;\tau \in \gpS_{2d+1}$.
    This completes the proof.
\end{proof}

By an argument similar to \cite[Theorem 7.9]{Katada1}, we obtain the following $\kgrop$-module structure of $A^{L}_1 Q$.

\begin{theorem}
The $\kgrop$-module $A^{L}_1 Q$ has a unique composition series
$$A^{L}_1 Q\supset A^{L}_{1,\ge 1}\supset A^{L}_{1,\ge 2}\supset 0$$
with composition factors $S^{21}\circ\mathfrak{a}^{\#}, S^{1^2}\circ\mathfrak{a}^{\#}$ and $S^{1}\circ \mathfrak{a}^{\#}$. 
In particular, $A^{L}_1 Q$ is indecomposable.
\end{theorem}

\begin{remark}
The $\kgrop$-module $A^{L}_d$ has a composition series of finite length whose set of composition factors consists of $S^{\lambda}\circ \mathfrak{a}^{\#}$ for each $\lambda$ that is obtained by deleting one box from the Young diagram of $\mu$ for each composition factor $S^{\mu}\circ \mathfrak{a}^{\#}$ of $A_{d+1}$.
\end{remark}

\begin{conjecture}\label{conjectureindecomposableALQ}
    For $d\ge 2$, the $\kgrop$-module $A^{L}_d Q$ is indecomposable.
\end{conjecture}

We have $\A$-submodules $\A^{L} Q', \A^{L} Q''$ and $\A^{L} Q$ of $\A^{L}(L,H^{\otimes -})$ that are generated by the anti-symmetric elements
\begin{gather*}
  Q'= i \otimes \tilde{c}- P_{(12)}\circ (i \otimes \tilde{c})\in \A^{L}_1(L,H^{\otimes 3}),\\
  Q''= i \otimes \tilde{c}^{\otimes 2}- P_{(34)}\circ (i \otimes \tilde{c}^{\otimes 2})\in \A^{L}_2(L,H^{\otimes 5}),\\
  Q^L=\{Q',Q''\},
\end{gather*}
respectively.
Then we also have the following descending filtration of $\A^{L} Q$ in $\AMod$:
\begin{gather*}
    \A^{L} Q=\A^{L} Q_{\ge 1}\supset \A^{L} Q_{\ge 2}\supset\cdots \supset \A^{L} Q_{\ge d}\supset \cdots,
\end{gather*}
where 
$$\A^{L} Q_{\ge d}=\A^{L} Q\cap \A^{L}_{\ge d}(L,H^{\otimes -}).$$

\begin{conjecture}\label{conjectureindecomposableAL}
    The $\A$-modules $\A^{L} Q', \A^{L} Q'', \A^{L} Q$ and $\A^{L}(L,H^{\otimes -})$ are indecomposable.
\end{conjecture}

\section{Perspectives}\label{sectionhandlebodygp}

Here, we consider modules over the handlebody groups induced by $\A$-modules.
Let $\mathcal{B}$ denote the category of bottom tangles in handlebodies.
Then $\mathcal{B}$ identifies with the opposite $\mathcal{H}^{\op}$ of the category $\mathcal{H}$ of isotopy classes
of embeddings of handlebodies relative to the bottom square.
The automorphism group $\Aut_{\mathcal{H}}(n)=\mathcal{H}_{n,1}$ is known as the handlebody group.
Let $\mathcal{B}_q$ denote the non-strictification of the category $\mathcal{B}$.
Let $\hat{\A}=\varprojlim_{d\ge 0}\A/{\A_{\ge d}}\cong \prod_{d\ge 0}\A_d$ denote the degree completion of the category $\A$.
Habiro and Massuyeau \cite{Habiro--Massuyeau} constructed a functor
$$Z: \mathcal{B}_q\to \hat{\A},$$
which is an extension of the Kontsevich integral for bottom tangles.

Let $M$ be an $\A$-module which factors through $\A/\A_{\ge d}$ for some $d\ge 1$.
Then we have a $\mathcal{B}_q$-module
\begin{gather*}
  \tilde{M}:  \mathcal{B}_q\xrightarrow{Z} \hat{\A}\xrightarrow{\pi_{d}} \A/\A_{\ge d}\xrightarrow{M} \KMod, 
\end{gather*}
where $\pi_{d}: \hat{\A}=\varprojlim_{d\ge 0}\A/{\A_{\ge d}} \to \A/\A_{\ge d}$ is the projection.
The functor $\tilde{M}$ induces an $\mathcal{H}_{n,1}$-module structure on the $\K$-vector space $M(n)$ for each $n\ge 0$.
For example, for each $m\ge 0$, the $\A$-module $\A^{L}(L^{\otimes m},H^{\otimes -})/\A^{L}_{\ge d}(L^{\otimes m},H^{\otimes -})$ induces an $\mathcal{H}_{n,1}$-module $\A^{L}(L^{\otimes m},H^{\otimes n})/\A^{L}_{\ge d}(L^{\otimes m},H^{\otimes n})$ for each $n\ge 0$.
Moreover, for $d\ge 2$, the $\mathcal{H}_{n,1}$-module restricts to a non-trivial module over the \emph{twist group}, which is the kernel of the canonical map from $\mathcal{H}_{n,1}$ to $\Aut(F_n)$.

Recently, the stable cohomology of $\mathcal{H}_{n,1}$ with some twisted coefficients has been studied \cite{Ishida--Sato, Randal-Williams--Wahl, Soulie}.
It would be interesting to study the stability of the cohomology of the handlebody group and the twist group with coefficients induced by $\A$-modules.

\bibliographystyle{plain}
\bibliography{reference}

\end{document}